\documentclass[leqno,11pt]{amsart}
\usepackage{amsmath,amssymb}
\usepackage{color}
\usepackage{comment}

\theoremstyle{plain}
\newtheorem{thm}{Theorem}[section]
\newtheorem{lemma}[thm]{Lemma}
\newtheorem{prop}[thm]{Proposition}
\newtheorem{corollary}[thm]{Corollary}
\theoremstyle{definition}

\newtheorem{remark}[thm]{Remark}
\newtheorem{remarks}[thm]{Remarks}

\newtheorem{example}[thm]{Example}

\def\Id{\operatorname{Id}}
\def\dist{\operatorname{dist}}
\def\med{\operatorname{med}}
\def\loc{\operatorname{loc}}

\def\la{\lambda}

\def\a{\alpha}

\def\hra{\to}
\def\bX{{X}}
\def\bY{{Y}}

\def\rn{\mathbb R\sp n}
\def\R{\mathbb R}
\def\Rn{\mathbb R\sp n}
\def\N{\mathbb N}

\def\lt{\left}
\def\rt{\right}

\def\ri{R_{I}}

\def\M{\mathcal M}

\def\Mpl{\mathcal M_+}

\def\muf{\mu_{\Phi}}

\def\mufn{\mu_{\Phi,n}}

\newtoks\by
\newtoks\paper
\newtoks\book
\newtoks\jour
\newtoks\yr
\newtoks\pages
\newtoks\vol
\newtoks\publ
\newtoks\eds
\newtoks\proc
\def\ota{{\hbox{???}}}
\def\cLear{\by=\ota\paper=\ota\book=\ota\jour=\ota\yr=\ota
\pages=\ota\vol=\ota\publ=\ota}
\def\endpaper{\the\by, \textit{\the\paper},
{\the\jour} \textbf{\the\vol} (\the\yr), \the\pages.\cLear}
\def\endbook{\the\by, \textit{\the\book}, \the\publ.\cLear}
\def\endprep{\the\by, \textit{\the\paper}, \the\jour.\cLear}
\def\endproc{\the\by, \textit{\the\paper}, \the\publ, \the\pages.\cLear}
\def\name#1#2{#1 #2}
\def\et{ and }

\numberwithin{equation}{section}

\begin{document}

\title[Higher-order  embeddings]
{Higher-order Sobolev  embeddings \\ and isoperimetric inequalities}

\date{\today}
\author {Andrea Cianchi, Lubo\v s Pick and Lenka Slav\'ikov\'a}
\address{Dipartimento di Matematica e Informatica \lq\lq U. Dini"\\
Universit\`a di Firenze\\
Piazza Ghiberti 27\\
50122 Firenze\\
Italy} \email{cianchi@unifi.it}

\address{Department of Mathematical Analysis\\
Faculty of Mathematics and Physics\\
Charles University\\
Sokolovsk\'a~83\\
186~75 Praha~8\\
Czech Republic} \email{pick@karlin.mff.cuni.cz}

\address{Department of Mathematical Analysis\\
Faculty of Mathematics and Physics\\
Charles University\\
Sokolovsk\'a~83\\
186~75 Praha~8\\
Czech Republic} \email{slavikova@karlin.mff.cuni.cz}


\thanks{This research was partly supported by the
the research project of MIUR (Italian Ministry of University) Prin
2008 ``Geometric aspects of partial differential equations and
related topics", by GNAMPA of the Italian INdAM (National Institute
of High Mathematics), and by the grants 201/08/0383 and P201/13/14743S of the Grant
Agency of the Czech Republic. The third named author was also partly supported by the grant SVV-2013-267316.}

\begin{abstract}
Optimal higher-order Sobolev type embeddings are shown to follow via
isoperimetric inequalities.  This establishes a higher-order
analogue of a well-known link between first-order Sobolev embeddings
and isoperimetric inequalities. Sobolev type inequalities of any
order, involving arbitrary rearrangement-invariant norms, on open
sets in $\rn$, possibly endowed with a measure density,
are
reduced to much simpler one-dimensional inequalities for suitable
integral operators depending on the isoperimetric function of the
relevant sets.
 As a consequence, the optimal target space
  in the relevant Sobolev embeddings can be determined both in standard and in
  non-standard  classes of function spaces and
underlying measure spaces. In particular, our results are applied to
any-order Sobolev embeddings in regular (John) domains of the
Euclidean space, in Maz'ya classes of (possibly irregular) Euclidean
domains described in terms of their isoperimetric function, and in
 families of product probability spaces, of which the Gauss space
is a classical instance.
\end{abstract}

\maketitle

\tableofcontents

\section{Introduction}\label{S:intro}
Sobolev inequalities  and isoperimetric inequalities  had
traditionally been investigated along independent lines of research,
which had led to the cornerstone results by Sobolev \cite{sob1,
sob2}, Gagliardo \cite{gagliardo} and Nirenberg \cite{nirenberg} on
the one hand, and by De Giorgi \cite{DeG} on the other hand, until
their intimate connection was discovered some half a century ago.
Such breakthrough goes back to the work of Maz'ya \cite{Ma1960,
Ma1961}, who proved that quite general Sobolev inequalities are
equivalent to either isoperimetric or isocapacitary inequalities.
Independently, Federer and Fleming \cite{FF} also  exploited De
Giorgi's isoperimetric inequality to exhibit the best constant in
the special case of the Sobolev inequality for functions whose
gradient is integrable with power one in $\rn$.
These advances paved the way to
an~extensive research, along diverse directions, on the interplay
between isoperimetric and Sobolev inequalities, and to a~number of
remarkable applications, such as the classics by  Moser
\cite{Moser}, Talenti \cite{Ta}, Aubin \cite{Aubin}, Br\'ezis and
Lieb \cite{BL}. The contributions to this field now constitute the
corpus of a vast literature, which includes the papers \cite{AFT,
BCR1, BWW, BH, BLbis, BK, BK1, CK, Cheeger, Ci_ind, Ci1, CFMP1,
CP_gauss, EKP, EFKNT, Gr, HK, HS1, HS2, KM, Kl, Ko, LPT, LYZ, M, Zh}
and the monographs \cite{BZ, CDPT, chavel, Heb, Mabook, Saloff}.
Needless to say, this list of references is by no means exhaustive.

 The strength of the approach  to Sobolev embeddings via isoperimetric
inequalities stems from the fact that not only it applies to a broad
range of situations, but also typically yields sharp results.  The
available results, however, essentially deal with first-order
Sobolev inequalities, apart from few exceptions on quite specific
issues concerning the higher-order case. Indeed, isoperimetric
inequalities are usually considered
 ineffectual in proving optimal higher-order Sobolev
embeddings. Customary techniques that are  crucial in the derivation
of first-order Sobolev inequalities from isoperimetric inequalities,
such as symmetrization, or just truncation, cannot be adapted to the
proof of higher-order Sobolev inequalities. A major drawback is that
these operations do not preserve higher-order (weak)
differentiability. A new  approach to the sharp
Sobolev  inequality in $\rn$, based on mass transportation
techniques, has been introduced in \cite{CNV}, and has later been
developed in various papers to attack other Sobolev type
inequalities, but still in the first-order
case. On the other hand,  methods which can be employed to handle
higher-order Sobolev inequalities, such as representation formulas,
Fourier transforms, atomic decomposition, are not flexible enough to
produce sharp conclusions in full generality. A paradigmatic
instance in this connection is provided by the standard Sobolev
embedding in $\rn$ to which we alluded above, whose original proof
via representation formulas \cite{sob1, sob2} does not include the
borderline case when derivatives are just integrable with power one.
This case was restored in \cite{gagliardo} and \cite{nirenberg}
through a completely different technique that rests upon
one-dimensional integration  combined with a clever use of
H\"older's inequality.

One main purpose of the present paper is to show that,
this notwithstanding,
 isoperimetric inequalities do imply optimal  higher-order Sobolev
embeddings in quite general frameworks. Sobolev embeddings for
functions defined on underlying domains in $\rn$, equipped with
fairly general measures, are included  in our discussion. Also,
Sobolev-type norms built upon any rearrangement-invariant Banach
function norm are considered. The use of isoperimetric inequalities
is shown to allow for a unified approach to the relevant embeddings,
which is based on the reduction to considerably simpler
one-dimensional inequalities. Such reduction principle is crucial in
a~characterization of the best possible target for arbitrary-order
Sobolev embeddings, in the class of all rearrangement-invariant
Banach function spaces. As a~consequence, the optimal target in
arbitrary-order Sobolev embeddings involving various customary and
non-standard underlying domains and norms can be exhibited. In fact,
establishing optimal higher-order Gaussian Sobolev embeddings,
namely Sobolev embeddings in $\rn$ endowed with the Gauss measure,
was our original motivation for the present research. Failure of
standard strategies in the solution of this problem   led us to
develop the general picture which is now the subject of this paper.
\par
A key step in our proofs amounts to the development of a~sharp
iteration method involving subsequent applications of optimal
Sobolev embeddings. We consider this method of independent interest
for its possible use in different problems, where regularity
properties of functions endowed with higher-order derivatives are in
question.

\section{An overview}\label{S:over}
We shall deal with Sobolev inequalities in an open connected set --
briefly, a domain -- $\Omega$ in $\rn$, $n \geq 1$, equipped with a
finite measure $\nu$ which is absolutely continuous with respect to
the Lebesgue measure, with density $\omega$. Namely,
 $d\nu(x)=\omega(x)\,dx$, where $\omega$ is a   Borel function  such that $\omega(x)>0$ a.e.~in $\Omega$.
Throughout the paper, we assume, for simplicity of notation, that
$\nu$ is normalized in such a way that $\nu (\Omega )=1$.
 The basic case when $\nu$ is the Lebesgue measure will be referred to
as Euclidean. Sobolev embeddings of arbitrary order for functions
defined in $\Omega$, with unconstrained values on $\partial
\Omega$, will be considered. However, the even simpler case of
functions vanishing (in the suitable sense) on $\partial \Omega$
together with their derivatives up to the order $m-1$ could be included
in our discussion.

   The isoperimetric inequality relative to     $(\Omega ,
\nu )$ tells us that
\begin{equation}\label{isop2}
P_{\nu }(E, \Omega ) \geq I_{\Omega , \nu} (\nu (E)),
\end{equation}
where $E$ is any measurable subset of $\Omega$, and $P_{\nu }(E, \Omega )$ stands for its perimeter in
$\Omega$ with respect to $\nu$. Moreover, $ I_{\Omega , \nu}$
denotes the largest non-decreasing function in $[0, \tfrac{1}{2}]$
for which \eqref{isop2} holds, called the isoperimetric function (or
isoperimetric profile) of $(\Omega , \nu)$, which  was introduced in
\cite{Ma1960}.

In the Euclidean case, $(\Omega , \nu)$ will be simply denoted by
$\Omega$, and $I_{\Omega , \nu}$ by $I_{\Omega }$. The isoperimetric
function $I_{\Omega , \nu}$ is known only in few special instances,
e.g.~when $\Omega$ is an Euclidean ball \cite{Mabook}, or agrees
with the space $\rn$ equipped with the Gauss measure \cite{Bor}.
 However, the asymptotic behavior of
$I_{\Omega , \nu}$ at $0$ -- the piece of information relevant in
our applications -- can be evaluated for various classes of
domains, including Euclidean bounded domains whose boundary is
locally a graph of a Lipschitz function \cite{Mabook}, or, more
generally, has a prescribed modulus of continuity
\cite{Cirelative, La}; Euclidean John domains, and even $s$-John
domains; the space $\rn$ equipped with the Gauss measure
\cite{Bor}, or with product probability measures which generalize
it \cite{BCR1, BCR}. The literature on isoperimetric inequalities
is very rich. Let us limit ourselves to mentioning that, besides
those quoted above, recent contributions on isoperimetric problems
in (domains in) $\rn$ endowed with a measure $\nu$ include
\cite{CMV, DHHT, FigalliMa, RCBM}.

Given a  Banach function space  $X(\Omega , \nu )$ of measurable
functions on $\Omega$, and  a positive integer $m \in \N$, the
$m$-th order Sobolev type space built upon $X(\Omega , \nu )$ is the
normed linear space $V^m X(\Omega , \nu )$ of all functions on $\Omega$
whose $m$-th order weak derivatives belong to $X(\Omega , \nu )$,
equipped with a natural norm induced by $X(\Omega , \nu )$.

A Sobolev  embedding amounts to the boundedness of the identity
operator from the Sobolev space $V^mX(\Omega , \nu )$ into another
function space $Y(\Omega , \nu )$ and will be denoted by
\begin{equation}\label{1001}
 V^mX(\Omega , \nu ) \to Y(\Omega , \nu ).
\end{equation}
When $m=1$, we
 refer to~\eqref{1001} as a~first-order embedding; otherwise, we call it
a~higher-order embedding.

Necessary and sufficient conditions for the validity of first-order
Euclidean  Sobolev embeddings with $X(\Omega ) = L^1(\Omega )$ and
$Y(\Omega ) = L^q(\Omega )$ for some $q \geq 1$ can be given through
the isoperimetric function $I_{\Omega }$. Sufficient conditions for
first-order Sobolev embeddings when $X(\Omega ) = L^p(\Omega)$ for
some $p>1$ and $Y(\Omega ) = L^q(\Omega)$, for some $q \geq 1$ can
also be provided in terms of $I_{\Omega }$. These results were
established in \cite{Ma1960, Ma1961}, and are exposed in detail in
\cite[Section 6.4.3]{Mabook}.

More recently, first-order Sobolev embeddings of the general form
\eqref{1001} (with $m=1$), where $X(\Omega , \nu )$ and $Y(\Omega ,
\nu )$ are Banach function spaces whose norm depends only on the
measure of level sets of functions, called rearrangement-invariant
spaces in the literature, have been shown to follow from
one-dimensional inequalities for suitable Hardy type operators which
depend on the isoperimetric function $I_{\Omega , \nu}$, and involve
the representation function norms $\| \cdot \|_{X(0,1)}$ and $\|
\cdot \|_{Y(0,1)}$ of $X(\Omega , \nu )$ and $Y(\Omega , \nu )$,
respectively.

Although a reverse implication need not hold in very pathological
settings (e.g.~in  Euclidean domains of Nikod\'ym type
\cite[Remark 6.5.2]{Mabook}),
 first-order Sobolev inequalities are known to be equivalent to
the  associated one-dimensional Hardy inequalities
 in most
 situations of interest in applications. This is
the case, for instance, in the basic case when $\Omega$ is a regular
Euclidean domain -- specifically,
 a John domain in $\rn $, $n \geq 2$ (see Section \ref{S:higher-euclid} for a definition). The class of John
 domains
 includes other more classical families of domains, such as  Lipschitz
domains, and
 domains with the cone property.
The  John domains  arise in connection with the study of
holomorphic dynamical systems and quasiconformal mappings.
 John domains are known to support a first-order
Sobolev inequality with the same exponents as in the standard
 Sobolev inequality \cite{Bo, HK, KM}. In fact,
 being a  John domain is a necessary condition for such a Sobolev
 inequality to hold in the class of two-dimensional simply
 connected open sets, and  in quite general  classes
 of  higher dimensional domains \cite{BK}.
The isoperimetric function $I_\Omega$ of any John domain is
known to  satisfy
\begin{equation}\label{Ijohn}
I_\Omega (s) \approx s^{\frac 1{n'}}
\end{equation}
near $0$, where $n'=\frac n{n-1}$.
Here, and in what follows, the
notation $\approx $ means that the two sides are bounded by each
other up to multiplicative constants independent of appropriate quantities. For instance, in \eqref{Ijohn} such  constants depend only on $\Omega$.

As a
consequence of \eqref{Ijohn}, one can show that the first-order
Sobolev embedding
\begin{equation}\label{E:*}
V^1X(\Omega ) \to Y(\Omega )
\end{equation}
 holds if and only if
the Hardy type inequality
\begin{equation}\label{E:**}
  \left\|\int_t\sp1f(s) s^{-1+ \frac 1n} \,ds \right\|_{ Y(0,1)}
  \leq C \left\|f\right\|_{X(0,1)}
\end{equation}
holds for some constant $C$, and for every nonnegative $f\in
X(0,1)$. Results of this kind, showing that Sobolev embeddings
follow from (and are possibly equivalent to) one-dimensional
inequalities will be referred to as reduction principles or reduction theorems. The
 equivalence of \eqref{E:*} and \eqref{E:**} is a key tool in determining the optimal target
$Y(\Omega )$ for $V^1X(\Omega )$ in \eqref{E:*} within families of
rearrangement-invariant function spaces, such as Lebesgue, Lorentz,
Orlicz spaces, provided that such an optimal target does exist
\cite{Ci_ind, Cianchiibero, EKP}. An even more standard version of
this reduction result, which holds for functions vanishing on
$\partial \Omega$, and is called P\'olya-Szeg\"o symmetrization
principle,
 is a crucial step in exhibiting the sharp
constant in the classical Sobolev inequalities to which we alluded
above \cite{Aubin, BL, Moser, Ta}.

A version of this picture for higher-order Sobolev inequalities is
exhibited in the present paper. We show that any $m$-th order
Sobolev embedding involving arbitrary rearrangement-invariant norms
can be reduced to a suitable one-dimensional inequalities for an
integral operator, with a kernel depending on $I_{\Omega , \nu}$ and
$m$.

Just to give an idea of the conclusions which follow from
our results,
let us mention  that, if, for instance, $\Omega$ is an Euclidean
John domain in $\rn$, $n \geq 2$, then a full higher-order
analogue of the equivalence of \eqref{E:*} and \eqref{E:**} holds.
Namely, the $m$-th order Sobolev embedding
\begin{equation*}
V^mX(\Omega ) \to Y(\Omega )
\end{equation*}
 holds if and only if
the Hardy type inequality
\begin{equation}\label{dec110}
  \left\|\int_t\sp1f(s) s^{-1+ \frac mn} \,ds \right\|_{ Y(0,1)}
  \leq C \left\|f\right\|_{X(0,1)}
\end{equation}
holds for some constant $C$, and for every nonnegative $f\in X(0,1)$
(Theorem \ref{T:eucl_reduction-john}, Section
\ref{S:higher-euclid}).

Our approach to   reduction principles for higher-order Sobolev
embedding relies  on the iteration of first-order results. Loosely
speaking, iteration is understood in the sense that, given a
rearrangement-invariant space and $m\in\N$, a first order optimal
Sobolev embedding is applied to show that the $(m-1)$-th order
derivatives of functions from the relevant Sobolev space belong to
a~suitable rearrangement-invariant~space. Another first-order
optimal Sobolev embedding is then applied to show that the
$(m-2)$-th order derivatives belong to another
rearrangement-invariant~space, and so on. Eventually, $m$ optimal
first-order Sobolev embeddings are exploited  to deduce that the
functions themselves belong to a certain space.

Let us warn that, although this strategy is quite natural in
principle, its implementation is not   straightforward. Indeed,
 even in the basic setting when $\Omega$ is
an Euclidean domain with a smooth boundary, and standard families
of norms are considered, iteration of optimal first-order
embeddings need not lead to optimal higher-order counterparts.

To see this, recall, for instance, that, if $\Omega$ is a regular
domain in $\mathbb R^2$, then
     \begin{equation}\label{ex1} V^2 L^1 (\Omega ) \to L^\infty (\Omega ).
     \end{equation}
       On the other hand, iterating twice the classical first-order Sobolev
       embedding only tells us that
             \begin{equation}\label{ex1'} V^2L^1(\Omega ) \to V^1L^2(\Omega ) \to
L^q(\Omega )
\end{equation}
for every $q<\infty$, and neither of the iterated embeddings can
be improved in the framework of Lebesgue space. This shows that subsequent applications of optimal first-order
Sobolev embeddings in the
    class of Lebesgue spaces  do not necessarily yield optimal higher-order
    counterparts.

    One might relate the loss of optimality in  the chain of
    embeddings
    \eqref{ex1'}
     to the lack of an optimal Lebesgue target space for
    the first-order Sobolev embedding of $V^1L^2(\Omega )$ when
    $n=2$. However,  non-optimal  targets
    may appear after iteration even in situations where optimal first-order target
    spaces do exist.  Consider, for example,  Euclidean Sobolev embeddings involving
Orlicz
    spaces. The optimal target in
    Sobolev embeddings of any order always exists  in this class of spaces,   and can be
    explicitly determined \cite{Ci_ind, Ci_higher}. In particular,
     Orlicz spaces naturally arise  in the borderline case of the Sobolev embedding
theorem.
 Indeed, if $\Omega$ is a regular domain in $\rn$ and $1 \leq m
   <n$, then
   \begin{equation}\label{limiting} V^m L^{\frac nm} (\Omega ) \to \exp L^{\frac
   n{n-m}}(\Omega )
   \end{equation}
  \cite{Yu, Po, Str}; see also \cite{Tr} for
  $m=1$. Here, $\exp L^\alpha (\Omega )$, with $\alpha >0$, denotes the Orlicz
  space associated with the Young function given by $e^{t^\alpha} -1$
  for $t \geq 0$. Observe that the target space in \eqref{limiting} is actually
  optimal
  in the class of all Orlicz spaces \cite{Ci_ind, Cianchiibero}.
Assume, for example, that $n \geq 3$ and $m
  =2$. Then \eqref{limiting} reduces to
   $$ V^2 L^{\frac n2} (\Omega ) \to \exp L^{\frac
   n{n-2}}(\Omega ).$$
   Via the iteration of optimal first-order embeddings, one gets
  \[
  V^2  L^{\frac n2} (\Omega ) \to V^1 L^{n}
(\Omega ) \to \exp L^{\frac
   n{n-1}}(\Omega ) \supsetneqq \exp L^{\frac
   n{n-2}}(\Omega ).
   \]
   Thus, subsequent applications of optimal Sobolev embeddings even in
   the
    class of Orlicz spaces, where optimal target spaces always exist, need not
result in optimal higher-order
    Sobolev  embeddings.

  The underlying idea behind the method that we shall introduce is that such a loss of
  optimality of the target space under iteration does not occur,
  provided that first-order (in fact, any-order)
  Sobolev embeddings whose targets are optimal among all rearrangement-invariant spaces are iterated. We thus proceed via
  a two-step argument, which can be outlined as follows.
Firstly, given any function norm $\| \cdot \|_{X(0,1)}$ and the
isoperimetric function $I_{\Omega , \nu}$ of $(\Omega , \nu)$, the
optimal  target  among all rearrangement-invariant function norms
for the first-order Sobolev space $V^1X(\Omega , \nu)$ is
characterized; secondly, first-order Sobolev embeddings with
an~optimal target are iterated to derive optimal targets in
arbitrary-order Sobolev embeddings.

In order to grasp this procedure in   a simple situation, observe
that, when applied in the proof of embedding \eqref{ex1}, it amounts
to strengthening the chain in \eqref{ex1'} by
\begin{equation}\label{ex1''}
V^2L^1(\Omega ) \to V^1L^{2,1}(\Omega ) \to L^\infty (\Omega ),
\end{equation}
where $L^{2,1}(\Omega )$ denotes a Lorentz space (strictly contained
in $L^2(\Omega )$). We refer to \cite{On, Pe, T2} for standard
Sobolev embeddings in Lorentz spaces. Note that both targets in the
embeddings in \eqref{ex1''} are actually optimal among all
rearrangement-invariant spaces.

As mentioned above, our reduction principle asserts that the Sobolev
embedding \eqref{1001}  follows from a suitable one-dimensional
inequality for an  integral operator depending on $I_{\Omega,\nu}$,
$m$, $\| \cdot \|_{X(0,1)}$ and $\| \cdot \|_{Y(0,1)}$.
Interestingly, in contrast with the first-order case, the relevant
integral operator is not just of Hardy type, but involves a genuine
kernel. The latter takes back the form of a
  basic (weighted) Hardy operator only
if, loosely speaking, the isoperimetric function $I_{\Omega,
\nu}(s)$ does not decay too fast to $0$ when $s$ tends to $0$. This
is the case, for instance, of \eqref{dec110}.  A major consequence
of the reduction principle is a characterization of a target space
$Y(\Omega , \nu)$ in embedding \eqref{1001}, depending on $X(\Omega,
\nu)$, $m$, and $I_{\Omega, \nu}$, which turns to be optimal among
all rearrangement-invariant spaces whenever Sobolev embeddings and
associated one-dimensional inequalities in the reduction principle
are actually equivalent. This latter property depends on the
geometry of $(\Omega , \nu)$, and is fulfilled in most customary
situations, to some of which a consistent part of this paper is
devoted.

Besides regular Euclidean domains, namely the
John domains  which we have already briefly discussed, the
implementations of our   results that will be presented concern
Maz'ya classes of (possibly irregular) Euclidean domains, and
product probability spaces, of which  the Gauss space and the Boltzmann spaces are
distinguished instances.

The Maz'ya classes are defined as families of domains whose isoperimetric
function is bounded from below by some fixed power.
 Sobolev embeddings in all
 domains from  a class of this type take the same form, and a worst, in
 a sense, domain from the relevant class can be singled out to demonstrate the sharpness of
 the
 results.
\par
 The product probability spaces in $\rn$ that are taken into
 account were analyzed in \cite{BCR1, BCR}, and   share common features with
 the Gauss space, namely $\rn$ endowed with the probability
 measure
$d \gamma _n (x) =  (2\pi)^{-\frac n2} e^{-\frac{|x|^2}{2}}dx$. In
particular, the Boltzmann spaces can be handled via our approach.

\section{Spaces of measurable functions}\label{S:measurable}
In this section, we
 briefly recall some basic facts from the theory of
rearrangement-invariant spaces. For more details, a~standard reference is~\cite{BS}.

Let $(\Omega,\nu)$ be as in Section \ref{S:over}. Recall that we are assuming $\nu (\Omega )=1$. The
measure of any measurable set $E \subset \Omega$ is thus given by
$$\nu (E) = \int _E  \omega (x)\, dx.$$
We denote by $\M(\Omega , \nu)$ the set of all Lebesgue measurable
(and hence $\nu$-measurable) functions on~$\Omega$ whose values
belong to $[-\infty , \infty]$. We also define $\M _+(\Omega ,
\nu)= \{u \in \M(\Omega , \nu): u \geq 0\}$, and $\M _0(\Omega ,
\nu)= \{u \in \M(\Omega , \nu): u\,\, \hbox{is finite a.e. in}\,\,
\Omega \}$.

The \textit{decreasing rearrangement} $u\sp* : [0,1] \to [0, \infty ]$ of a function $u \in \M(\Omega , \nu)$  is
defined as
$$
u\sp*(s)=\sup\{t\in\R:\,\nu\left(\{x\in \Omega
:\,|u(x)|>t\}\right)>s\}\ \qquad \qquad \textup{for}\
s\in[0, 1].
$$
The operation $u\mapsto u\sp*$ is monotone in the sense that
\[
|u|\leq |v|\ \textup{a.e.\ in}\ \Omega\quad \textup{implies}\quad
u\sp*\leq v\sp*\ \textup{in}\ (0,1).
\]
We also define $u^{**} : (0,1] \to [0, \infty]$ as
$$
u^{**}(s) = \frac{1}{s}\int _0^s u^*(r)\, dr \qquad \hbox{for
$s\in (0,1]$}.
$$
Note that $u^{**}$ is also non-increasing, and $u^{*}\leq u^{**}$ in $(0, 1]$. Moreover,
\begin{equation}\label{subadd}
\int_0\sp s(u +
v)^{*}(r)\,dr
\leq
\int_0\sp su^{*}(r)\,dr + \int_0\sp sv
^{*}(r)\,dr \quad \textup{for}\ s \in (0,1),
\end{equation}
for every $u,v\in\mathcal M_+(\Omega , \nu)$.

A basic property of rearrangements is the \textit{Hardy-Littlewood
inequality}, which tells us that, if $u, v \in\M(\Omega , \nu)$,
then
\begin{equation}\label{E:HL}
\int _\Omega |u(x) v(x)| d\nu (x) \leq \int _0^1 u^*(s) v^*(s) ds.
\end{equation}
A special case of~\eqref{E:HL} states that for every $u
\in\M(\Omega , \nu)$ and every measurable set $E \subset \Omega$,
\[
\int _E |u(x)| d\nu (x)\leq \int _0\sp {\nu (E)} u\sp*(s)\,ds.
\]
We say that a functional $\|\cdot\|_{X(0,1)}: \M _+ (0, 1) \to
[0,\infty]$  is a \textit{function norm}, if, for all $f$, $g$ and
$\{f_j\}_{j\in\N}$ in $\M_+(0,1)$, and every $\lambda \geq0$, the
following properties hold:
\begin{itemize}
\item[(P1)]\qquad $\|f\|_{X(0,1)}=0$ if and only if $f=0$;
$\|\lambda f\|_{X(0,1)}= \lambda\|f\|_{X(0,1)}$; \par\noindent
\qquad $\|f+g\|_{X(0,1)}\leq \|f\|_{X(0,1)}+ \|g\|_{X(0,1)}$;
\item[(P2)]\qquad $  f \le g$ a.e.\  implies $\|f\|_{X(0,1)} \le
\|g\|_{X(0,1)}$; \item[(P3)]\qquad $  f_j \nearrow f$ a.e.\ implies
$\|f_j\|_{X(0,1)} \nearrow \|f\|_{X(0,1)}$;
\item[(P4)]\qquad $\|1\|_{X(0,1)}<\infty$; \item[(P5)]\qquad  $\int_{0}\sp1 f(x)\,dx \le C
\|f\|_{X(0,1)}$ for some constant $C$ independent of $f$.
\end{itemize}

If, in addition,
\begin{itemize}
\item[(P6)]\qquad $\|f\|_{X(0,1)} = \|g\|_{X(0,1)}$ whenever
$f\sp* = g\sp *$,
\end{itemize}
we say that $\|\cdot\|_{X(0,1)}$ is a
\textit{rearrangement-invariant function norm}.

With any rearrangement-invariant function norm $\|\cdot\|_{X(0,1)}$,
it is associated another functional on $\M_+(0,1)$, denoted by
$\|\cdot\|_{X'(0,1)}$, and defined, for $g \in  \M_+(0,1)$, as
$$
\|g\|_{X'(0,1)}=\sup_{\overset{f\geq 0}{\|f\|_{X(0,1)}\leq
1}}\int_0\sp1 f(s)g(s)\,ds.
$$
It turns out that  $\|\cdot\|_{X'(0,1)}$ is also a
rearrangement-invariant function norm, which is called
the~\textit{associate function norm} of $\|\cdot\|_{X(0,1)}$.
Moreover, for every function norm $\|\cdot\|_{X(0,1)}$ and every function $f\in\Mpl(0,1)$, we have
\begin{equation}\label{E:pre-duality}
\|f\|_{X(0,1)}=\sup_{\overset{g\geq 0}{\|g\|_{X'(0,1)}\leq
1}}\int_0\sp1 f(s)g(s)\,ds.
\end{equation}
We
also introduce yet another functional on $\M_+(0,1)$, denoted by
$\|\cdot\|_{X'_d(0,1)}$, and defined, for $g \in  \M_+(0,1)$, as
$$
\|g\|_{X_d'(0,1)}=\sup_{\|f\|_{X(0,1)}\leq 1} \int_0^1 f^*(t)
g(t)\,dt.
$$
Clearly, one has that  $\|g\|_{X_d'(0,1)}\leq \|g\|_{X'(0,1)}$ for
every $g \in  \M_+(0,1)$, and $\|g\|_{X_d'(0,1)}= \|g\|_{X'(0,1)}$
if $g$ is non-increasing.

Given   a rearrangement-invariant function norm
 $\|\cdot\|_{X(0,1)}$, the space $X(\Omega , \nu)$ is
defined as the collection of all  functions  $u \in\M(\Omega ,
\nu)$ such that the expression
\[
\|u\|_{X(\Omega,\nu)}=\|u\sp* \|_{X(0,1)}
\]
is finite. Such expression defines a norm on $X(\Omega,\nu)$, and
the latter  is a Banach space  endowed with this norm, called a~rearrangement-invariant space. Moreover,
$X(\Omega,\nu) \subset \M_0(\Omega , \nu)$ for any rearrangement-invariant space $X(\Omega,\nu)$.
The space $X(0,1)$ is called the
\textit{representation space} of $X(\Omega , \nu)$.

 We also denote by $X_{\loc}(\Omega,\nu)$ the space
of all functions  $u \in\M(\Omega , \nu)$ such that $u\chi _G \in
X(\Omega , \nu)$ for every compact set $G \subset \Omega$. Here,
$\chi _G$ denotes the characteristic function of  $G$.

The rearrangement-invariant space $X'(\Omega , \nu)$ built upon the
function norm $\|\cdot \|_{X'(0,1)}$ is called the \textit{associate
space} of $X(\Omega , \nu)$. It
turns out that $X''(\Omega , \nu)=X(\Omega , \nu)$. Furthermore,
the \textit{H\"older inequality}
\[
\int_{\Omega}|u(x)v(x)|\,d\nu(x)\leq\|u\|_{X(\Omega ,
\nu)}\|v\|_{X'(\Omega , \nu)}
\]
holds for every $u \in X(\Omega , \nu)$ and $v \in X'(\Omega ,
\nu)$.

For any rearrangement-invariant spaces
$X(\Omega , \nu)$ and $Y(\Omega , \nu)$, we have that
\begin{equation}\label{E:equivemb}
 X(\Omega , \nu) \rightarrow Y(\Omega , \nu) \quad \hbox{if and only
 if} \quad
Y'(\Omega , \nu) \rightarrow X'(\Omega , \nu),
\end{equation}
with the same embedding norms \cite[Chapter~1,
Proposition~2.10]{BS}.

Given any $\la>0$, the \textit{dilation operator} $E_{\la}$, defined at
$f\in \M(0,1)$ by
$$
  (E_{\la}f)(s)=\begin{cases}
  f(\la\sp{-1}s)\quad&\textup{if}\ 0<s\leq \la\\
  0&\textup{if}\ \la<s<1,
  \end{cases}
$$
is bounded on any rearrangement-invariant~space $X(0,1)$, with
norm not exceeding $\max\{1, \frac 1{\la}\}$.

\textit{Hardy's Lemma} tells us that if $f_1,f_2 \in \mathcal
M_+(0,1)$ satisfy
$$
\int _0^s f_1(r) dr \leq \int _0^s f_2(r) dr \quad \textup{for
every}\ s\in(0,1),
$$
then
$$
\int _0^1 f_1(r) h(r) dr \leq \int _0^1 f_2(r)h(r) dr
$$
for every non-increasing function $h : (0,1) \rightarrow [0,
\infty ]$. A consequence of this result is the
\textit{Hardy--Littlewood--P\'olya principle} which asserts that
if the functions $u,v\in\M(\Omega , \nu)$ satisfy
$$
\int_0\sp su\sp*(r)\,dr\leq \int_0\sp sv\sp*(r)\,dr \qquad
\textup{for}\ s\in(0,1),
$$
then
$$
\|u\|_{X(\Omega , \nu)}\leq \|v\|_{X(\Omega , \nu)}
$$
for every rearrangement-invariant~space $X(\Omega , \nu)$.

Let $X(\Omega , \nu)$ and $Y(\Omega , \nu)$ be rearrangement
invariant\ spaces. By~\cite[Chapter~1, Theorem~1.8]{BS},
$$
\hbox{$X(\Omega , \nu) \subset Y(\Omega , \nu)$ \qquad if and only
if \qquad $X(\Omega , \nu) \to Y(\Omega , \nu)$}.
$$

For every rearrangement-invariant space $X(\Omega , \nu)$, one has
that
\begin{equation}\label{l1linf}
L^\infty (\Omega , \nu ) \to X(\Omega , \nu) \to L^1(\Omega , \nu
).
\end{equation}
An embedding of the form
$$X_{\loc}(\Omega , \nu) \to Y_{\loc}(\Omega , \mu),$$
where $\mu$ is a measure enjoying the same properties as $\nu$,
means that,  for every compact set $G\subset \Omega$, there exists
a constant $C$ such that
$$ \|u\chi _G \|_{Y(\Omega , \mu)} \leq C \|u\chi _G \|_{X(\Omega , \nu)},$$
for every $u \in X_{\loc}(\Omega , \nu)$.

Throughout, we use the convention that $\frac1{\infty}=0$, and $0\cdot
\infty =0$.

A basic example of a~function norm is the standard
\textit{Lebesgue norm} $\|\cdot\|_{L\sp p(0,1)}$,  for
$p\in[1,\infty]$, upon which the Lebesgue spaces $L^p(\Omega ,
\nu)$ are built.

The \textit{Lorentz spaces} yield an extension of the Lebesgue
spaces. Assume that $1\leq p,q\le\infty$. We define the functionals
$\|\cdot\|_{L\sp{p,q}(0,1)}$ and $\|\cdot\|_{L\sp{(p,q)}(0,1)}$  as
$$
\|f\|_{L\sp{p,q}(0,1)}=
\left\|s\sp{\frac{1}{p}-\frac{1}{q}}f^*(s)\right\|_{L\sp q(0,1)}
\quad \hbox{and} \quad \|f\|_{L\sp{(p,q)}(0,1)}=
\left\|s\sp{\frac{1}{p}-\frac{1}{q}}f^{**}(s)\right\|_{L\sp
q(0,1)},$$
respectively, for $f \in \M_+(0,1)$. One can show that
\begin{equation}\label{E:lorentz-identity}
L\sp{p,q}(\Omega , \nu)=L\sp{(p,q)}(\Omega , \nu) \quad\textup{if}\ 1<p\leq\infty\,,
\end{equation}
with equivalent norms.
If one of the  conditions
\begin{equation}\label{E:new-star}
\begin{cases}
1<p<\infty,\ 1\leq q\leq\infty,\\
p=q=1,\\
p=q=\infty,
\end{cases}
\end{equation}
is satisfied, then $\|\cdot\|_{L\sp{p,q}(0,1)}$ is equivalent to
a~rearrangement-invariant function norm. The corresponding
rearrangement-invariant space $L\sp{p,q}(\Omega,\nu)$ is called
a~\textit{Lorentz space}.

Let us recall that $L\sp{p,p}(\Omega,\nu)=L\sp p(\Omega,\nu)$ for every $p\in[1,\infty]$ and that $1\leq q\leq r\leq \infty$ implies
$L\sp{p,q}(\Omega,\nu)\hra L\sp {p,r}(\Omega,\nu)$ with equality
if and only if $q=r$.

Assume now that $1\leq p,q\le\infty$, and a third parameter $\alpha
\in\R$ is called into play. We define the functionals
$\|\cdot\|_{L\sp{p,q;\alpha}(0,1)}$ and
$\|\cdot\|_{L\sp{(p,q;\alpha)}(0,1)}$ as
$$
\|f\|_{L\sp{p,q;\alpha}(0,1)}=
\left\|s\sp{\frac{1}{p}-\frac{1}{q}}\log \sp
\alpha\lt(\tfrac{2}{s}\rt) f^*(s)\right\|_{L\sp q(0,1)} \quad
\hbox{and} \quad \|f\|_{L\sp{(p,q;\alpha)}(0,1)}=
\left\|s\sp{\frac{1}{p}-\frac{1}{q}}\log \sp
\alpha\lt(\tfrac{2}{s}\rt) f^{**}(s)\right\|_{L\sp q(0,1)},
$$
respectively, for $f \in \M_+(0,1)$. If one of the following
conditions
\begin{equation}\label{E:lz_bfs}
\begin{cases}
1<p<\infty,\ 1\leq q\leq\infty,\ \alpha\in\R;\\
p=1,\ q=1,\ \a\geq0;\\
p=\infty,\ q=\infty,\ \a\leq 0;\\
p=\infty,\ 1\leq q<\infty,\ \a+\frac1q<0,
\end{cases}
\end{equation}
is satisfied, then $\|\cdot\|_{L\sp{p,q;\alpha}(0,1)}$ is equivalent
to a rearrangement-invariant  function norm,   called a
\textit{Lorentz--Zygmund function norm}. The corresponding
rearrangement-invariant space $L\sp{p,q;\alpha}(\Omega,\nu)$ is a
\textit{Lorentz--Zygmund space}. At a few occasions,  we shall need
also the so-called \textit{generalized Lorentz--Zygmund space}
$L\sp{p,q;\alpha,\beta}(\Omega,\nu)$, where $p,q\in[1,\infty]$ and
$\alpha,\beta\in\R$. It is the space built upon the functional given
by
\[
\|f\|_{L\sp{p,q;\alpha,\beta}(0,1)}=
\left\|s\sp{\frac{1}{p}-\frac{1}{q}}\log \sp
\alpha\lt(\tfrac{2}{s}\rt) \log\sp{\beta}(1+\log
\lt(\tfrac{2}{s}\rt))f^*(s)\right\|_{L\sp q(0,1)}
\]
for $f\in \M_+(0,1)$. The values of $p,q,\alpha$ and $\beta$, for
which $\|\cdot\|_{L\sp{p,q;\alpha,\beta}(0,1)}$ is actually equivalent to
a~rearrangement-invariant  function norm, are characterized
in~\cite{EOP}. For more details on (generalized) Lorentz--Zygmund
spaces, see e.g.~\cite{BR, EOP, OP}.
Assume that
one of the conditions
in~\eqref{E:lz_bfs} is satisfied. Then the associate space
$(L\sp{p,q;\alpha})'(\Omega,\nu)$ of the Lorentz--Zygmund space
$L\sp{p,q;\alpha}(\Omega,\nu)$ satisfies (up to equivalent norms)
\begin{equation}\label{E:lz_assoc}
\big(L\sp{p,q;\alpha}\big)'(\Omega,\nu)=
\begin{cases}
L\sp{p',q';-\alpha}(\Omega,\nu)&\textup{if}\quad 1<p<\infty,\
1\leq q\leq\infty,\
\alpha\in\R;\\
L\sp{\infty,\infty;-\alpha}(\Omega,\nu)&\textup{if}\quad p=1,\
q=1,\ \a\geq0;\\
L\sp{1,1;-\alpha}(\Omega,\nu)&\textup{if}\quad p=\infty,\ q=\infty,\ \a\leq0;\\
L\sp{(1,q';-\alpha-1)}(\Omega,\nu)&\textup{if}\quad p=\infty,\
1\leq q<\infty,\
\a+\frac1q<0\\
\end{cases}
\end{equation}
\cite[Theorems~6.11 and~6.12]{OP}. Moreover,
\begin{equation}\label{E:glz-identity}
L\sp{(p,q;\alpha)}(\Omega , \nu)=
\begin{cases}
L\sp{p,q;\alpha}(\Omega , \nu) &\textup{if}\ 1<p\leq\infty;\\
L\sp{1,1;\alpha+1}(\Omega , \nu) &\textup{if}\ p=q=1,\ \alpha>-1,
\end{cases}
\end{equation}
and
\[
L\sp p(\Omega , \nu)\hra L\sp{(1,q)}(\Omega , \nu)\quad\textup{for
every}\ 1<p\leq\infty,\ 1\leq q\leq\infty
\]
\cite[Theorem~3.16~(i),(ii)]{OP}.

A generalization of the Lebesgue spaces in a different direction
is provided by the Orlicz spaces. Let $A: [0, \infty ) \to [0,
\infty ]$ be a Young function, namely a convex (non trivial),
left-continuous function vanishing at $0$. Any such function takes
the form
\begin{equation}\label{young}
A(t) = \int _0^t a(\tau ) d\tau \qquad \quad \hbox{for $t \geq
0$},
\end{equation}
for some non-decreasing, left-continuous function $a: [0, \infty )
\to [0, \infty ]$ which is neither identically equal to $0$, nor to
$\infty$.  The Orlicz space $L^A (\Omega , \nu )$ is the
rearrangement-invariant space associated with  the
 \emph{Luxemburg function norm} defined  as
\begin{equation*}
\|f\|_{L^A(0,1)}= \inf \left\{ \lambda >0 :  \int_0\sp 1A \left(
\frac{f(s)}{\lambda} \right) ds \leq 1 \right\}
\end{equation*}
for $f \in \M _+(0,1)$. In particular, $L^A (\Omega, \nu)= L^p
(\Omega, \nu )$ if $A(t)= t^p$ for some $p \in [1, \infty )$, and
$L^A (\Omega, \nu)= L^\infty (\Omega, \nu)$ if $A(t)= \infty
\chi_{(1, \infty)}(t)$.

A Young function $A$ is said to dominate another Young function
$B$ near infinity if positive constants $c$ and $t_0$ exist such
that
\[
B(t)\leq A(c t) \qquad \textrm{for \,\,\,} t\geq t_0\,.
\]
The functions $A$ and $B$ are called equivalent near infinity if
they dominate each other near infinity. One has that
\begin{equation}\label{B.6}
L^A(\Omega, \nu )\to L^B(\Omega, \nu) \quad \textup{if and only
if}\
 \hbox{$A$ dominates $B$ near infinity}\,.
\end{equation}

 We denote by $L^p\log^\alpha L(\Omega , \nu)$ the Orlicz space
associated with a Young function equivalent to $t^p (\log
t)^\alpha$ near infinity, where either $p>1$ and $\alpha \in \R$, or
$p=1$ and $ \alpha \geq 0$. The notation $\exp L^\beta (\Omega
,\nu)$ will be used for the Orlicz space built upon a Young
function equivalent to $e^{t^\beta}$ near infinity, where $\beta>0$. Also,
$\exp\exp L^\beta (\Omega ,\nu)$ stands for the Orlicz space
associated with a Young function equivalent to $e^{e^{t^\beta }}$
near infinity.

The classes of Orlicz and (generalized) Lorentz-Zygmund spaces overlap, up to equivalent norms. For
instance, if  $1\leq p<\infty$ and $\alpha\in\R$, then
$$
L\sp{p,p;\alpha}(\Omega,\nu)=L\sp p(\log L)\sp
{p\alpha}(\Omega,\nu).
$$
Moreover, if $\beta>0$, then
$$
L\sp{\infty,\infty;-\beta}(\Omega,\nu)=\exp L\sp
{\frac1{\beta}}(\Omega,\nu)
$$
and \cite[Lemma~2.2]{EOP}
$$
L\sp{\infty,\infty;0,-\beta}(\Omega,\nu) = \exp\exp
L\sp{\frac1{\beta}}(\Omega,\nu).
$$

A common extension of the Orlicz and Lorentz spaces is provided by a
family of Orlicz-Lorentz spaces defined as follows. Given $p\in (1,
\infty )$, $q\in [1, \infty )$ and  a Young function $D$ such that
$$\int ^\infty \frac{D(t)}{t^{1+p}}\,dt < \infty\,,$$
we denote by $L(p,q,D)(\Omega , \nu)$ the Orlicz-Lorentz space
associated with the rearrangement-invariant function norm defined,
for $f \in \mathcal M_+(0,1)$,  as
\[
\|f\|_{L(p,q,D)(0,1)} = \left\|s^{-\frac 1p} f^* (s^{\frac
1q})\right\|_{L^D(0, 1)}.
\]
The fact that  $\|\cdot \|_{L(p,q,D)(0,1)}$
is actually a function norm follows via easy modifications in the
proof of \cite[Proposition 2.1]{Cianchiibero}.  Observe that the
class of the spaces $L(p,q,D)(\Omega , \nu)$ actually includes (up
to equivalent norms) the Orlicz spaces and various instances of
Lorentz and Lorentz-Zygmund spaces.

\section{Spaces of Sobolev
type and the isoperimetric function}\label{S:sobolev}

Let $(\Omega , \nu)$ be as in Section \ref{S:over}. Define the
 perimeter of  a measurable set $E$ in $(\Omega, \nu)$
as
$$P_{\nu }(E, \Omega ) = \int _{\Omega \cap \partial ^M E  } \omega(x) d\mathcal H
^{n-1}(x) ,$$ where $\partial ^ME$ denotes the essential boundary of
$E$, in the sense of geometric measure theory \cite{Mabook, Z}. The
isoperimetric function $I_{\Omega , \nu} : [0, 1] \to [0, \infty ]$
of $(\Omega , \nu)$ is then given by
\[
I_{\Omega , \nu} (s) = \inf \left\{P_{\nu }(E, \Omega ) : E
\subset \Omega, s \leq \nu (E) \leq \tfrac 12 \right\} \quad
\textup{if}\ s \in [0,\tfrac 12],
\]
and $I_{\Omega , \nu} (s) = I_{\Omega , \nu} (1 -s)$ if $s \in
(\frac 12, 1]$.
 The  isoperimetric inequality \eqref{isop2} in $(\Omega , \nu )$
is a straightforward consequence of this definition and of the
fact that $P_{\nu }(E, \Omega ) = P_{\nu }(\Omega \setminus E,
\Omega )$ for every set $E \subset \Omega$.

Let us observe that, actually, $I_{\Omega , \nu} (s)
< \infty$ for $s \in [0, \tfrac 12)$. To verify this fact, fix any
$x_0 \in \Omega$, and let $R>0$ be such that $\nu (\Omega \cap
B_R(x_0)) = \tfrac 12$. Here, $B_R(x_0)$ denotes the ball, centered
at $x_0$, with radius $R$. By the  polar-coordinates formula for
integrals,
\begin{align}\label{finiteI}
\tfrac 12 = \int _{\Omega \cap B_R(x_0)} \omega (x)\, dx  = \int _0
^R \int _{\Omega \cap \partial B_\rho(x_0)} \omega (x) \, d\mathcal
H ^{n-1} (x)\, d\rho  = \int _0 ^R P_{\nu }(\Omega \cap
B_\rho(x_0),\Omega)\, d\rho\,,
\end{align}
whence $P_{\nu }(\Omega \cap B_\rho(x_0),\Omega) < \infty$
for a.e. $\rho \in (0, R)$. The finiteness of $I_{\Omega , \nu}$ in
$[0, \tfrac 12)$ now follows by its very definition.

 The next result shows that the best possible behavior of an
 isoperimetric function at $0$ is that given by \eqref{Ijohn}, in
 the sense that
  $I_{\Omega , \nu}(s)$
cannot decay more slowly than $s^{\frac 1{n'}}$ as  $s \to 0$,
whatever $(\Omega , \nu )$ is.

\begin{prop}\label{lower}
There exists a positive constant $C=C(\Omega, \nu)$ such that
\begin{equation}\label{relisop} I_{\Omega , \nu} (s)
\leq C s^{\frac 1{n'}}  \quad \hbox{near $0$}.
\end{equation}
\end{prop}

\begin{proof}
Let $x_0$ be any  Lebesgue point of $\omega$, namely a point such
that
\begin{equation}\label{lebesgue1}
\lim _{r \to 0^+} \frac{1}{|B_r(x_0)|} \int _{B_r(x_0)} \omega (x)\,
dx
\end{equation}
exists and is finite. Here, $|E|$ denotes the Lebesgue
measure of a set $E \subset \rn$. By \eqref{lebesgue1}, there exists
$r_0
>0$ and $C>0$ such that
\begin{equation}\label{lebesgue2}
\int _{B_r(x_0)} \omega (x)\, dx    \leq  C  r^{n} \quad \hbox{if
$0 < r < r_0$.}
\end{equation}
By an analogous chain as in \eqref{finiteI},
\begin{align}\label{lebesgue3}
\int _{B_r(x_0)} \omega (x)\, dx
= \int _0 ^r P_{\nu }(B_\rho(x_0),\Omega)\, d\rho \geq \tfrac r2
\inf\{P_{\nu }(B_\rho(x_0),\Omega): \tfrac r2 \leq \rho \leq r\}
\end{align}
if $0 < r < r_0$. \relax From \eqref{lebesgue2} and
\eqref{lebesgue3} we deduce that there exists a constant $C$ such
that
\begin{align*}
 C |B_r(x_0)|^{\frac 1{n'}}  & \geq \inf\{P_{\nu }(B_\rho(x_0),\Omega): \tfrac
r2 \leq \rho \leq r\} \\  & =  \inf\{P_{\nu
}(B_\rho(x_0),\Omega): \tfrac {1}{2^n}|B_r(x_0)| \leq |B_\rho(x_0)| \leq
|B_r(x_0)|\} \qquad \hbox{if $0 < r < r_0$.}
\end{align*}
Thus, there exists a constant $C$ such that
\[
C s^{\frac 1{n'}} \geq \inf\{P_{\nu }(E,\Omega): s \leq |E| \leq \tfrac
12\},
\]
provided that $s$ is sufficiently small, and hence
\eqref{relisop} follows.
\end{proof}

Let $ m \in \N$ and let $X(\Omega , \nu)$ be a  rearrangement-invariant  space.
We define the $m$-th order Sobolev  space $V^m X(\Omega,
\nu )$ as
\[
V^m X(\Omega , \nu) =   \big\{u: \hbox{$u$ is $m$-times weakly
differentiable in $\Omega$, and $|\nabla ^m u| \in X(\Omega,
\nu)$}\big\}.
\]
Here, $\nabla ^m u$ denotes the vector of all $m$-th order weak
derivatives of $u$. We shall also denote $\nabla\sp 0u=u$. Let us notice that in the definition of $V^m
X(\Omega , \nu)$ it is only required that the derivatives of the
highest order $m$ of $u$ belong to $X(\Omega , \nu)$. This
assumption does not entail, in general, that also $u$ and its
derivatives up to the order $m-1$ belong to $X(\Omega , \nu)$, and
even to $L^1(\Omega , \nu)$. Thus, it may happen that $V^m X(\Omega
, \nu) \nsubseteq V^k X(\Omega , \nu)$ for $m>k$. Such inclusion
indeed fails, for instance, when   $(\Omega , \nu) = (\rn , \gamma
_n)$, the
 Gauss space, and $\| \cdot \|_{X (0, 1)} = \| \cdot \|_{L^\infty (0,
 1)}$ (or $\| \cdot \|_{X (0, 1)} = \| \cdot \|_{{\rm exp} L^\beta (0,
 1)}$ for some $\beta >0$). Examples of Euclidean domains for which
$V^m X(\Omega) \nsubseteq L^1(\Omega)$  are those of Nykod\'ym type,
see, e.g., \cite[Sections 5.2 and 5.4]{Mabook}.

However, if $I_{\Omega , \nu}(s)$ does not decay at
$0$ faster than linearly, namely if there exists a positive constant
$C$ such that
\begin{equation}\label{lower-est} I_{\Omega , \nu} (s)
\geq C s   \quad \hbox{for $s \in [0 , \frac 12 ]$},
\end{equation}
then
 any function $u\in
V^m X(\Omega , \nu)$ does at least belong to $L^1(\Omega , \nu)$,
together with all its derivatives up to the order $m-1$. This is a
consequence of the next result. Such result in the case when $\nu$
is the Lebesgue measure is established in \cite[Theorem
5.2.3]{Mabook}; the general case rests upon  an analogous argument.
We provide a proof for completeness.

\begin{prop}\label{mazya}  {\bf [Condition for $V^1L\sp1(\Omega , \nu ) \subset L^1(\Omega , \nu )$]}
Assume that \eqref{lower-est} holds. Then
$V^1L\sp1(\Omega , \nu ) \subset L^1(\Omega , \nu )$, and
\begin{equation}\label{mazya2mean}
\tfrac C2 \bigg\|u - \int _\Omega u\, d\nu  \bigg\|_{L^1(\Omega ,
\nu)} \leq \|\nabla u\|_{L^1(\Omega , \nu )}
\end{equation}
for every $u \in V^1L^1(\Omega , \nu )$, where $C$ is the same
constant as in \eqref{lower-est}.
\end{prop}

\begin{proof}
Let  $\med (u)$ denote the median of a function $u\in
\M(\Omega,\nu)$, given by
$$\med  (u) = \sup\{t\in \R :\,\nu (\{x\in \Omega
:\,u(x)>t\})>\tfrac 12\}.$$ We begin by showing that
\begin{equation}\label{mazya2}
C \|u - \med  (u)\|_{L^1(\Omega , \nu )} \leq   \|\nabla
u\|_{L^1(\Omega , \nu )}
\end{equation}
for every $u \in V^1L^1(\Omega , \nu )$. On replacing, if necessary,
$u$ by $u - \med  (u)$, we may assume, without loss of generality,
that $\med  (u)=0$. Let us set $u_+ = \frac 12(|u|+u)$ and $u_- =
\frac 12(|u|-u)$, the positive and the negative parts of $u$,
respectively. Thus,
\begin{equation}\label{mazya3}
\nu (\{u_{\pm} >t \}) \leq \tfrac12 \qquad \hbox{for $t >0
$.}
\end{equation}
By~\eqref{isop2} and~\eqref{lower-est},
\[
P_{\nu }(\{u_{\pm} >t \},\Omega)\geq I_{\nu,\Omega }(\nu(\{u_{\pm}
>t \}))\geq C\nu(\{u_{\pm}>t \}).
\]
Therefore, owing to~\eqref{mazya3}, and to the
coarea formula, we have that
\begin{align*}
C \|u_{\pm }\|_{L^1(\Omega , \nu )} & = C \int _0^\infty
\nu(\{u_{\pm}>t \})  \, dt \leq   \int _0^\infty P_{\nu }(\{u_{\pm}
>t \}, \Omega )\, dt \\ & =
  \int _0^\infty \int _{ \partial ^M \{u_{\pm} >t \}\cap \Omega  } \omega (x) d\mathcal H
^{n-1}(x)
 \, dt =
   \int _\Omega |\nabla u_{\pm}| d\nu .
\end{align*}
Hence, \eqref{mazya2} follows. In particular, \eqref{mazya2} tells
us that $V^1L\sp1(\Omega , \nu ) \subset L^1(\Omega , \nu )$.
Inequality \eqref{mazya2mean} is a consequence of  \eqref{mazya2}
and of the fact that
$$
\bigg\|u - \int _\Omega u\, d\nu \bigg\|_{L^1(\Omega , \nu)}
\leq 2 \|u - \med  (u)\|_{L^1(\Omega , \nu )}
$$
for every $u \in
L^1(\Omega , \nu )$.
\end{proof}

\begin{corollary}\label{C:mazya-m}
Assume that~\eqref{lower-est} holds. Let $m\geq 1$. Let $X(\Omega,\nu)$ be any rearrangement-invariant space.
Then $V\sp m X(\Omega,\nu)\subset V\sp k L\sp1(\Omega,\nu)$ for every $k=0,\dots, m-1$.
\end{corollary}

\begin{proof}
By property (P5) of rearrangement-invariant spaces,  $V\sp m
X(\Omega,\nu)\to V\sp m L^1(\Omega,\nu)$. Thus, the conclusion
follows from an iterated use of Proposition~\ref{mazya}.
\end{proof}

Under \eqref{lower-est}, an assumption which will always be kept in
force hereafter, $V^m X(\Omega , \nu)$ is easily seen to be a normed
linear space, equipped with the norm
\[
\|u\|_{V^m X(\Omega , \nu)} = \sum _{k=0}^{m-1}\|\nabla ^k
u\|_{L^1(\Omega , \nu )} + \|\nabla ^m u\|_{X(\Omega , \nu )}.
\]
Standard arguments show that $V^m X(\Omega , \nu)$  is complete, and
hence a Banach space, under the additional assumption that
\[
L\sp1_{\loc}(\Omega , \nu ) \to
L^1_{\loc}(\Omega ).
\]
We  also define the subspace $ V^m_\bot X(\Omega , \nu )$ of $V^m
X(\Omega , \nu )$ as
\[
 V^m_\bot X(\Omega , \nu ) =  \bigg\{u \in
V^m X(\Omega , \nu ): \int _{\Omega} \nabla ^k u \,d \nu =0,\,\,
\hbox{for }\, k= 0, \dots , m-1 \bigg\}.
\]

The Sobolev embedding \eqref{1001} turns out to be equivalent to a Poincar\'e type
inequality for functions in $V^m_\bot X(\Omega , \nu )$.

\begin{prop}\label{poincsob}{\bf [Equivalence of Sobolev and
Poincar\'e inequalities]}
 Assume that $(\Omega , \nu)$   fulfils \eqref{lower-est} and that $m \geq 1$.
Let $\| \cdot \|_{X(0,1)}$ and $\| \cdot \|_{Y(0,1)}$ be
rearrangement-invariant function norms. Then
\begin{equation}\label{embedd} V^mX(\Omega , \nu ) \to Y(\Omega ,
\nu )\end{equation} if and only if  there exists a constant $C$
such that
\begin{equation}\label{poinc}
\|u\|_{Y(\Omega , \nu )} \leq C \|\nabla ^mu\|_{X(\Omega , \nu )}
\end{equation}
for every $u \in  V^m_\bot X(\Omega , \nu )$.
\end{prop}

\begin{proof}
Assume that \eqref{embedd} holds.
Thus, there exists a constant $C$ such
that
\begin{align}\label{poinc3}
\|u\|_{Y(\Omega ,\nu)} \leq C \Big(\sum _{k=0}^{m-1}\|\nabla ^k
u\|_{L^1(\Omega , \nu )}  + \|\nabla ^m u\|_{X(\Omega , \nu)}\Big)
\end{align}
for every $u \in V^m X(\Omega , \nu )$. Iterating inequality
\eqref{mazya2mean} implies that there exist constants $C_1,
\dots, C_m$ such that
\begin{equation}\label{poinc4}
\|u\|_{L^1(\Omega , \nu)} \leq C_1 \|\nabla u\|_{L^1(\Omega , \nu)}
\leq C_2 \|\nabla ^2u\|_{L^1(\Omega , \nu)} \leq \dots \leq C_m
\|\nabla ^m u\|_{L^1(\Omega , \nu)}
\end{equation}
 for every $u \in V^m_\bot X(\Omega
, \nu )$. By property (P5) of rearrangement-invariant function norms, there
exists a constant $C$, independent of $u$, such that $\|\nabla ^m
u\|_{L^1(\Omega , \nu)} \leq C\|\nabla ^m u\|_{X(\Omega , \nu)}$.
Thus, \eqref{poinc} follows from \eqref{poinc3} and \eqref{poinc4}.

Suppose next that  \eqref{poinc} holds. Given $k \in \N$, denote by
$\mathcal P ^k$ the space of polynomials whose degree does not
exceed $k$. Observe that $\mathcal P ^k \subset L^1(\Omega , \nu)$
for every $k \in \N$. Indeed, $\nabla ^h P =0$ for every $P \in
\mathcal P ^k$, provided that $h>k$, and hence $\mathcal P ^k
\subset V^h X(\Omega , \nu)$ for any rearrangement-invariant space
$X(\Omega , \nu)$. The inclusion $\mathcal P ^k \subset L^1(\Omega ,
\nu)$ thus follows via Corollary \ref{C:mazya-m}. Next, it is not
difficult to verify that, for each $u \in V^m X(\Omega , \nu)$,
there exists a (unique) polynomial $P_u \in \mathcal P^{m-1}$ such
that $u - P_u \in V^m _\bot X(\Omega , \nu)$. Moreover, the
coefficients of $P_u$ are linear combinations of the components of
$\int _{\Omega} \nabla ^k u \,d\nu $, for $k=0, \dots , m-1$, with
coefficients depending on $n$, $m$ and $(\Omega , \nu)$. Now, we
claim that
\begin{equation}\label{Pm}
  \mathcal P^m
\subset Y(\Omega ,\nu). \end{equation}
 This inclusion is trivial in
the case when $\Omega$ is bounded, owing to axioms (P2) and (P4) of
the definition of rearrangement-invariant function norms, since any
polynomial is bounded in $\Omega$. To verify \eqref{Pm} in the
general case, consider, for each $i=1, \dots , n$, the polynomial
$Q(x)= x_i^m \in \mathcal P ^m$. Let $P_Q \in \mathcal P^{m-1}$ be
the polynomial associated with $Q$ as above, such that $Q - P_Q \in
V^m _\bot X(\Omega , \nu)$. Note that the polynomial $P_Q$ also
depends only on $x_i$. From \eqref{poinc} applied with $u = Q - P_Q$
we deduce that $Q - P_Q \in Y(\Omega ,\nu)$. This inclusion and the
inequality $|Q - P_Q| \geq C |x_i|^m$, which holds, for a suitable
positive constant $C$,  if  $|x_i|$ is sufficiently large, tell us,
via axiom (P2) of the definition of rearrangement-invariant function
norms, that $|x_i|^m \in Y(\Omega ,\nu)$ as well. Thus, $|x|^m \in
Y(\Omega ,\nu)$, and by axiom (P2) again, any polynomial of degree
not exceeding $m$ also belongs to $Y(\Omega ,\nu)$. Hence,
\eqref{Pm} follows.
Thus, given any $u \in  V^m X(\Omega , \nu)$, we have that
\begin{align*}
\|u\|_{Y(\Omega ,\nu)} & \leq \|u- P_u \|_{Y(\Omega ,\nu)} + \|P_u
\|_{Y(\Omega ,\nu)} \\  & \leq C \|\nabla ^m u\|_{X(\Omega
,\nu)} + \sum _{k=0}^{m-1}C\int _{\Omega }
|\nabla ^k u |\,d\nu \sum _{ \alpha _1 + \dots + \alpha _n=k} \|
|x_1|^{\alpha _1} \cdots |x_n|^{\alpha _n} \|_{Y(\Omega ,\nu)}
\\  & \leq C \|\nabla ^m
u\|_{X(\Omega ,\nu)} + C'\sum
_{k=0}^{m-1}\|\nabla ^k u\|_{L^1(\Omega , \nu )},
\end{align*}
for some constants $C$ and $C'$ independent of $u$.
Hence, embedding
\eqref{embedd} follows.
\end{proof}

Let us  incidentally mention that more customary Sobolev type spaces
$W^m X(\Omega , \nu)$ can be defined as
\[
W^m X(\Omega , \nu) =   \big\{u: \hbox{$u$ is $m$-times weakly
differentiable in $\Omega$,  \, $|\nabla ^k u| \in X(\Omega, \nu)$
for $k=0, \dots , m$}\big\},
\]
and equipped with the norm
$$\|u\|_{W^m X(\Omega , \nu)} = \sum _{k=0}^{m} \|\nabla
^k u\|_{X(\Omega , \nu )}.$$    The  space $W^m X(\Omega , \nu)$ is
a normed linear space, and it is a Banach space if
\[
X_{\loc}(\Omega , \nu ) \to L^1_{\loc}(\Omega ).
\]
By the second embedding in \eqref{l1linf},
\begin{equation}\label{inclWV}
W^m X(\Omega , \nu) \to V^m X(\Omega , \nu)
 \end{equation}
 for every $(\Omega , \nu)$
 fulfilling \eqref{lower-est}, but, in general, $W^m X(\Omega , \nu) \subsetneqq V^m X(\Omega ,
 \nu)$.  For instance, if $(\Omega , \nu) = (\rn , \gamma _n)$, the
 Gauss space, and $\| \cdot \|_{X (0, 1)} = \| \cdot \|_{L^\infty (0,
 1)}$ (or $\| \cdot \|_{X (0, 1)} = \| \cdot \|_{{\rm exp} L^\beta (0,
 1)}$ for some $\beta >0$), then $V^m X(\Omega , \nu) \neq W^m X(\Omega ,
 \nu)$.
 However, the spaces $W^m X(\Omega , \nu)$   and  $V^m
X(\Omega , \nu)$ agree if condition \eqref{lower-est} is slightly
strengthened to
\begin{equation}\label{WV2}
\int _0 \frac {ds}{I_{\Omega , \nu} (s)} < \infty.
\end{equation}
Note that~\eqref{WV2} indeed implies~\eqref{lower-est}, since $\frac
{1}{I_{\Omega , \nu}}$ is a non-increasing function.

\begin{prop}\label{WV} {\bf [Condition for $W^mX(\Omega , \nu )= V^mX(\Omega , \nu )$]} Let $(\Omega ,\nu )$ be as above, and
let $m \in \N$. Assume that \eqref{WV2} holds. Let $\| \cdot
\|_{X(0,1)}$ be a rearrangement-invariant function norm. Then
\begin{equation}\label{WV1}
W^mX(\Omega , \nu )= V^mX(\Omega , \nu ),
\end{equation}
up to equivalent norms.
\end{prop}

A proof of this proposition relies upon one of our main results, and
can be found at the end of  Section \ref{S:proof-convergent}.

\section{Main results}\label{S:main}

The present section contains the main results of this paper, which link
  embeddings and Poincar\'e
inequalities for Sobolev-type spaces of arbitrary order to
isoperimetric inequalities. The relevant results
depend only on a lower bound for the isoperimetric function
$I_{\Omega,\nu}$ of $(\Omega , \nu)$ in terms of some other
non-decreasing function $I:[0, 1]\to[0,\infty)$; precisely, on the
existence of a positive constant $c$ such that
\begin{equation}\label{isop-ineq}
I_{\Omega,\nu} (s) \geq c I(c s)  \quad \textup{for}\ s\in[0, \tfrac
12].
\end{equation}
As mentioned in Proposition~\ref{mazya} and the preceding remarks, it is reasonable to suppose that the function $I_{\Omega,\nu}$ satisfies the estimate~\eqref{lower-est}. In the light of this fact, in what follows  we shall assume
that
\begin{equation}\label{E:lower-bound}
\inf _{t \in (0,1)} \frac {I(t)}t>0.
\end{equation}

\begin{thm}\label{T:reduction}{\bf [Reduction principle]}
Assume that  $(\Omega,\nu)$  fulfils \eqref{isop-ineq} for some
non-decreasing function  $I$ satisfying \eqref{E:lower-bound}.
 Let $m\in\N$, and let
$\|
\cdot \|_{X(0,1)}$ and $\|\cdot\|_{Y(0,1)}$ be
rearrangement-invariant function norms. If there exists a constant
$C_1$ such that
\begin{equation}\label{E:kernel-cond}
\left\|\int_t\sp1\frac{f(s)}{I(s)}\left(\int_t\sp
s\frac{dr}{I(r)}\right)\sp{m-1}\,ds\right\|_{Y(0,1)} \leq
C_1\left\|f\right\|_{X(0,1)}
\end{equation}
for every nonnegative $f \in  X(0,1)$, then
\begin{equation}\label{E:main-emb}
V^m  X(\Omega,\nu) \to Y(\Omega,\nu),
\end{equation}
and there exists a~constant $C_2$ such that
\begin{equation}\label{E:main-poincare}
\|u\|_{Y(\Omega,\nu)}\leq C_2\left\|\nabla\sp m
u\right\|_{X(\Omega,\nu)}
\end{equation}
for every $u \in V^m_{\bot } X(\Omega,\nu)$.
\end{thm}

\begin{remark}\label{R:star}
It turns out that inequality \eqref{E:kernel-cond} holds for
every  nonnegative $f \in  X(0,1)$ if and only if it just holds for
every nonnegative and non-increasing $f \in  X(0,1)$.
This fact will be proved in
Corollary~\ref{C:restriction-of-main-inequality-to-nonincreasing-functions},
Section \ref{S:proof-convergent},  and can be of use in concrete
applications of Theorem \ref{T:reduction}. Indeed, the available
criteria for the validity of one-dimensional inequalities for
integral operators take, in general, different forms according to
whether trial functions are arbitrary, or just monotone.
\end{remark}

As already stressed in Sections \ref{S:intro} and \ref{S:over}, the
first-order case ($m=1$) of Theorem~\ref{T:reduction} is already
well known; the novelty here amounts to the higher-order case when
$m>1$. To be more precise, when $m=1$, a version of
Theorem~\ref{T:reduction} in the standard Euclidean case, for
functions vanishing on $\partial \Omega$, is by now classical, and
has been exploited in the proof of Sobolev inequalities with sharp
constants, including \cite{Aubin, Moser, Ta, BL}. An argument
showing that \eqref{E:kernel-cond} with $m=1$ implies
\eqref{E:main-emb} and \eqref{E:main-poincare}, for functions with
arbitrary boundary values,
  for Orlicz norms,
on regular Euclidean domains, or, more generally, on domains in
Maz'ya classes, is presented \cite[Proof of Theorem 2 and Remark
2]{Ci_ind}. A proof for arbitrary rearrangement-invariant norms, in
Gauss space, is given in \cite{CP_gauss}. The same proof translates
verbatim to general measure spaces $(\Omega,\nu)$ as in Theorem
\eqref{T:reduction} -- see e.g.~\cite{MM}.

A major feature of Theorem~\ref{T:reduction} is the difference
occurring in \eqref{E:kernel-cond} between the first-order case
($m=1$) and the higher-order case ($m>1$). Indeed, the integral
operator appearing in \eqref{E:kernel-cond} when $m=1$ is just a
weighted Hardy-type operator, namely a primitive of $f$ times a
weight, whereas, in the higher-order case, a genuine kernel, with a
more complicated structure, comes into play. In fact, this seems to
be the first known instance where such a kernel operator is needed
in a reduction result for Sobolev-type embeddings. Of course, this
makes the proof of inequalities of the form \eqref{E:kernel-cond}
more challenging, although several contributions on one-dimensional
inequalities for kernel operators are fortunately available in the
literature (see e.g. the survey papers \cite{KOPR, M-R, Step}, and
the monographs \cite{EE, EKM}).

\begin{remark}\label{reverse}
{\rm As we shall see, the Sobolev embedding \eqref{E:main-emb} (or
the Poincar\'e inequality \eqref{E:main-poincare}) and inequality
\eqref{E:kernel-cond}, in which the function $I$ is equivalent to the isoperimetric function $I_{\Omega,\nu}$ on some neighborhood of zero, are actually equivalent in customary families
of measure spaces $(\Omega,\nu)$, and hence, Theorem~\ref{T:optran}
enables us to determine the optimal rearrangement-invariant target
spaces in Sobolev embeddings for these measure spaces. Incidentally,
let us mention that when $m=1$, this is the case whenever the
geometry of $(\Omega , \nu)$ allows the construction of a family of
trial functions $u$ in \eqref{E:main-emb} or \eqref{E:main-poincare}
characterized by the following properties: the level sets of $u$ are
isoperimetric (or almost isoperimetric) in $(\Omega , \nu)$;
$|\nabla u|$ is constant (or almost constant) on the boundary of the
level sets of $u$. If $m>1$, then the latter requirement has to be
complemented by requiring that the derivatives of $u$ up to the
order $m$ restricted to the boundary of the level sets satisfy
certain conditions depending on $I$.
The relevant conditions have, however, a technical nature, and it is  not
worth to state them explicitly. In fact, heuristically speaking,
properties \eqref{E:kernel-cond}, \eqref{E:main-poincare} and
\eqref{E:main-emb} turn out to be equivalent for every $m \geq 1$ on
the same measure spaces $(\Omega,\nu)$ as for $m=1$. Such
equivalence certainly holds in any customary, non-pathological
situation, including the three  frameworks to which our results will
be applied, namely John domains,  Euclidean domains from Maz'ya
classes and product probability spaces in $\rn$ extending the Gauss
space.}
\end{remark}

Now we are in a~position to characterize the space which, in the
situation discussed in Remark \ref{reverse}, is  the optimal
rearrangement-invariant target space in the Sobolev embedding
\eqref{E:main-emb}. Such an optimal space is the one associated with
the rearrangement-invariant function norm $\| \cdot
\|_{X_{m,I}(0,1)}$, whose associate norm is defined as
\begin{equation}\label{E:eucl_opt_norm}
\|f\|_{X_{m,I}'(0,1)}
  =\left\|\frac{1}{I(t)}\int _0^t \left(\int_s\sp
t\frac{dr}{I(r)}\right)\sp{m-1} f\sp{*}(s)ds\right\|_{X'(0,1)}
\end{equation}
for  $f \in \mathcal M_+(0,1)$.

\begin{thm}\label{T:optran}{\bf [Optimal
target]} Assume that $(\Omega,\nu)$, $m$, $I$ and $\| \cdot
\|_{X(0,1)}$ are as in  Theorem~\ref{T:reduction}. Then the
functional $\|\cdot\|_{X'_{m,I}(0,1)}$, given by
\eqref{E:eucl_opt_norm}, is a~rearrangement-invariant function norm,
whose associate norm $\|\cdot\|_{X_{m,I}(0,1)}$ satisfies
\begin{equation}\label{E:optemb}
 V^m X(\Omega,\nu) \to X_{m,I}(\Omega , \nu ),
 \end{equation}
and there exists a~constant $C$ such that
\begin{equation}\label{E:optran}
\|u\|_{X_{m,I}(\Omega,\nu)}\leq C\left\|\nabla\sp m
u\right\|_{X(\Omega,\nu)}
\end{equation}
for every $u \in V^m_\bot X(\Omega,\nu)$.

Moreover, if
$(\Omega , \nu)$ is such that~\eqref{E:main-emb}, or equivalently
\eqref{E:main-poincare}, implies \eqref{E:kernel-cond}, and hence
\eqref{E:kernel-cond}, \eqref{E:main-emb} and
\eqref{E:main-poincare} are equivalent, then the function norm
$\|\cdot \|_{X_{m,I}(0,1)}$ is
 optimal in
\eqref{E:optemb} and \eqref{E:optran} among all
rearrangement-invariant norms.
\end{thm}

An important special case of Theorems~\ref{T:reduction} and
\ref{T:optran} is enucleated in the following corollary.

\begin{corollary}\label{linfinity}
{\bf [Sobolev embeddings into $L^\infty$]} Assume that
$(\Omega,\nu)$, $m$, $I$ and $\| \cdot \|_{X(0,1)}$ are as in
Theorem~\ref{T:reduction}. If
\begin{equation}\label{linfinity1}
\bigg\|\frac{1}{I(s)} \bigg(\int _0^s \frac
{dr}{I(r)}\bigg)^{m-1}\bigg\|_{X'(0,1)} < \infty\,,
\end{equation}
then
\begin{equation}\label{linfinity2}
 V^m X(\Omega,\nu) \to L^\infty(\Omega , \nu ),
 \end{equation}
and there exists a~constant $C$ such that
\begin{equation}\label{linfinity3}
\|u\|_{L^\infty (\Omega,\nu)}\leq C\left\|\nabla\sp m
u\right\|_{X(\Omega,\nu)}
\end{equation}
for every $u \in V^m_\bot X(\Omega,\nu)$.

 Moreover, if $(\Omega , \nu)$ is such
that~\eqref{E:main-emb}, or equivalently \eqref{E:main-poincare},
implies \eqref{E:kernel-cond}, and hence \eqref{E:kernel-cond},
\eqref{E:main-emb} and \eqref{E:main-poincare} are equivalent, then
\eqref{linfinity1} is necessary for \eqref{linfinity2} or
\eqref{linfinity3} to hold.
\end{corollary}

\begin{remark}\label{reminf}
If
$(\Omega ,\nu)$ is such  that~\eqref{E:main-emb}, or equivalently
\eqref{E:main-poincare}, implies \eqref{E:kernel-cond}, and hence
\eqref{E:kernel-cond}, \eqref{E:main-emb} and
\eqref{E:main-poincare} are equivalent,
then \eqref{linfinity2} cannot hold, whatever $\| \cdot \|_{X(0,1)}$
is, if $I$ decays so fast at $0$ that $$\int _0  \frac {dr}{I(r)} =
\infty.$$
\end{remark}

Our last main result concerns the preservation of optimality in
targets among all rearrangement-invariant spaces under iteration
of Sobolev embeddings of arbitrary order.

\begin{thm}\label{T:reit}{\bf [Iteration principle]}
 Assume that
$(\Omega,\nu)$, $I$ and $\| \cdot \|_{X(0,1)}$ are as in
Theorem~\ref{T:reduction}.  Let $k,h\in\N$. Then
\[
(X_{k,I})_{h,I}(\Omega,\nu)= X_{k+h,I}(\Omega , \nu),
\]
up to equivalent norms.
\end{thm}

We now focus on the  case when
\begin{equation}\label{E:doubling}
\int_0\sp{s}\frac{dr}{I(r)}\approx \frac{s}{I(s)} \quad \hbox{for
$s\in(0,1)$.}
\end{equation}
If the function $I$ satisfies \eqref{E:doubling}, then the results
of Theorems \ref{T:reduction}, \ref{T:optran} and~\ref{T:reit} can
be somewhat simplified. This is the content of the next three
corollaries. Let us preliminarily observe that, since the right-hand
side of \eqref{E:doubling} does not exceed its left-hand side for
any non-decreasing function $I$, only the estimate in the reverse
direction is relevant in \eqref{E:doubling}.

\begin{corollary}\label{T:convergent}{\bf [Reduction principle under
\eqref{E:doubling}]} Let $(\Omega,\nu)$, $m$, $I$, $\| \cdot
\|_{X(0,1)}$ and $\| \cdot \|_{Y(0,1)}$ be as in
Theorem~\ref{T:reduction}. Assume, in addition, that $I$ fulfils
\eqref{E:doubling}. If there exists a constant $C_1$ such that
\begin{equation}\label{E:condition-convergent}
\left\|\int_t\sp1 f(s) \frac {s^{m-1}}{I(s)^m}
\,ds\right\|_{Y(0,1)} \leq C_1\left\|f\right\|_{X(0,1)}
\end{equation}
for every nonnegative $f \in  X(0,1)$, then
\begin{equation}\label{E:conv-embedding}
V^m  X(\Omega,\nu) \to Y(\Omega,\nu),
\end{equation}
and there
exists a~constant $C_2$ such that
\begin{equation}\label{E:conv-poincare}
\|u\|_{Y(\Omega,\nu)}\leq C_2\left\|\nabla\sp m
u\right\|_{X(\Omega,\nu)}
\end{equation}
for every $u \in V^m_{\bot } X(\Omega,\nu)$.
\end{corollary}

Let us notice that a remark parallel to Remark \ref{R:star} applies
on the equivalence of the validity of \eqref{E:condition-convergent}
for any $f$, or for any non-increasing $f$ (cf.~Proposition~\ref{link}).

The next corollary tells us that, under the extra
condition~\eqref{E:doubling}, the optimal rearrangement-invariant
target space takes a simplified form. Namely, it can be equivalently
defined via the rearrangement-invariant function norm $\| \cdot \|_{
X_{m,I}^\sharp(0,1)}$ obeying

\begin{equation}\label{E:xji-convm}
\|f\|_{( X_{m,I}^\sharp)'(0,1)}=\left\|\frac
{t^{m-1}}{I(t)^m} \int _0^t f\sp*(s)\,ds\right\|_{X'(0,1)}
\end{equation}
for every $f \in \mathcal M_+(0,1)$.

\begin{corollary}\label{T:optimalconv}{\bf [Optimal
target under \eqref{E:doubling}]} Assume that $(\Omega,\nu)$, $m$,
$I$ and $\| \cdot \|_{X(0,1)}$ are as in
Corollary~\ref{T:convergent}.  Then the functional $\|\cdot
\|_{(X_{m, I}^\sharp)'(0,1)}$, given by \eqref{E:xji-convm}, is
a~rearrangement-invariant function norm, whose
associate norm $\|\cdot \|_{X_{m, I}^\sharp(0,1)}$ satisfies
\begin{equation}\label{E:main-embedding-opt}
V^m X(\Omega,\nu) \to  X_{m, I}^\sharp(\Omega , \nu ),
\end{equation}
and there exists a~constant
$C$ such that
\begin{equation}\label{E:main-opt}
\|u\|_{X_{m, I}^\sharp(\Omega,\nu)}\leq C\left\|\nabla ^m
u\right\|_{X(\Omega,\nu)}
\end{equation}
for every $u \in V^m_\bot X(\Omega,\nu)$.

 Moreover, if $(\Omega ,
\nu)$ is such that the validity of \eqref{E:conv-embedding}, or
equivalently \eqref{E:conv-poincare},
 implies \eqref{E:condition-convergent},
and hence \eqref{E:condition-convergent}, \eqref{E:conv-embedding}
and  \eqref{E:conv-poincare} are equivalent,
 then the
function norm $\| \cdot \|_{ X_{m, I}^\sharp(0,1)}$ is optimal in
\eqref{E:main-embedding-opt} and \eqref{E:main-opt} among all
rearrangement-invariant norms.
\end{corollary}

We conclude this section with a~stability result for the
iterated embeddings under the additional
condition~\eqref{E:doubling}.

\begin{corollary}\label{T:iterconv}{\bf [Iteration principle under \eqref{E:doubling}]}
Assume that $(\Omega,\nu)$,
$I$ and $\| \cdot \|_{X(0,1)}$ are as in
Corollary~\ref{T:convergent}.
 Let $k,h\in\mathbb N$.  Then
\[
 \big(X_{k, I}^\sharp\big)_{h,I}^\sharp(\Omega,\nu)= X_{k+h,
I}^\sharp (\Omega , \nu),
\]
up to equivalent norms.
\end{corollary}

\section{Euclidean--Sobolev embeddings}\label{S:higher-euclid}

The main results of this section are reduction theorems and their consequences for
 Euclidean Sobolev embeddings of arbitrary order
$m$ on  John domains, and on domains from  Maz'ya classes.

We begin with the reduction theorem for John domains. Recall that a bounded open set $\Omega$ in $\rn$ is called a~\textit{John
domain} if there exist a constant $c \in (0,1)$ and a point $x_0 \in
\Omega$ such that for every $x \in \Omega$ there exists a
rectifiable curve $\varpi : [0, l] \to \Omega$, parameterized by
arclength, such that $\varpi (0)=x$, $\varpi (l) = x_0$, and
$${\dist}\, (\varpi (r) , \partial \Omega ) \geq c r \qquad
\hbox{for $r \in [0, l]$.}$$

\begin{thm}\label{T:eucl_reduction-john}{\bf [Reduction principle for
John domains]} Let $n\in\N$, $n\geq 2$, and let $m\in\N$.
Assume that $\Omega$ is a John domain in $\rn$. Let $\| \cdot
\|_{X(0,1)}$ and $\| \cdot \|_{Y(0,1)}$ be rearrangement-invariant
function norms. Then the following assertions are equivalent.

\textup{(i)} The Hardy type inequality
\begin{equation}\label{E:eucl_condition-john}
  \left\|\int_t\sp1f(s) s^{-1+ \frac mn} \,ds \right\|_{ Y(0,1)}
  \leq C_1\left\|f\right\|_{X(0,1)}
\end{equation}
holds for some constant $C_1$, and for every nonnegative $f\in
X(0,1)$.

\textup{(ii)} The Sobolev embedding
\begin{equation}\label{E:eucl_embedd-john}
V^mX(\Omega ) \to Y(\Omega )
\end{equation}
 holds.

\textup{(iii)}
   The Poincar\'e
inequality
\begin{equation}\label{E:eucl_reduction-john}
  \left\|u \right\|_{Y(\Omega)}
  \leq C_2\left\|\nabla ^m u\right\|_{X(\Omega)}
\end{equation}
 holds for some constant
$C_2$ and   every $u\in V^m_\bot X(\Omega)$.
\end{thm}

Forerunners of Theorem \ref{T:eucl_reduction-john} are known. The
first order case ($m=1$) on Lipschitz domains was obtained
in~\cite{EKP}. In the case when $m=2$, and functions vanishing on
$\partial \Omega$ are considered, the equivalence of
\eqref{E:eucl_condition-john} and \eqref{E:eucl_reduction-john} was
proved in~\cite{Ci1}, as a consequence of a non-standard
rearrangement inequality for second-order derivatives (see also
\cite{Cianchiduke} for a related one-dimensional second-order
rearrangement inequality). The equivalence of
\eqref{E:eucl_condition-john} and \eqref{E:eucl_embedd-john}, when
 $m\leq n-1$ and $\Omega$ is a Lipschitz domain, was
established in~\cite{T2} by a~method relying upon interpolation
techniques. Such a method does not carry over to the more general
setting of Theorem \ref{T:eucl_reduction-john}, since it requires
that $\Omega$ be an extension domain.

Let us also warn that results reducing higher-order Sobolev
embeddings to one-dimensional inequalities can be obtained via more
standard methods, such as, for instance, representation formulas of
convolution type combined with O'Neil rearrangement estimates for
convolutions, or plain iteration of certain first-order pointwise
rearrangement estimates \cite{MM1}. However, these approaches  lead
to optimal Sobolev embeddings only under additional technical
assumptions on the involved rearrangement-invariant function norms
$\|\cdot \|_{X(0,1)}$ and $\|\cdot \|_{Y(0,1)}$.

Given  a rearrangement-invariant function norm $\| \cdot
\|_{X(0,1)}$ and   $m \in \N$, we
define $\| \cdot \|_{X_{m,\operatorname{John}}(0,1)}$ as
the rearrangement-invariant function norm, whose associate function
norm is given by
\begin{equation}\label{E:eucl_opt_norm-john}
\|f\|_{X_{m,\operatorname{John}}'(0,1)}
  =\left\|s\sp{-1+\frac mn}\int _0^s f\sp{*}(r)dr\right\|_{X'(0,1)}
\end{equation}
for  $f \in \M_+ (0,1)$. The function norm $\| \cdot \|_{X_{m,\operatorname{John}}(0,1)}$ is
 optimal, as a target, for Sobolev embeddings of $V^mX(\Omega )$.

\begin{thm}\label{T:eucl_optimal-john}{\bf [Optimal
target for John domains]} Let $n$, $m$, $\Omega$ and  $\| \cdot
\|_{X(0,1)}$ be as in Theorem \ref{T:eucl_reduction-john}.
Then the
functional $\|\cdot \|_{X_{m,\operatorname{John}}'(0,1)}$, given by
\eqref{E:eucl_opt_norm-john}, is a~rearrangement-invariant function
norm, whose associate norm $\|\cdot \|_{X_{m,\operatorname{John}}(0,1)}$ satisfies
\begin{equation}\label{E:eucl_opt_norm_-john}
V^mX(\Omega ) \to X_{m,\operatorname{John}}(\Omega ),
\end{equation}
and
\begin{equation}\label{E:eucl_opt_poinc_-john}
\|u\|_{X_{m,\operatorname{John}}(\Omega )}\leq C\left\|\nabla ^m
u\right\|_{X(\Omega )}
\end{equation}
for some constant $C$ and every $u \in V^m_\bot X(\Omega)$.

Moreover,  the function norm $\|\cdot\|_{X_{m,\operatorname{John}}(0,1)}$ is
optimal in \eqref{E:eucl_opt_norm_-john} and
\eqref{E:eucl_opt_poinc_-john} among all rearrangement-invariant
norms.
\end{thm}

The iteration principle for optimal target norms in Sobolev embeddings on John domains reads as follows.

\begin{thm}\label{T:iteration-john}{\bf [Iteration principle for
John domains]} Let $n\in\N$, $\Omega$ and $\| \cdot \|_{X(0,1)}$ be
as in Theorem \ref{T:eucl_reduction-john}.
Let $k,
h \in \N$. Then
\[
\big(X_{k,\operatorname{John}}\big)_{h,\operatorname{John}}(\Omega )=
X_{k+h,\operatorname{John}}(\Omega ),
\]
up to equivalent norms.
\end{thm}

Let us now focus on  Maz'ya classes of domains. Given $\alpha \in
[\frac 1{n'} , 1]$,  we denote by $\mathcal J_\alpha$
the~\textit{Maz'ya class} of all  Euclidean domains $\Omega $
satisfying \eqref{isop-ineq}, with  $I(s) = s^\alpha$ for $s \in [0,
\frac 12]$, namely domains $\Omega$ in $\rn$ such that
\[
I_\Omega (s) \geq C s^\alpha \quad \hbox{for $s \in [0, \frac 12]$,}
\]
for some positive constant $C$. Thanks to \eqref{Ijohn}, any John
domain belongs to the class $\mathcal J_{\frac 1{n'}}$.

The reduction theorem in the class $\mathcal
J_\alpha$ takes the following form.

\begin{thm}\label{T:eucl_reduction-div}{\bf [Reduction principle
for Maz'ya classes]} Let $n\in\N$, $n\geq 2$,  $m\in\N$, and
 $\alpha \in [\frac 1{n'}, 1]$. Let $\| \cdot \|_{X(0,1)}$ and
$\| \cdot \|_{Y(0,1)}$ be rearrangement-invariant function norms.
   Assume that either $\alpha \in [\frac 1{n'}, 1)$ and
there exists a constant $C_1$ such that
\begin{equation}\label{E:eucl_condition-div}
  \left\|\int_t\sp1 f(s) s^{-1+m(1- \alpha)}\,ds\right\|_{ Y(0,1)}
  \leq C_1\left\|f\right\|_{X(0,1)}
\end{equation}
for every nonnegative $f\in X(0,1)$, or $\alpha =1$ and there
exists a constant $C_1$ such that
\begin{equation}\label{E:eucl_condition-div1}
  \left\|\int_t\sp1 f(s) \frac 1s \bigg(\log \frac st \bigg)^{m-1}\,ds\right\|_{
Y(0,1)}
  \leq C_1\left\|f\right\|_{X(0,1)}
\end{equation}
for every nonnegative $f\in X(0,1)$.  Then the Sobolev embedding
\begin{equation}\label{E:eucl_embedd-div}
V^mX(\Omega ) \to Y(\Omega )
\end{equation}
 holds  for every
$\Omega \in \mathcal J _\alpha$ and, equivalently, the Poincar\'e
inequality
\begin{equation}\label{E:eucl_reduction-div}
  \left\|u \right\|_{Y(\Omega)}
  \leq C_2\left\|\nabla ^m u\right\|_{X(\Omega)}
\end{equation}
holds for every $\Omega \in \mathcal J _\alpha$, for some constant
$C_2$, depending on $\Omega, m, X$ and $Y$, and   every $u\in V^m_\bot X(\Omega)$.

Conversely, if the Sobolev embedding \eqref{E:eucl_embedd-div},
or, equivalently, the Poincar\'e inequality
\eqref{E:eucl_reduction-div}, holds for every $\Omega \in \mathcal
J _\alpha$, then either inequality \eqref{E:eucl_condition-div}, or
 \eqref{E:eucl_condition-div1} holds, according to whether  $\alpha \in [\frac 1{n'}, 1)$ or $\alpha =1$ .
\end{thm}

A major consequence of Theorem \ref{T:eucl_reduction-div} is the
identification of the optimal rearrangement-invariant target space
$Y(\Omega)$ associated with a given domain $X(\Omega)$ in
embedding \eqref{E:eucl_embedd-div} as $\Omega$ is allowed to
range among all domains in the class $\mathcal J _\alpha $. This
is the content of the next result.  The rearrangement-invariant
function norm yielding such an optimal space will be denoted by
$\| \cdot \|_{X_{m, \alpha}(0,1)}$. Given a
rearrangement-invariant function norm $\| \cdot \|_{X(0,1)}$,
 $m \in \N$, and $\alpha \in [\frac 1{n'}, 1]$, it is characterized
through its associate function norm defined  by
\begin{equation}\label{E:eucl_opt_norm-mazya}
\|f\|_{X_{m, \alpha}'(0,1)}
  = \begin{cases}
  \left\|s\sp{-1+ m(1-\alpha )}\int _0^s
  f\sp{*}(r)dr\right\|_{X'(0,1)} & \hbox{if} \,\, \alpha \in [\frac 1{n'}, 1),
  \\
  \\
  \left\|\frac 1s\int _0^s  \big(\log \frac sr
\big)^{m-1}f\sp{*}(r)dr\right\|_{X'(0,1)} & \hbox{if} \,\,\alpha
=1,
\end{cases}
\end{equation}
for  $f \in \M _+ (0,1)$.

\begin{thm}\label{T:eucl_optimal-div}{\bf [Optimal target for
Maz'ya classes]} Let  $n\in\N$, $n\geq 2$,  $m\in\N$,  $\alpha$ and
$\| \cdot \|_{X(0,1)}$  be as in Theorem \ref{T:eucl_reduction-div}.
Then the functional $\|\cdot \|_{X_{m, \alpha}'(0,1)}$,
given by \eqref{E:eucl_opt_norm-mazya}, is
a~rearrangement-invariant function norm, whose associate norm
$\|\cdot \|_{X_{m, \alpha}(0,1)}$ satisfies
\begin{equation}\label{E:eucl_opt_norm_1-div}
V^mX(\Omega ) \to X_{m,\alpha}(\Omega )
\end{equation}
for every $\Omega \in \mathcal J _\alpha$, and
\begin{equation}\label{E:eucl_opt_poinc_1-div}
  \left\|u \right\|_{X_{m,\alpha}(\Omega )}
  \leq C\left\|\nabla ^m u\right\|_{X(\Omega)}
\end{equation}
 for every $\Omega \in \mathcal J _\alpha$, for some constant
$C$, depending on $\Omega, m, X$ and $Y$, and   every $u\in V^m_\bot X(\Omega)$.

 Moreover,  the
function norm $\| \cdot \|_{X_{m,\alpha}(0,1)}$ is optimal in
\eqref{E:eucl_opt_norm_1-div} and \eqref{E:eucl_opt_poinc_1-div}
among all rearrangement-invariant norms, as  $\Omega$ ranges in $\mathcal
J _\alpha$.
\end{thm}

Theorem \ref{T:eucl_optimal-div} is a straightforward consequence of
Theorem  \ref{T:eucl_reduction-div}, and either  Corollary
\ref{T:optimalconv} or Theorem \ref{T:optran}, according to whether
$\alpha \in [\frac 1{n'}, 1)$ or $\alpha =1$.

The stability of the process of finding optimal
rearrangement-invariant targets in Euclidean Sobolev embeddings on
Maz'ya domains under iteration is the object of the last main result
of the present section. This is the key ingredient which bridges the
first-order case of Theorems \ref{T:eucl_reduction-div} and
\ref{T:eucl_optimal-div} to their higher-order versions.

\begin{thm}\label{T:iteration-div}{\bf [Iteration principle for Maz'ya classes]}
Let $n\in\N$,  $\alpha \in [\frac 1{n'}, 1]$ and $\| \cdot
\|_{X(0,1)}$ be as in Theorem \ref{T:eucl_reduction-div}. Let $k, h
\in \N$. Assume that $\Omega \in \mathcal J _\alpha$.
 Then,
\[
(X_{k,\alpha})_{h,\alpha}(\Omega )= X_{k+h,\alpha}(\Omega ),
\]
up to equivalent norms.
\end{thm}

Theorem \ref{T:iteration-div} follows from a specialization of
Corollary \ref{T:iterconv} ($\alpha \in [\frac 1{n'}, 1)$),  or
Theorem~\ref{T:reit} ($\alpha =1$).

\begin{remark}\label{R:comparison}
Note that there is one important difference between the reduction
and the  optimal-target theorem concerning John domains on the one
hand, and their counterparts for general Maz'ya domains on the other
hand. Namely, the equivalence in
Theorem~\ref{T:eucl_reduction-john} and the optimality result
in~Theorem~\ref{T:eucl_optimal-john} are valid \textit{for each single
John domain}, whereas the necessity of
condition~\eqref{E:eucl_condition-div}
for~\eqref{E:eucl_condition-div1} (or~\eqref{E:eucl_embedd-div}) in
Theorem~\ref{T:eucl_reduction-div} as well as the optimality of the
target space in Theorem~\ref{T:eucl_optimal-div} are valid in the
 \textit{class of all} $\Omega\in\mathcal J_{\alpha}$. This is
 inevitable, since, of course, each class $\mathcal J_{\alpha}$
 contains all regular domains, and for such domains Sobolev
 embeddings with stronger target norms hold.
\end{remark}

The remaining part of this section is devoted to applications of
Theorems \ref{T:eucl_reduction-div}--\ref{T:iteration-div} to
customary function  norms. Consider first the case when Lebesgue or
Lorentz norms are concerned. Our conclusions take a different form,
according to whether $\alpha \in [\frac 1{n'}, 1)$, or $\alpha =1$.

We begin by assuming that $\alpha
\in [\frac 1{n'}, 1)$.  Note that results for regular (i.e.~John)
domains are covered by the  choice $\alpha = \frac 1{n'}$.

Sobolev embeddings involving usual Lebesgue norms are contained
in the following theorem.

\begin{thm}\label{EX:eucl_lebesgue}
Let $n\in\N$, $n\geq 2$, and let $\Omega \in \mathcal J _\alpha$ for
some $\alpha \in [\frac1{n'},1)$. Let $m\in\N$ and $p \in [1,
\infty]$. Then
\begin{equation}\label{exeucl1}
 V^mL^p(\Omega ) \to \begin{cases}
L^{\frac {p}{1-mp(1-\alpha)}}(\Omega ) & \hbox{if $m(1-\alpha)<1$
and $1 \leq p<
\frac 1{m(1-\alpha )}$,} \\
L^{r} (\Omega ) & \hbox{for any $r\in[1,\infty)$, if $m(1-\alpha)<1$
and $ p= \frac 1{m(1-\alpha
)}$,} \\
L^\infty (\Omega )  & \hbox{otherwise.}
\end{cases}
\end{equation}
Moreover, in the first and the third cases, the target spaces in \eqref{exeucl1} are optimal among all Lebesgue
spaces, as $\Omega$ ranges in $\mathcal
J _\alpha$.
\end{thm}

Although the target spaces in~\eqref{exeucl1} cannot be improved in
the class of Lebesgue spaces, the conclusions of \eqref{exeucl1} can
be strengthened if more general rearrangement-invariant spaces are
employed. Such a strengthening can be obtained as a special case of
a Sobolev embedding for Lorentz spaces which reads as follows.

\begin{thm}\label{EX:eucl_lorentz}
Let $n\in\N$, $n\geq 2$, and let $\Omega \in \mathcal J _\alpha$ for
some $\alpha \in [\frac1{n'},1)$. Let $m\in\N$ and $p,q \in [1,
\infty]$. Assume that one of the conditions in~\eqref{E:new-star} holds. Then
\begin{equation}\label{lorentz1}
 V^mL^{p,q}(\Omega ) \to
 \begin{cases}
L^{\frac {p}{1-mp(1-\alpha)}, q}(\Omega ) & \hbox{if $m(1-\alpha)<1$
and $1 \leq p<
\frac 1{m(1-\alpha )}$,} \\
L^{\infty, q; -1} (\Omega ) & \hbox{if $m(1-\alpha)<1$, $ p= \frac
1{m(1-\alpha
)}$ and $q>1$,} \\
L^\infty (\Omega )  & \hbox{otherwise, }
\end{cases}
\end{equation}
Moreover, the target spaces in \eqref{lorentz1} are optimal among
all rearrangement-invariant spaces, as $\Omega$ ranges in $\mathcal
J _\alpha$.
\end{thm}

The particular choice of parameters $p=q$, $1\leq
p<\frac1{m(1-\alpha)}$ in Theorem~\ref{EX:eucl_lorentz} shows that
$$
V\sp mL\sp{p}(\Omega)\hra L^{\frac {p}{1-mp(1-\alpha)}, p}(\Omega).
$$
This is  a~non-trivial strengthening  of the first embedding
in~\eqref{exeucl1}, since $L^{\frac {p}{1-mp(1-\alpha)},
p}(\Omega)\subsetneqq L^{\frac {p}{1-mp(1-\alpha)}}$.
Likewise, the choice
$m(1-\alpha)<1$ and $p=q=\frac 1{m(1-\alpha )}$ shows that also
the second embedding in~\eqref{exeucl1} can be in fact essentially
improved by
$$
V\sp mL\sp{p}(\Omega)\hra L\sp{\infty,p;-1}(\Omega).
$$

Assume now that $\alpha =1$. The embedding theorem in Lebesgue
spaces  takes the following form.

\begin{thm}\label{T:eucl_lebesgue-alpha=1}
Let $n\in\N$, $n\geq 2$, and let $\Omega \in \mathcal J _1$. Let
$m\in\N$ and $p \in [1, \infty]$. Then
\begin{equation}\label{exeucl1-alpha=1}
 V^mL^p(\Omega ) \to \begin{cases}
L^{p}(\Omega ) & \hbox{if  $1 \leq p< \infty$,} \\
L^{r} (\Omega ) & \hbox{for any $r\in[1,\infty)$, if
 $ p= \infty$.}
\end{cases}
\end{equation}
Moreover, in the former case of \eqref{exeucl1-alpha=1}, the target space is optimal among all
Lebesgue spaces, as $\Omega$ ranges in $\mathcal J _1$.
\end{thm}

Optimal embeddings for Lorentz-Sobolev spaces are provided in the
next theorem.

\begin{thm}\label{T:eucl_lorentz-alpha=1}
Let $n\in\N$, $n\geq 2$, and let $\Omega \in \mathcal J _1$. Let
$m\in\N$ and $p,q \in [1, \infty]$. Assume that one of the conditions in~\eqref{E:new-star} holds. Then
\begin{equation}\label{E:lorentz-alpha=1}
 V^mL^{p,q}(\Omega ) \to
\begin{cases}
L^{p,q}(\Omega ) & \textup{if}\  1 \leq p< \infty, \\
\exp L\sp{\frac 1m} (\Omega ) & \textup{if}\ p= q=\infty.
\end{cases}
\end{equation}
The target spaces are optimal in \eqref{E:lorentz-alpha=1} among all
rearrangement-invariant spaces, as $\Omega$ ranges in $\mathcal J _1$.
\end{thm}

Our last application in this section concerns Orlicz-Sobolev spaces.
Let $n\in\N$, $n\geq 2$, $m\in\N$,  $\alpha \in [\frac1{n'},1)$, and
let $A$ be a Young function. We may assume, without loss of
generality, that $m<\frac 1{1-\alpha}$ and
\begin{equation}\label{orlicz1}
\int _0 \bigg(\frac t{A(t)}\bigg)^{\frac {m(1-\alpha )}{1-m(1-\alpha
)}} dt < \infty.
\end{equation}
Indeed, by \eqref{B.6}, the function $A$ can be modified near $0$,
if necessary, in such a way that \eqref{orlicz1} is fulfilled, on
leaving the space $V^m L^A(\Omega )$ unchanged (up to equivalent
norms).

If $m< \frac 1{1-\alpha}$ and the integral
\begin{equation}\label{orlicz1bis}
\int ^\infty \bigg(\frac t{A(t)}\bigg)^{\frac {m(1-\alpha
)}{1-m(1-\alpha )}} dt
\end{equation}
diverges, we define the function $H_{m, \alpha} : [0, \infty ) \to
[0, \infty )$  as
\[
H_{m, \alpha}(s) = \bigg(\int _0^s \bigg(\frac t{A(t)}\bigg)^{\frac
{m(1-\alpha )}{1-m(1-\alpha )}} dt\bigg)^{1-m(1-\alpha )} \qquad
\hbox{for $s \geq 0$,}
\]
and the Young function $A_{m,\alpha}$ as
\[
A_{m,\alpha} (t) = A(H_{m, \alpha}^{-1}(t)) \qquad \hbox{for $t \geq
0$.}
\]

\begin{thm}\label{EX:eucl_orlicz}
Assume that $n\in\N$, $n\geq 2$,  $m\in\N$, $\alpha \in
[\frac1{n'},1)$ and   $\Omega \in \mathcal J _\alpha$. Let  $A$ be a
Young function fulfilling \eqref{orlicz1}. Then
\begin{equation}\label{orlicz4}
 V^mL^{A}(\Omega ) \to \begin{cases}
L^{A_{m,\alpha}}(\Omega ) & \hbox{if $m< \frac 1{1-\alpha}$, and the
integral
\eqref{orlicz1bis} diverges,} \\
L^\infty (\Omega )  & \hbox{if either $m\geq \frac 1{1-\alpha}$, or
$m< \frac 1{1-\alpha}$ and the integral \eqref{orlicz1bis}
converges.}
\end{cases}
\end{equation}
Moreover, the target spaces in \eqref{orlicz4} are optimal among all
Orlicz spaces, as $\Omega$ ranges in $\mathcal J _\alpha$.
\end{thm}
Theorem \ref{EX:eucl_orlicz} follows from Theorem
\ref{T:eucl_reduction-div}, via \cite[Theorem 4]{CianchiMZ}.

The first case of embedding \eqref{orlicz4} can be enhanced, on
replacing the optimal Orlicz target spaces with the optimal
rearrangement-invariant target spaces. The latter turn out to belong
to the family of Orlicz-Lorentz spaces defined in Section
\ref{S:measurable}.

Assume that $m< \frac 1{1-\alpha}$, and the integral
\eqref{orlicz1bis} diverges. Let $a$ be the left-continuous function
appearing in \eqref{young},
 and let $B$ be the
Young function given by
\[
B(t) = \int _0^t b(\tau ) d\tau \qquad \quad \hbox{for $t \geq 0$},
\]
where $b$ is the non-decreasing, left-continuous function in $[0,
\infty )$ obeying
\[
 b^{-1}(s ) = \bigg(\int _{a^{-1}(s )}^{\infty}\bigg(\int _0^\tau
\bigg(\frac{1}{a(t)}\bigg)^{\frac{m(1-\alpha)}{1-m(1-\alpha)}}
 dt\bigg)^{-\frac{1}{m(1-\alpha )}}\frac{d\tau}{a(\tau )^{\frac{1}{1-m(1-\alpha
)}}}\bigg)^{\frac{m(1-\alpha )}{m(1-\alpha ) -1}}\,\,\,\quad{\rm
for}\,\,\,
 s \geq 0\,.
\]
Here,  $a^{-1}$ and $b^{-1}$ denote the (generalized)
left-continuous inverses of $a$ and $b$, respectively.

Recall from Section \ref{S:measurable} that
 $L(\frac 1{m(1-\alpha )}, 1,
B)(\Omega)$ is the Orlicz-Lorentz space built upon  the function
norm given by
\[
\|f\|_{L(\frac 1{m(1-\alpha )}, 1, B)(0,1)} = \|s^{-m (1-\alpha )}
f^* (s)\|_{L^B(0, 1)}
\]
for $f \in \mathcal M _+(0,1)$.

\begin{thm}\label{EX:eucl_orlicz-lorentz}
Assume that $n\in\N$,
$n\geq 2$,  $m\in\N$, $\alpha \in [\frac1{n'},1)$ and   $\Omega \in
\mathcal J _\alpha$. Let  $A$ be a Young function fulfilling
\eqref{orlicz1}.
 Assume that $m<
\tfrac 1{1-\alpha}$, and  the integral in \eqref{orlicz1bis}
diverges. Then
\begin{equation}\label{orlicz5}
 V^mL^{A}(\Omega ) \to L(\tfrac 1{m(1-\alpha )}, 1,
B)(\Omega),
\end{equation}
 and the target space in \eqref{orlicz5} is optimal among
all rearrangement-invariant spaces, as $\Omega$ ranges in $\mathcal
J _\alpha$.
\end{thm}

Embedding \eqref{orlicz5} is a consequence of  Theorem
\ref{T:eucl_reduction-div}, and
  of
\cite[inequality (3.1)]{Cianchiibero}.

\begin{example}\label{log}
{\rm Consider the case when
$$\hbox{ $A(t) \approx t^p (\log t)^\beta$ near infinity, where either $p>1$
and $\beta \in R$, or $p=1$ and $ \beta \geq 0$.}$$ Hence,
$L^A(\Omega) = L^p{\log}^\beta L(\Omega )$. An application of Theorem
\ref{EX:eucl_orlicz} tells us that
\begin{equation}\label{log1}
 V^mL^p{\log}^\beta L(\Omega ) \to \begin{cases}
L^{\frac p{1-pm(1-\alpha )}}{\log}^\frac \beta{1-pm(1-\alpha
)}L(\Omega ) & \hbox{if
$mp(1-\alpha )< 1$,} \\
\exp L^{\frac 1{1-(1+\beta )m(1-\alpha)}} (\Omega ) & \hbox{if
$mp(1-\alpha )= 1$ and $\beta < \frac{1-m(1-\alpha
)}{m(1-\alpha )}$,} \\
\exp\exp L^{\frac 1{1-m(1-\alpha)}} (\Omega ) & \hbox{if
$mp(1-\alpha )= 1$ and $\beta = \frac{1-m(1-\alpha )}{m(1-\alpha
)}$,}
\\
L^\infty (\Omega )  & \hbox{if either $mp(1-\alpha )> 1$,} \\ &
\hbox{or $mp(1-\alpha )= 1$ and $\beta > \frac{1-m(1-\alpha
)}{m(1-\alpha )}$.}
\end{cases}
\end{equation}
Moreover, the target spaces in \eqref{log1} are optimal among all
Orlicz spaces, as $\Omega$ ranges in $\mathcal J _\alpha$.

The first three embeddings in \eqref{log1} can be improved on
allowing more general rearrangement-invariant target spaces. Indeed,
we have that
\begin{equation}\label{log2}
 V^mL^p{\log}^\beta L(\Omega ) \to \begin{cases}
L^{\frac p{1-pm(1-\alpha )},p;\frac \beta p}(\Omega ) & \hbox{if
$mp(1-\alpha )<
1$,} \\
L^{\infty,\frac 1{m(1-\alpha)};m(1-\alpha)\beta -1}(\Omega ) &
\hbox{if $mp(1-\alpha )= 1$ and $\beta < \frac{1-m(1-\alpha
)}{m(1-\alpha )}$,} \\
L^{\infty,\frac 1{m(1-\alpha)};-m(1-\alpha), -1}(\Omega )& \hbox{if
$mp(1-\alpha )= 1$ and $\beta = \frac{1-m(1-\alpha )}{m(1-\alpha
)}$,}
\end{cases}
\end{equation}
the targets  being optimal among all rearrangement-invariant spaces
in \eqref{log2}  as $\Omega$ ranges among all domains in $\mathcal J _\alpha$.
This is a consequence of Theorem \ref{EX:eucl_orlicz-lorentz}, and
of the fact that the Orlicz-Lorentz spaces $L(\frac 1{m(1-\alpha )},
1, B)(\Omega)$ associated with the present choices of the function
$A$ agree (up to equivalent norms) with the (generalized)
Lorentz-Zygmund spaces appearing on the right-hand side of
\eqref{log2}.
 }
\end{example}

\section{Sobolev embeddings in product probability
spaces}\label{S:probability}

The  class of product probability measures in $\rn$, $n \geq 1$,
which we consider in this section  arises in connection with the
study of  generalized hypercontractivity theory and integrability
properties of the associated heat semigroups. The isoperimetric
problem in the corresponding probability spaces was studied in
\cite{BCR} -- see also \cite{ BCR1, BH,  Le1, Le2}.

Assume that $\Phi: [0,\infty) \to [0,\infty)$ is a strictly
increasing convex function on $[0,\infty)$, twice  continuously
differentiable in $(0, \infty)$, such that $\sqrt{\Phi}$ is concave
and $\Phi (0)=0$. Let  $\mu_{\Phi}$ be the probability measure on
$\R$ given by
\begin{equation}\label{E:measure}
d\mu_{\Phi}(x)=c_{\Phi}e\sp{-\Phi(|x|)}\,dx,
\end{equation}
where $c_{\Phi}$ is a constant chosen in such a way that
$\mu_{\Phi}(\R)=1$. The product
 measure $\mufn$ on $\rn$, $n \geq 1$, generated by   $\mu_{\Phi}$, is then defined as
\begin{equation}\label{prodmeasure}
\mufn=\underbrace{\muf\times\dots\times\muf}_{n-\textup{times}}.
\end{equation}
Clearly, $\mu_{\Phi, 1}= \mu_{\Phi}$, and  $(\rn,\mufn)$ is a probability space for every $n\in\N$.

The main example of a~measure $\mu_{\Phi}$ is obtained by taking
\begin{equation}\label{phigauss}
\Phi(t)=\tfrac12t\sp{2}.
\end{equation}
 This choice yields
$\mu_{\Phi,n}=\gamma_n$,  the Gauss measure which obeys
\[
d \gamma _n (x) =  (2\pi)^{-\frac n2} e^{-\frac{|x|^2}{2}}dx.
\]
More generally, given any $\beta\in[1,2]$,   the  Boltzmann measure
$\gamma _{n,\beta }$ in $\rn$, associated with
\begin{equation}\label{E:def-boltzmann}
\Phi (t) = \tfrac 1\beta t^\beta,
\end{equation}
satisfies the above assumptions.
\par
Let $H:\R\to(0,1)$ be defined as
\begin{equation}\label{H}
H(t)=\int_t\sp{\infty}c_{\Phi}e\sp{-\Phi(|r|)}\,dr \quad \hbox{for
$t \in \R$,}
\end{equation}
and let $F_\Phi : [0,1] \to [0, \infty )$ be given by
\[
F_\Phi (s)  = c_{\Phi}e\sp{-\Phi(|H^{-1}(s)|)} \quad \hbox{for $s
\in (0,1)$, \,\, and} \quad F_\Phi (0)= F_\Phi (1) =0.
\]
 Since $\mu _\Phi$ is
a probability measure and $\mu_{\Phi,n}$ is defined  by
\eqref{prodmeasure}, it is easily seen that, for each $i=1, \dots ,
n$,
\begin{equation}\label{measphi}
\mu_{\Phi,n} (\{(x_1, \dots , x_n): x_i >t\}) = H(t) \quad \hbox{for
$t \in \R$,}
\end{equation}
and
\[
P_{\mu_{\Phi,n}}(\{(x_1, \dots , x_n): x_i >t\}, \rn)=
c_{\Phi}e\sp{-\Phi(|t|)} = -H'(t) \quad \hbox{for $t \in \R$.}
\]
Hence, $F_\Phi (s)$ agrees with the perimeter of any half-space of
the form $\{x_i >t\}$, whose measure is $s$.

Next, define $L_\Phi : [0, 1] \to [0, \infty)$ as
\begin{equation}\label{Lphi}
L_\Phi (s)  =  s \Phi ' \big(\Phi ^{-1}(\log \big (\tfrac
2s\big)\big)\big)\qquad  \hbox{for $s \in (0,1]$, \,\, and} \quad
L_\Phi (0)=0.
\end{equation}
 Then the isoperimetric function of $(\rn , \mu_{\Phi
, n})$ satisfies
\begin{equation}\label{E:profile}
I_{(\rn , \mu_{\Phi , n})}(s)   \approx F_\Phi (s) \approx L_\Phi
(s) \quad \hbox{for $s\in[0,\frac 12]$}
\end{equation}
(see \cite[Proposition 13 and Theorem 15]{BCR}; note that the second
equivalence in \eqref{E:profile} also relies upon Lemma \ref{delta2} \textup{(ii)}
of Section \ref{S:probability_proofs}). Furthermore, half-spaces,
whose boundary is orthogonal to a coordinate axis,
 are \lq\lq approximate solutions" to the isoperimetric
problem   in $(\rn , \mu_{\Phi , n})$ in the sense that there exist
constants $C_1$ and $C_2$, depending on $n$, such that, for every $s
\in (0, 1)$, any such half-space $V$ with measure $s$  satisfies
$$C_1 P_{\mu_{\Phi , n}} (V, \rn ) \leq I_{(\rn , \mu_{\Phi ,
n})}(s)\leq C_2 P_{\mu_{\Phi , n}} (V, \rn ).$$
In the special case
when $ \mu_{\Phi , n}=  \gamma _n$, the Gauss measure, equation
\eqref{E:profile} yields
\[
I_{(\rn , \gamma _n)}(s)   \approx s \big(\log \tfrac 2s\big)^{\frac
12} \quad \hbox{for $s\in(0,\frac 12]$}.
\]
Moreover, any half-space is, in fact, an exact minimizer in the
isoperimetric inequality \cite{Bor, ST}.

\smallskip

Our reduction theorem for Sobolev embeddings in product probability
spaces reads as follows.

\begin{thm}\label{T:prob-reduction}{\bf [Reduction principle for
product probability spaces]} Let $n \in \N$, $m \in\N$, let
$\mu_{\Phi,n}$ be the probability measure defined by
\eqref{prodmeasure}, and let $\| \cdot \|_{X(0,1)}$ and $\| \cdot
\|_{Y(0,1)}$ be rearrangement-invariant function norms. Then the
following facts are equivalent.

\textup{(i)} The inequality
\begin{equation}\label{E:hardy}
  \left\|\left(\frac{\Phi^{-1}(\log \frac{2}{t})}{\log \frac{2}{t}} \right)^{m} \int_t^1 \frac{f(s)}{s} \left(\log \frac{s}{t}\right)^{m-1}\,ds\right\|_{ Y(0,1)}
  \leq C_1\left\|f\right\|_{X(0,1)}
\end{equation}
holds for some constant $C_1$, and for every nonnegative $f\in
X(0,1)$.

\textup{(ii)}
The embedding
\begin{equation}\label{E:embedding}
  V^mX(\rn,\mufn) \to Y(\rn,\mufn)
\end{equation}
holds.

\textup{(iii)}
 The Poincar\'e inequality
\begin{equation}\label{E:sobolev}
  \left\|u \right\|_{Y(\rn,\mufn)}
  \leq C_2\left\|\nabla^m u\right\|_{X(\rn,\mufn)}
\end{equation}
holds for some constant $C_2$ and every
$u\in V^m_\bot X(\rn,\mufn)$.
\end{thm}

Let us notice that inequality \eqref{E:hardy} is not
just a specialization of \eqref{E:kernel-cond}, but even a further
simplification of such specialization.

Let
 $\| \cdot \|_{X(0,1)}$ be a rearrangement-invariant function norm,
 and let
$n, m  \in\N$. The  rearrangement-invariant function norm $\|\cdot
\|_{X_{m,\Phi}(0,1)}$ which yields the optimal
rearrangement-invariant target space $Y(\rn,\mufn)$ in embedding
\eqref{E:embedding} is defined as follows. Consider the
rearrangement-invariant function norm
$\|\cdot\|_{\widetilde X_m(0,1)}$ whose associate
norm fulfils
$$
\|g\|_{\widetilde X'_m(0,1)}=
\left\|\frac{1}{s} \int_0^s \left(\log \frac sr\right)^{m-1}
g^*(r)\,dr\right\|_{X'(0,1)}
$$
for  $g \in \M _+ (0,1)$. Then $\|\cdot
\|_{X_{m,\Phi}(0,1)}$ is given by
\begin{equation}\label{E:optimal_range}
\|f\|_{X_{m,\Phi}(0,1)} = \left\|\left(\frac{\log
\frac 2s}{\Phi^{-1}(\log \frac 2s)}\right)^m
f^*(s)\right\|_{\widetilde X_m(0,1)}
\end{equation}
for  $f \in \M _+ (0,1)$.

\begin{remark}\label{R:note}
Note that if $\Phi(t)=t$,   and $m\in\N$, we have that
\[
\widetilde X_m(0,1)=X_{m,\Phi}(0,1)
\]
for every r.i.~norm $\|\cdot\|_{X(0,1)}$.
\end{remark}

\begin{thm}\label{T:optimal_range}{\bf [Optimal target for product
probability spaces]} Let $n$,  $m$,  $\mu_{\Phi,n}$ and  $\| \cdot
\|_{X(0,1)}$ be as in Theorem \ref{T:prob-reduction}.
Then the functional $\|\cdot
\|_{X_{m,\Phi}(0,1)}$, given by \eqref{E:optimal_range},  is
a~rearrangement-invariant function norm satisfying
\begin{equation}\label{E:optimal_range_m}
 V^mX(\rn,\mufn) \to X_{m,\Phi}(\Rn,\mufn),
\end{equation}
and there exists a constant $C$ such that
\begin{equation}\label{E:optimal_poinc_m}
  \left\|u \right\|_{X_{m,\Phi}(\Rn,\mufn)}
  \leq C \left\|\nabla^m u\right\|_{X(\rn,\mufn)},
\end{equation}
for every $u\in
V^m_\bot X(\rn,\mufn)$.
\par\noindent
 Moreover, the function norm
$\|\cdot\|_{X_{m,\Phi}(0,1)}$ is  optimal in \eqref{E:optimal_range_m} and in
\eqref{E:optimal_poinc_m} among all rearrangement-invariant norms.
\end{thm}

\begin{remark}\label{n}
{ \rm   Let us emphasize that inequality \eqref{E:hardy} implies
embedding \eqref{E:embedding} with a norm independent of $n$, and
the Poincar\'e inequality \eqref{E:sobolev} with constant $C_2$
independent of $n$. The norm of the optimal embedding
\eqref{E:optimal_range_m}, and the constant $C$ in the corresponding
Poincar\'e inequality \eqref{E:optimal_poinc_m} are independent of
$n$ as well.}
\end{remark}

For a broad class of rearrangement-invariant
function norms $\|\cdot \|_{X(0,1)}$ the expression of the
associated optimal Sobolev target norm $\|\cdot
\|_{X_{m,\Phi}(0,1)}$ can be substantially simplified, as observed
in the next proposition.

\begin{prop}\label{P:prop3}
Let $m\in \mathbb N$ and let $\Phi$ be as
in~\eqref{E:measure}. Suppose that  $\|\cdot\|_{X(0,1)}$ is a
rearrangement-invariant function norm such that the operator
$$
f\mapsto f^{**}
$$
is bounded on $X'(0,1)$. Then
$$
\|f\|_{X_{m,\Phi}(0,1)} \approx \left\|\left(\frac{\log \frac 2s}{\Phi^{-1}(\log \frac 2s)}\right)^m f^*(s)\right\|_{X(0,1)}
$$
up to multiplicative constants independent of $f\in \mathcal M_+(0,1)$.
\end{prop}

The rearrangement-invariant spaces on which the
operator \lq\lq ** " is bounded are fully characterized in terms of
their upper Boyd index. In particular, the assumptions of
Proposition \ref{P:prop3} are satisfied if and only if the upper
Boyd index of $X'(0,1)$ is strictly smaller that $1$ \cite[Theorem
5.15]{BS}.

The   iteration principle for Sobolev embeddings on product
probability measure spaces, on which Theorem \ref{T:prob-reduction}
rests, reads as follows.

\begin{thm}\label{T:reiteration}{\bf [Iteration principle for product probability
spaces]}
 Let $n$,  $\mu_{\Phi,n}$ and  $\| \cdot
\|_{X(0,1)}$ be as in Theorem \ref{T:prob-reduction}, and let
 $k, h \in\N$. Then,
\[
\big(X_{k,\Phi}\big)_{h,\Phi}(\Rn,\mufn)= X_{k+h,\Phi}(\Rn,\mufn)\,,
\]
up to equivalent norms.
\end{thm}

Specialization of Theorems \ref{T:prob-reduction},
\ref{T:optimal_range} and~\ref{T:reiteration} to the case of
\eqref{phigauss} easily leads to the following results for Gaussian
Sobolev embeddings of any order.

\begin{thm}\label{reductiongauss}{\bf [Reduction principle in Gauss
space]} Let $n \in \N$, $m \in\N$, and let $\| \cdot \|_{X(0,1)}$ and $\|
\cdot \|_{Y(0,1)}$ be rearrangement-invariant function norms. Then
the following facts are equivalent.

\textup{(i)} The inequality
\[
  \left\|\frac 1{\big( \log \frac 2s \big)^{\frac {m}2}}\int_{s}^1
\frac{f(r)}{r}\left(\log \frac rs\right)^{m-1}
  \,dr\right\|_{ Y(0,1)}
  \leq C_1\left\|f\right\|_{X(0,1)}
\]
holds for some constant $C_1$, and for every nonnegative $f\in
X(0,1)$.

\textup{(ii)}   The embedding
\[
  V^mX(\rn,\gamma_n) \to Y(\rn,\gamma_n)
\]
holds.

\textup{(iii)}  The Poincar\'e inequality
\[
  \left\|u \right\|_{Y(\rn,{\gamma_n})}
  \leq C_2\left\|\nabla ^m u\right\|_{X(\rn,{\gamma_n})}
\]
holds for some constant $C_2$,
and for every
$u\in V^m_\bot X(\rn,{\gamma_n})$.
\end{thm}

Given $n, m  \in\N$, and a rearrangement-invariant function norm $\|
\cdot \|_{X(0,1)}$,  define the rearrangement-invariant function norm $\|\cdot \|_{X_{m, G}(0,1)}$ by
\begin{equation}\label{opnorm}
\|f\|_{X_{m,G}(0,1)} =
\left\|\left(\log \frac 2s\right)^\frac m2
f^*(s)\right\|_{\widetilde X_m(0,1)}
\end{equation}
for $f\in\M_+(0,1)$.

\begin{thm}\label{optimalgauss}{\bf [Optimal target in Gauss space]}
Let $n \in \N$, $m \in\N$, and let $\| \cdot \|_{X(0,1)}$  be a
rearrangement-invariant function norm. Then the functional $\|\cdot
\|_{X_{m, G}(0,1)}$, given by \eqref{opnorm},  is
a~rearrangement-invariant function norm satisfying
\begin{equation}\label{opnormm}
V^mX(\rn,{\gamma_n}) \to X_{m,G}(\Rn,{\gamma_n})
\end{equation}
and
\begin{equation}\label{oppoinc}
  \left\|u \right\|_{X_{m,G}(\Rn,{\gamma_n})}
  \leq C\left\|\nabla ^m u\right\|_{X(\rn,{\gamma_n})}
\end{equation}
 for some constant $C$ and every
$u\in V^m_\bot X(\rn,{\gamma_n})$.
\par\noindent
 Moreover, the function norm $\|\cdot\|_{X_{m,G}(0,1)}$ is  optimal in
\eqref{opnormm} and \eqref{oppoinc} among all
rearrangement-invariant norms.
\end{thm}

Observe that, even for $m=1$, Theorems~\ref{reductiongauss} and~\ref{optimalgauss} provide us with a
characterization of Gaussian Sobolev embeddings which somewhat simplifies
 earlier results in a similar direction
\cite{CP_gauss, MM}.

\begin{thm}\label{iterationgauss}{\bf [Iteration principle in Gauss space]}
Let $n, k, h \in\N$, and let $\| \cdot \|_{X(0,1)}$  be a
rearrangement-invariant function norm. Then,
\[
\big(X_{k,G}\big)_{h,G}(\Rn,{\gamma_n})=
X_{k+h,G}(\Rn,{\gamma_n})\,,
\]
up to equivalent norms.
\end{thm}

Of course, versions of Theorems
\ref{reductiongauss}--\ref{iterationgauss}, with the Gauss measure
replaced with the Boltzmann measure, given by the choice
\eqref{E:def-boltzmann},  can similarly be deduced from Theorems
\ref{T:prob-reduction}, \ref{T:optimal_range}
and~\ref{T:reiteration}. The reduction principle and  the optimal target space then
take the following form.

\begin{thm}\label{T:reduction_power}{\bf [Reduction principle in
Boltzmann spaces]} Assume that $n, m \in\N$,  and   $\beta
\in [1, 2]$.  Let $\| \cdot \|_{X(0,1)}$ and $\| \cdot \|_{Y(0,1)}$
be rearrangement-invariant function norms. Then the following facts
are equivalent.

\textup{(i)} The inequality
\[
  \left\|\frac 1{\big( \log \frac 2s \big)^{\frac {m(\beta -1)}\beta}}\int_{s}^1
\frac{f(r)}{r}\left(\log \frac rs\right)^{m-1}
  \,dr\right\|_{ Y(0,1)}
  \leq C_1\left\|f\right\|_{X(0,1)}
\]
holds for some constant $C_1$, and for every nonnegative $f\in X(0,1)$.

\textup{(ii)}
  The embedding
\[
  V^mX(\rn, \gamma_{n, \beta}) \to Y(\rn, \gamma_{n, \beta })
\]
holds.

\textup{(iii)} The Poincar\'e inequality
\[
  \left\|u \right\|_{Y(\rn, \gamma_{n, \beta })}
  \leq C_2\left\|\nabla ^m u\right\|_{X(\rn, \gamma_{n, \beta })}
\]
holds for some constant $C_2$ and  for every
$u\in V^m_\bot X(\rn, \gamma_{n, \beta })$.
\end{thm}

Given $n, m\in\N$, $\beta\in[1,2]$, and a rearrangement-invariant function norm $\|
\cdot \|_{X(0,1)}$,  define the rearrangement-invariant function norm $\|\cdot \|_{X_{m, B,\beta}(0,1)}$ by
\begin{equation}\label{opnorm-B}
\|f\|_{X_{m,B,\beta}(0,1)}
  = \left\|\left(\log \frac 2s\right)^\frac{m(\beta-1)}{\beta} f^*(s)\right\|_{\widetilde X_m(0,1)}
\end{equation}
for $f\in\M_+(0,1)$.

\begin{thm}\label{optimalboltz}{\bf [Optimal target in Boltzmann spaces]}
Let $n, m \in\N$, and let $\| \cdot \|_{X(0,1)}$  be a
rearrangement-invariant function norm. Then the functional $\|\cdot
\|_{X_{m, B,\beta}(0,1)}$, given by \eqref{opnorm-B},  is
a~rearrangement-invariant function norm satisfying
\begin{equation}\label{opnormm-B}
V^mX(\rn,{\gamma_{n,\beta}}) \to X_{m, B,\beta}(\Rn,{\gamma_{n,\beta}})
\end{equation}
and
\begin{equation}\label{oppoinc-B}
  \left\|u \right\|_{X_{m, B,\beta}(\Rn,{\gamma_{n,\beta}})}
  \leq C\left\|\nabla ^m u\right\|_{X(\rn,{\gamma_{n,\beta}})}
\end{equation}
for some constant $C$ and every
$u\in V^m_\bot X(\rn,{\gamma_{n,\beta}})$.

 Moreover, the function norm $\|\cdot\|_{X_{m, B,\beta}(0,1)}$ is  optimal in
\eqref{opnormm-B} and \eqref{oppoinc-B} among all
rearrangement-invariant norms.
\end{thm}

We  present an application of the results of this  section to the
particular case when $\mu _{\Phi , n}$ is a Boltzmann measure, and
the norms are of Lorentz--Zygmund type.

\begin{thm}\label{EX:expo_lz}
Let $n, m \in\N$, let $\beta \in [1, 2]$ and let $p,q\in[1,\infty]$

and $\alpha\in\R$ be such that
 one of the conditions in~\eqref{E:lz_bfs} is satisfied. Then
\[
V\sp mL\sp{p,q;\alpha}(\Rn,\gamma_{n,\beta})\to
\begin{cases}
L\sp{p,q;\alpha+\frac{m(\beta-1)}{\beta}}(\Rn,\gamma_{n,\beta})\quad&\textup{if}\quad
p<\infty;\\
L\sp{\infty,q;\alpha-\frac{m}{\beta}}(\Rn,\gamma_{n,\beta})\quad&\textup{if}\quad
p=\infty.
\end{cases}
\]
Moreover, in both cases, the target space is optimal among all
rearrangement-invariant spaces.
\end{thm}

When  $\beta =2$,   Theorem \ref{EX:expo_lz}
yields the following sharp Sobolev type embeddings in Gauss space.

\begin{thm}\label{EX:expo_lzgauss}
Let $n, m \in\N$,  and let $p,q\in[1,\infty]$ and $\alpha\in\R$ be
such that
 one of the conditions in~\eqref{E:lz_bfs} is satisfied. Then
\begin{equation*}
V\sp mL\sp{p,q;\alpha}(\Rn,\gamma_{n})\to
\begin{cases}
L\sp{p,q;\alpha+\frac{m}{2}}(\Rn,\gamma_{n})\quad&\textup{if}\quad p<\infty;\\
L\sp{\infty,q;\alpha-\frac{m}{2}}(\Rn,\gamma_{n})\quad&\textup{if}\quad
p=\infty.
\end{cases}
\end{equation*}
Moreover, in both cases, the target space is optimal among all
rearrangement-invariant spaces.
\end{thm}

A further specialization of the indices $p, q,
\alpha$ appearing in Theorem \ref{EX:expo_lzgauss} leads to the
following basic embeddings.

\begin{corollary}\label{EX:expo_lzgauss-bis}
Let $n, m \in\N$.

\textup{(i)} Assume that  $p \in [1, \infty)$. Then
$$V\sp mL\sp{p}(\Rn,\gamma_{n})\to L\sp{p}(\log L)^{\frac
{mp}2}(\Rn,\gamma_{n}),$$ and the target space is optimal among all
rearrangement-invariant spaces.

\textup{(ii)} Assume that  $\gamma >0$. Then
$$V\sp m\exp L^\gamma(\Rn,\gamma_{n})\to \exp L^{\frac {2\gamma}{2 + m \gamma}}(\Rn,\gamma_{n}),$$
and the target space is optimal among all rearrangement-invariant
spaces.

\textup{(iii)}
$$V\sp m L^\infty (\Rn,\gamma_{n})\to \exp L^{\frac {2}{ m}}(\Rn,\gamma_{n}),$$
and the target space is optimal among all rearrangement-invariant
spaces.
\end{corollary}

Note that the target space in the second embedding of Theorem
\ref{EX:expo_lzgauss}, and in the embeddings (ii) and (iii) of
Corollary \ref{EX:expo_lzgauss-bis} increases in $m$. This is caused
by the fact that,  $V\sp mL\sp{\infty,q;\alpha}(\Rn,\gamma_{n})
\nsubseteq V\sp kL\sp{\infty,q;\alpha}(\Rn,\gamma_{n})$ if $m >k$.

\section{Optimal target function norms}\label{onedim}

In this section we collect some  basic properties
about certain one-dimensional operators playing a role in the  proofs
of our main results.

Let $T: \Mpl(0,1) \to \Mpl(0,1) $ be a~sublinear operator, namely an
operator such that
$$T(\lambda f) = \lambda  Tf, \quad \hbox{and} \quad T(f+ g) \leq C
(Tf + Tg),$$ for some positive constant $C$, and for every
$\lambda \geq 0$ and $f, g \in \Mpl(0,1)$.

Given two rearrangement-invariant spaces $X(0,1)$ and $Y(0,1)$, we
say that $T$ is bounded from $X(0,1)$ into $Y(0,1)$, and write
\begin{equation}\label{Tembed}
T:X(0,1)\to Y(0,1),
\end{equation}
if the quantity
$$
\|T\|=\sup\left\{\|Tf\|_{Y(0,1)};\ f\in X(0,1)\cap\Mpl(0,1),\
\|f\|_{X(0,1)}\leq 1\right\}
$$
is finite. Such a quantity will be called the norm of $T$. The space
$Y(0,1)$  will be called optimal, within a certain class, in
\eqref{Tembed} if, whenever $Z(0,1)$ is another
rearrangement-invariant space, from the same class, such that
$T:X(0,1)\to Z(0,1)$, we have that $Y(0,1)\to Z(0,1)$. Equivalently,
the function norm $\| \cdot \|_{Y(0,1)}$ will be said to be optimal
in \eqref{Tembed} in the relevant class.

Two operators $T$ and $T'$ from $\Mpl(0,1)$ into
$\Mpl(0,1)$ will be called mutually associate if
\[
\int_0\sp1 Tf(s)g(s)\,ds=\int_0\sp 1f(s)T'g(s)\,ds
\]
 for every $f,g\in\M_+(0,1)$.

\begin{lemma}\label{R:nova}
Let $T$ and $T'$  be mutually associate operators, and let $X(0,1)$
and $Y(0,1)$ be rearrangement-invariant spaces. Then,
$$
T:X(0,1)\to Y(0,1)\quad \hbox{if and only if} \quad T':Y'(0,1)\to
X'(0,1),
$$
and
$$
\|T\|=\|T'\|.
$$
\end{lemma}
\begin{proof} The conclusion is a consequence of the following chain:
\begin{align*}
\|T\| &= \sup _{ \overset{f\geq 0}{\|f\|_{X(0,1)}\leq 1}
}\|Tf\|_{Y(0,1)}
= \sup _{\overset{f\geq 0}{\|f\|_{X(0,1)}\leq 1}}\sup _{\overset{g\geq 0}{\|g\|_{Y'(0,1)}\leq 1}}\int_0\sp1Tf(s)g(s)\,ds\\
\nonumber & = \sup _{\overset{g\geq 0}{\|g\|_{Y'(0,1)}\leq 1}}\sup
_{\overset{f\geq 0}{\|f\|_{X(0,1)}\leq 1}}\int_0\sp 1f(s)T'g(s)\,ds
 =
 \sup _{\overset{g\geq 0}{\|g\|_{Y'(0,1)}\leq 1}}\|T'g\|_{X'(0,1)} = \|T'\|.
\end{align*}
 \end{proof}

Let $I:[0,1]\to[0,\infty)$ be a measurable function
satisfying~\eqref{E:lower-bound}. We define the operators $H_I$ and $R_I$ from $\Mpl(0,1)$ into $\Mpl(0,1)$ by
\begin{equation}\label{Hbis}
H_If(t) =\int_t\sp1\frac{f(s)}{I(s)}\,ds \qquad \hbox{ for
$t\in(0,1]$,}
\end{equation}
and
\[
R_If(t) =\frac1{I(t)}\int_0\sp t{f(s)}\,ds \qquad \hbox{ for
$t\in(0,1]$,}
\]
for $f\in\M_+(0,1)$.
 Moreover, given $j\in\N$, we set
\begin{equation}\label{Hj}
H_I\sp j=\underbrace{H_I\circ H_I\circ \dots\circ H_I}_{
j-\textup{times}} \quad \hbox{and} \quad R_I\sp
j=\underbrace{R_I\circ R_I\circ \dots\circ R_I}_{ j-\textup{times}}.
\end{equation}
 We also set
$H_I\sp 0=R_I\sp0=\Id$.

\begin{remarks}\label{R:hrq}
(i) The operators $H_I$ and $R_I$ are mutually associate.
Hence, $H_I^j$ and $R_I^j$ are also mutually associate for $j \in
\N$.

(ii) By the Hardy--Littlewood inequality~\eqref{E:HL}, we have, for
every$f\in\Mpl(0,1)$,
\[
R_If(t)\leq\ri f\sp*(t) \qquad \hbox{ for $t\in(0,1]$.}
\]
More generally, for every $f\in\Mpl(0,1)$ and $j\in\N$, one has that
\begin{equation}\label{E:riij}
R_I\sp jf(t)\leq\ri\sp j f\sp*(t) \qquad \hbox{ for $t\in(0,1]$.}
\end{equation}

(iii) For every $j\in\N$ and $f\in\Mpl(0,1)$, we have that
\begin{equation}\label{E:kernel-formula}
H_I\sp j
f(t)=\frac1{(j-1)!}\int_t\sp1\frac{f(s)}{I(s)}\left(\int_t\sp
s\frac{dr}{I(r)}\right)\sp{j-1}\,ds  \quad \hbox{ for $t\in(0,1)$.}
\end{equation}
Equation \eqref{E:kernel-formula} holds for $j=1$ by the very
definition of $H_I$. On the other hand, if \eqref{E:kernel-formula}
is assumed to hold for some $j\in\N$, then
\begin{align*}
H_I\sp{j+1}f(t)=\int_t\sp1\frac {H_I\sp jf(s)}{I(s)}\,ds &=
\frac1{(j-1)!} \int_t\sp1\frac {1}{I(s)}\int_s\sp1\frac
{f(r)}{I(r)}\left(\int_s\sp
r\frac{d\tau}{I(\tau)}\right)\sp{j-1}\,dr\,ds\\
&= \frac1{(j-1)!} \int_t\sp1\frac {f(r)}{I(r)}\int_t\sp r\frac
{1}{I(s)}\left(\int_s\sp
r\frac{d\tau}{I(\tau)}\right)\sp{j-1}\,ds\,dr\\
&=\frac1{j!}\int_t\sp1\frac {f(r)}{I(r)}\left(\int_t\sp
r\frac{d\tau}{I(\tau)}\right)\sp{j}\,dr.
\end{align*}
Hence, \eqref{E:kernel-formula} follows by induction. Similarly, for
every $j\in\N$ and $f\in\Mpl(0,1)$, we also have that
\begin{equation}\label{E:rij}
R\sp j_If(t)=\frac1{(j-1)!}\frac{1}{I(t)}\int _0^t
f(s)\left(\int_s\sp t\frac{dr}{I(r)}\right)\sp{j-1}\,ds \quad \hbox{
for $t\in(0,1]$.}
\end{equation}
\end{remarks}

Given any $j\in\N$ and any rearrangement-invariant function norm $\|\cdot \|_{X(0,1)}$, equation \eqref{E:rij} implies that
\begin{equation}\label{E:xji}
\|f\|_{X_{j,I}'(0,1)}= {(j-1)!}\|R\sp{j}_If\sp* \|_{X'(0,1)},
\end{equation}
for  $f \in \Mpl(0,1)$, where $\|\cdot\|_{X_{j,I}'(0,1)}$ is the functional introduced in~\eqref{E:eucl_opt_norm}. We also formally set
$\|\cdot\|_{X_{0,I}'}=\| \cdot \|_{X'(0,1)}$.

\begin{prop}\label{P:norm}
Let $I:[0,1]\to[0,\infty)$ be a~measurable function
satisfying~\eqref{E:lower-bound}. Let $\|\cdot\|_{X(0,1)}$ be a~rearrangement-invariant~function norm and let $j\in\N$. Then the functional
$\|\cdot\|_{X_{j,I}'(0,1)}$ defined in~\eqref{E:xji} is a
rearrangement-invariant function norm, whose associate norm
$\|\cdot\|_{X_{j,I}(0,1)}$ fulfils
\begin{equation}\label{E:bdj}
H_I\sp j:X(0,1)\to X_{j,I}(0,1).
\end{equation}
Moreover, the space $X_{j,I}(0,1)$ is
the optimal target  in~\eqref{E:bdj} among all
rearrangement-invariant spaces.
\end{prop}

\begin{proof}
We begin by  showing that the functional $\|\cdot\|_{X_{j,I}'(0,1)}$
is a rearrangement-invariant function norm.
Let $f,g\in\Mpl (0,1)$. By~\eqref{subadd}, $\int_0\sp
t(f+g)\sp*(s)\,ds\leq \int_0\sp tf\sp*(s)\,ds+\int_0\sp
tg\sp*(s)\,ds$ for $t \in (0,1)$.  Thus, by Hardy's lemma (see
Section~\ref{S:measurable}) applied, for each fixed $t \in (0,1)$,
with $f_1(s)=(f+g)\sp*(s)$, $f_2(s)=f\sp*(s)+g\sp*(s)$ and
$h(s)=\chi _{(0,t)}(s)\big(\int_s\sp t\frac{dr}{I(r)}\big)\sp{j-1}$,
we obtain the triangle inequality
$$
\|f+g\|_{X_{j,I}'(0,1)}\leq \|f\|_{X_{j,I}'(0,1)}+\|g\|_{X_{j,I}'(0,1)}.
$$
Other properties in the axiom (P1) of the definition of
rearrangement-invariant function norm, as well as the axioms (P2),
(P3) and (P6) are obviously satisfied. Next, it follows
from~\eqref{E:lower-bound} that there exists a~positive constant~$C$
such that $ \tfrac1{I(t)}\leq \tfrac Ct$ for $t\in(0,1)$. Therefore,
\begin{align*}
\|1\|_{X_{j,I}'(0,1)} &= \left\|\frac1{I(t)}\int_0\sp
t\left(\int_s\sp
t\frac{dr}{I(r)}\right)\sp{j-1}\,ds\right\|_{X'(0,1)}\leq C\sp j
\left\|\frac1{t}\int_0\sp t\left(\int_s\sp
t\frac{dr}{r}\right)\sp{j-1}\,ds\right\|_{X'(0,1)}\\
&=C\sp j \left\|\frac1{t}\int_0\sp t\left(\log \frac
ts\right)\sp{j-1}\,ds\right\|_{X'(0,1)}= (j-1)!
C\sp j\|1\|_{X'(0,1)},
\end{align*}
and (P4) follows. As far as (P5) is concerned, note that
$$
\int_0\sp{1} f \sp*(s)\,ds\leq 2\int_0\sp{\frac12} f \sp*(s)\,ds
$$
for every $f\in\Mpl (0,1)$. Thus, by (P5) for the norm $\|\cdot
\|_{X'(0,1)}$, there exists a positive constant $C$ such that, if
$f\in\Mpl (0,1)$, then
\begin{align*}
&\left\|\frac1{I(t)}\int_0\sp tf\sp*(s)\left(\int_s\sp
t\frac{dr}{I(r)}\right)\sp{j-1}\,ds\right\|_{X'(0,1)}  \geq C
\int_0^1 \frac1{I(t)}\int_0\sp tf\sp*(s)\left(\int_s\sp
t\frac{dr}{I(r)}\right)\sp{j-1}\,ds \,dt \\
&  = \frac Cj \int _0^1 f\sp*(s) \left(\int_s\sp
1\frac{dr}{I(r)}\right)\sp{j}\,ds \geq \frac Cj
\left(\int_{\frac12}\sp 1\frac{dr}{I(r)}\right)\sp{j}\int _0^{\frac
12}f\sp*(s)\,ds \geq C'\|f\|_{L\sp1(0,1)}
\end{align*}
where $C'=\frac C{2j}(\int_{\frac12}\sp 1\frac{dr}{I(r)})\sp{j}$.
Hence, property (P5) follows.

To prove~\eqref{E:bdj}, note that, by~\eqref{E:riij} and~\eqref{E:xji}, we have
\[
\|R\sp j_If\|_{X'(0,1)}\leq \|R\sp
j_If\sp*\|_{X'(0,1)}=\frac{1}{(j-1)!}\|f\|_{X'_{j,I}(0,1)}
\]
for $f \in \M_+(0,1)$. Hence,
\[
R\sp j_I:X'_{j,I}(0,1)\hra X'(0,1).
\]
Since $R\sp j_I$ and $H\sp j_I$ are
mutually associate,  equation \eqref{E:bdj} follows via   Lemma \ref{R:nova}.

It remains to prove that $X_{j,I}(0,1)$ is  optimal in \eqref{E:bdj}
among all rearrangement-invariant spaces. To this purpose, assume
that  $Y(0,1)$ is  another rearrangement-invariant space such that
$H\sp j_I:X(0,1)\to Y(0,1)$. Then, by Lemma \ref{R:nova} again,
$R\sp j_I: Y'(0,1)\to X'(0,1)$, namely
$$
\|R\sp j_I f\|_{X'(0,1)}\leq C\|f\|_{Y'(0,1)}
$$
for some positive constant $C$, and every $ f\in\Mpl(0,1)$. Thus, in
particular, by \eqref{E:xji},
$$
\|f \|_{X_{j,I}'(0,1)}=
(j-1)!\|R\sp j_If\sp*\|_{X'(0,1)}\leq
(j-1)!C\|f\sp*\|_{Y'(0,1)}=(j-1)!C\|f\|_{Y'(0,1)}
$$
for every $f\in\Mpl (0,1)$. Hence, $Y'(0,1)\hra X_{j,I}'(0,1)$, and, equivalently,
$X_{j,I}(0,1)\hra Y(0,1)$. This shows that $X_{j,I}(0,1)$ is
optimal in \eqref{E:bdj} among all rearrangement-invariant spaces.
\end{proof}

We   introduce one more sequence of function norms, based on the
iteration of the first-order  function norm $\|f\|_{X_{1,I}'(0,1)}$.
Let $I:[0,1]\to[0,\infty)$ be a~measurable function
satisfying~\eqref{E:lower-bound}. Let $\|\cdot\|_{X(0,1)}$ be
a~rearrangement-invariant~function norm. Let $j\in\N\cup\{0\}$. We
define $\|\cdot\|_{X_{j}(0,1)}$   as the rearrangement-invariant
function norm whose associate norm $\|\cdot \|_{X_j'(0,1)}$ is
given, via iteration, by
$\|\cdot\|_{X_{0}'(0,1)}=\|\cdot\|_{X'(0,1)}$, and, for $j \geq 1$,
by
\begin{equation}\label{E:krus}
\|f\|_{X_j'(0,1)}=\|R_If\sp*\|_{X_{j-1}'(0,1)}
\end{equation}
for $f\in\M_+(0,1)$. Note that
\begin{equation}\label{E:eq-one}
\|f \|_{X_{1}(0,1)}=
\|f\|_{X_{1,I}(0,1)}.
\end{equation}

\begin{remark}\label{R:hrq2}
By Proposition~\ref{P:norm}, applied $j$ times with $j=1$, we obtain
that, for every $j\in\N\cup\{0\}$, the functional
$\|\cdot\|_{X_{j}'(0,1)}$ is actually a rearrangement-invariant~function norm. Moreover, its associate function norm
$\|\cdot\|_{X_{j}(0,1)}$ fulfils
\begin{equation}\label{sept1}
H_I:X_{j}(0,1)\to X_{j+1}(0,1),
\end{equation}
and $\|\cdot\|_{X_{j+1}(0,1)}$ is the optimal target function norm
in \eqref{sept1} among all rearrangement-invariant~function norms.
By Lemma \ref{R:nova}, we also have
$$
R_I:X_{j+1}'(0,1)\to X_{j}'(0,1).
$$
\end{remark}

\begin{remark}\label{R:comparison-of-spaces}
Note that, by the very definition of $X_{j}(0,1)$,
\[
X_j(0,1)=\underbrace{(\dots(X_{1,I})_{1,I}\dots)_{1,I}}_{
j-\textup{times}}(0,1)
\]
for $j \in \N$. In particular,
\begin{equation}\label{E:reit-iter}
(X_k)_h(0,1)=X_{k+h}(0,1)
\end{equation}
for every $k,h\in\N$.
\end{remark}

We now turn our attention to the special situation when $I$
satisfies, in addition, condition \eqref{E:doubling}. In this case,
most of the results take a~simpler form. We start with a result concerned with the equivalence of two couples of
functionals under \eqref{E:doubling}.

\begin{prop}\label{link}
Let $I:[0,1]\to[0,\infty)$ be a~non-decreasing function
satisfying~\eqref{E:doubling} and let $\|\cdot \|_{X(0,1)}$ be any
rearrangement-invariant function norm. Then:

\textup{(i)} For every $j\in\N$, and $f \in \mathcal M_+(0,1)$,
\begin{equation}\label{link1}
\left\|\int_t\sp1\frac{f(s)}{I(s)}\left(\int_t\sp
s\frac{dr}{I(r)}\right)\sp{j-1}\,ds\right\|_{X(0,1)} \approx
\left\|\int_t\sp1f(s)\frac{s^{j-1}}{I(s)^j}\,ds\right\|_{X(0,1)},
\end{equation}
up to multiplicative constants independent of $\|\cdot \|_{X(0,1)}$
and $f$.

\textup{(ii)} For every $j\in\N$, and $f \in \mathcal M_+(0,1)$,
\[
\left\|\frac{1}{I(s)}\int _0^s f(t)\left(\int_t\sp s
\frac{dr}{I(r)}\right)\sp{j-1} dt\right\|_{X(0,1)} \approx
\left\|\frac{s^{j-1}}{I(s)^j} \int _0^s f(t)dt\right\|_{X(0,1)},
\]
up to multiplicative constants independent of $\|\cdot \|_{X(0,1)}$
and $f$.
\end{prop}

\begin{proof} We first note that, owing to the monotonicity of $I$, we have, for every  $j\in\N$,
\[
\left(\frac s{I(s)}\right)\sp{j-1}
=\frac{2\sp{j-1}}{I(s)\sp{j-1}}\left(\int_{\frac
s2}\sp{s}\,dr\right)\sp{j-1} \leq 2\sp{j-1}\left(\int_{\frac
s2}\sp{s}\frac{dr}{I(r)}\right)\sp{j-1} \quad \hbox{for
$s\in(0,1)$.}
\]
Thus,
\begin{align*}
\int_{2t}\sp1& \frac{f(s)}{I(s)}\left(\frac
s{I(s)}\right)\sp{j-1}\,ds
 \leq 2\sp{j-1} \int_{2t}\sp1\frac{f(s)}{I(s)}\left(\int_{\frac
s2}\sp{s}\frac{dr}{I(r)}\right)\sp{j-1}\,ds \\ & \leq 2\sp{j-1}
\int_{2t}\sp1\frac{f(s)}{I(s)}\left(\int_t\sp{s}\frac{dr}{I(r)}\right)\sp{j-1}\,ds
\leq 2\sp{j-1}
\int_{t}\sp1\frac{f(s)}{I(s)}\left(\int_t\sp{s}\frac{dr}{I(r)}\right)\sp{j-1}\,ds
\quad \hbox{for  $t\in(0,\frac12]$.}
\end{align*}
Hence, the right-hand side of \eqref{link1} does not exceed a
constant times its left-hand side, owing to the boundedness of the
dilation operator in rearrangement-invariant spaces. Note that this
inequality holds even without the assumption~\eqref{E:doubling}. On
the other hand,~\eqref{E:doubling} implies
\begin{align*}
\int_t\sp1\frac{f(s)}{I(s)}\left(\int_t\sp
s\frac{dr}{I(r)}\right)\sp{j-1}\,ds & \leq
\int_t\sp1\frac{f(s)}{I(s)}\left(\int_0\sp
s\frac{dr}{I(r)}\right)\sp{j-1}\,ds \\ & \leq C\sp{j-1}
\int_t\sp1\frac{f(s)}{I(s)}\left(\frac{s}{I(s)}\right)\sp{j-1}\,ds
\quad \hbox{for  $t\in(0,1)$,}
\end{align*}
hence the converse inequality in~\eqref{link1} follows. This proves (i).
\par
The proof of (ii) is similar.
\end{proof}

Given
$j\in\N$ and a~rearrangement-invariant function norm
$\|\cdot\|_{X(0,1)}$, we define the functional
$\|\cdot\|_{(X_{j,I}^\sharp)'(0,1)}$ by
\begin{equation}\label{E:xji-conv}
\|f\|_{(X_{j,I}^\sharp)'(0,1)}=\left\|\frac{t\sp{j-1}}{I(t)^j}\int _0^t f\sp*(s)\,ds\right\|_{X'(0,1)}
\end{equation}
for $f \in \Mpl (0,1)$.

\begin{remark}\label{R:jn}
It follows from Proposition~\ref{link} and its proof that for every rearrangement-invariant
norm $\|\cdot\|_{X(0,1)}$ and every $j\in\N$, we have
\[
(X_{j,I}^\sharp)'(0,1)\hra X_{j,I}'(0,1),
\]
and if moreover~\eqref{E:doubling} is satisfied, then, in fact,
\[
(X_{j,I}^\sharp)'(0,1)= X_{j,I}'(0,1).
\]
\end{remark}

This observation has a straightforward consequence.

\begin{prop}\label{linkcor}
Let $I:[0,1]\to[0,\infty)$ be a~non-decreasing function
satisfying~\eqref{E:doubling} and let $\|\cdot \|_{X(0,1)}$ be any
rearrangement-invariant function norm. Then
$$X_{j,I}(0,1) = X_{j,I}^\sharp (0,1),$$
up to equivalent norms.
\end{prop}

The following result is a counterpart of Proposition \ref{P:norm}
under \eqref{E:doubling}. It
 follows  from Proposition \ref{P:norm}, with $j=1$ and     $I$ replaced with the function
 $(0,1) \ni t \mapsto \frac{I(t)\sp{j}}{t\sp{j-1}}$, which obviously
 satisfies~\eqref{E:lower-bound}.

\begin{prop}\label{P:norm-conv}
Let $I:[0,1]\to[0,\infty)$ be a~non-decreasing function
satisfying~~\eqref{E:doubling}. Let $X(0,1)$ be a  rearrangement
invariant~space and let $j\in\N$. Then the functional
$\|\cdot\|_{(X_{j,I}^\sharp)'(0,1)}$ defined in~\eqref{E:xji-conv}
is a rearrangement-invariant function norm. Moreover,
\[
H\sp j_I:X(0,1)\to  X_{j,I}^\sharp (0,1),
\]
and $X_{j,I}^\sharp(0,1)$ is  optimal in~\eqref{E:bdj} among all
rearrangement-invariant spaces.
\end{prop}

\section{Proofs of the main results}\label{S:proof-convergent}

Here we are concerned with the  proof of the results of
Section~\ref{S:main}. In what follows, $R_I^m$ denotes the operator
defined as in \eqref{Hj}.

\begin{lemma}\label{T:lemma}
Let $I:[0,1] \rightarrow [0,\infty)$ be a non-decreasing function
fulfilling~\eqref{E:lower-bound}, and let $m\in \mathbb N \cup
\{0\}$. Then, for every $f\in \M_+(0,1)$,
\begin{equation}\label{E:1-operator-R}
R_I^m f^*(t) \leq 2^m R_I^m f^*(s)  \quad \hbox{if $0<\tfrac t2 \leq
s \leq t \leq 1$.}
\end{equation}
Consequently, for every $f\in \M_+(0,1)$,
\begin{equation}\label{E:1bis}
(d-c) R^m_I f^*(d)
\leq 2^{m+1} \int_c^d R^m_I f^*(s)\,ds \quad \hbox{if $0\leq c
<d\leq 1$.}
\end{equation}
\end{lemma}

\begin{proof}
We prove inequality \eqref{E:1-operator-R}  by induction. Fix any $f \in
\M_+(0,1)$. If $m=0$ then~\eqref{E:1-operator-R} is satisfied thanks to the
monotonicity of $f^*$. Next, let $m\geq 1$, and assume that
\eqref{E:1-operator-R} is fulfilled with $m$ replaced with $m-1$. If $0<\tfrac
t2 \leq s \leq t \leq 1$, then
\begin{align*}
R^m_I f^*(t)= &\frac{1}{I(t)} \int_0^t R^{m-1}_I f^*(r)\,dr \leq \frac{2^{m-1}}{I(s)} \int_0^t R^{m-1}_I f^*\left(\frac{r}{2}\right) \,dr
= \frac{2^m}{I(s)} \int_0^{\frac{t}{2}} R^{m-1}_I f^*(r)\,dr\\
&\leq \frac{2^m}{I(s)} \int_0^{s} R^{m-1}_I f^*(r)\,dr = 2^m R^m_I f^*(s),
\end{align*}
where the first inequality holds according to the induction
assumption and to the fact that $I$ is non-decreasing on $[0,1]$.
Inequality \eqref{E:1-operator-R} follows.

Now, let $0\leq c <d \leq 1$,   $m\in \mathbb N \cup \{0\}$ and
$f\in \M_+(0,1)$. Thanks to~\eqref{E:1-operator-R},
\begin{align*}
\int_c^d R^m_I f^*(d) \,ds = 2 \int_{\frac{c+d}{2}}^d R^m_I f^*(d)\,ds \leq 2^{m+1} \int_{\frac{c+d}{2}}^d R^m_I f^*(s)\,ds \leq 2^{m+1} \int_{c}^d R^m_I f^*(s)\,ds.
\end{align*}
This proves \eqref{E:1bis}.
\end{proof}

Given $m\in \mathbb N$ and a non-decreasing  function
$I:[0,1]\rightarrow [0,\infty)$ fulfilling~\eqref{E:lower-bound}, we
define the operator $G\sp m_I$ at every $f\in \M_+(0,1)$ by
\[
G^m_I f(t)=\sup_{t\leq s \leq 1} R^m_{I}f\sp*(s) \quad \hbox{for
$t\in (0,1)$.}
\]
Note that, trivially, $R^m_{I}f\sp* \leq G^m_I f$ for every $f\in
\M_+(0,1)$. Moreover,  $G^m_I f$ is a non-increasing function, and
hence $(R^m_{I}f\sp*)^* \leq G^m_I f$ as well.

The following lemma tells us that the operator $G^m_I$ does not
essentially change if  $I$ is replaced with its left-continuous
representative.

\begin{lemma}\label{T:equal_a_e}
Let $m\in \mathbb N$, let $I:[0,1]\rightarrow [0,\infty)$ be  a
non-decreasing function fulfilling~\eqref{E:lower-bound}, and let
$I_0:[0,1] \rightarrow [0,\infty)$ be the left-continuous function
which agrees with $I$ a.e.\ in $[0,1]$. Then, for every $f\in
\mathcal M_+(0,1)$,
$$
G^m_{I}f=G^m_{I_0}f
$$
up to a countable subset of $(0,1)$.
\end{lemma}

\begin{proof}
Define $M=\{t\in (0,1): I(t)\neq I_0(t)\}$. The set $M$ is at most
countable. We shall prove that, for every $g\in \mathcal M_+(0,1)$,
\begin{equation}\label{E:eq_M}
\sup_{t\leq s \leq 1} \frac{1}{I(s)}\int_0\sp{s}g(r)\,dr =
\sup_{t\leq s \leq 1} \frac{1}{I_0(s)}\int_0\sp{s}g(r)\,dr \quad \hbox{for $t\in (0,1)\setminus
M$.}
\end{equation}
The conclusion will then follow by applying~\eqref{E:eq_M} to the
function $g=R^{m-1}_{I_0}f^*$, and by  the fact that
$\frac{1}{I(s)}\int_0\sp{s}(R^{m-1}_{I}f^*)(r)\,dr=\frac{1}{I(s)}\int_0\sp{s}(R^{m-1}_{I_0}f^*)(r)\,dr$
for $s\in (0,1)$. Fix $g\in \mathcal M_+(0,1)$
and $t\in (0,1)$. Given $s\in (t,1]$, we have that
$$
\frac{1}{I(s)}\int_0\sp{s}g(r)\,dr \leq \Big(\lim_{\tau\to s_{-}}\frac{1}{I(\tau)}\Big)\int_0\sp{s}g(r)\,dr  = \lim_{\tau\to s_{-}} \frac{1}{I(\tau)}\int_0\sp{\tau}g(r)\,dr
\leq \sup_{t<\tau\leq 1} \frac{1}{I(\tau)}\int_0\sp{\tau}g(r)\,dr.
$$
On taking the supremum over all $s\in (t,1]$, we get that
$$
\sup_{t<s\leq 1} \frac{1}{I(s)}\int_0\sp{s}g(r)\,dr \leq \sup_{t<s\leq 1}\Big (\lim_{\tau\to s_{-}}\frac{1}{I(\tau)}\Big)\int_0\sp{s}g(r)\,dr
\leq\sup_{t<\tau\leq 1} \frac{1}{I(\tau)}\int_0\sp{\tau}g(r)\,dr.
$$
Hence, since $I_0(s) = \lim_{\tau\to s_{-}}I(\tau)$ for $s \in (0,1)$,
$$
\sup_{t< s \leq 1} \frac{1}{I(s)}\int_0\sp{s}g(r)\,dr =
\sup_{t< s \leq 1} \frac{1}{I_0(s)}\int_0\sp{s}g(r)\,dr \quad \hbox{
 for  $t\in(0,1)$.}
$$
This yields \eqref{E:eq_M}.
\end{proof}

\begin{prop}\label{T:proposition-extra}
Let $m\in \mathbb N$, let $I:[0,1]\rightarrow [0,\infty)$ be a
left-continuous non-decreasing function
fulfilling~\eqref{E:lower-bound}, and let $f\in \M_+(0,1)$. Define
\begin{equation}\label{E:E}
E=\{t\in (0,1): R\sp m_{I} f^*(t) < G\sp m_If(t)\}.
\end{equation}
Then $E$ is an open subset of $(0,1)$. Hence, there exists an at
most countable  collection $\{(c_k,d_k)\}_{k\in S}$ of pairwise
disjoint open intervals in $(0,1)$ such that
\begin{equation}\label{E:decomp}
E=\cup_{k\in S} (c_k,d_k).
\end{equation}
Moreover,
\begin{equation}\label{E:not_M}
G^m_I f(t) =R^m_{I}f^{*}(t) \quad \hbox{if $t\in (0,1)\setminus E$,}
\end{equation}
and
\begin{equation}\label{E:in_M}
G^m_I f(t) = R^m_{I} f^{*}(d_k)    \quad \hbox{if $t\in (c_k,d_k)$
for some $k\in S$.}
\end{equation}
\end{prop}

\begin{proof}
Fix $t\in(0,1)$. If $G\sp m_If(t)=\infty$, then both  functions $G\sp m_If$ and $R\sp m_If\sp*$  are identically equal to $\infty$, and hence there is nothing to
prove. Assume that $G\sp m_If(t)<\infty$. Then we claim that
$\sup_{t\leq s \leq 1} R^m_{I}f\sp*(s)$ is attained. This follows
from the fact that the function $R^m_{I}f\sp*(s)$ is
upper-semicontinuous, since $I(s) R^m_{I}f\sp*(s)$ is continuous,
and $\frac{1}{I(s)}$ is upper-semicontinuous. Notice that this
latter property holds since $I$ is left-continuous and
non-decreasing, and hence lower-semicontinuous.
\par
Suppose now that $t\in E$. Then, due to the upper-semicontinuity of
$R^m_{I}f\sp*$, there exists $\delta >0$ such that
\begin{equation}\label{E:open}
R^m_{I} f^{*}(r) < G^m_I f(t) \quad \hbox{if   $r\in
(t-\delta,t+\delta)$.}
\end{equation}
Let $c \in [t,1]$ be such that
$
R^m_{I}f^{*}(c) = G^m_I f(t).
$
Then, thanks to~\eqref{E:open}, $c\in [t+\delta,1]$.  It easily
follows that $G^m_I f(t)=G^m_I f(r)$ for every $r\in
(t-\delta,t+\delta)$, a piece of information that, combined
with~\eqref{E:open}, yields   $r\in E$. This shows that $E$ is an
open set. Assertion~\eqref{E:not_M} is trivial and~\eqref{E:in_M}
is an easy consequence of the definition of  $G^m_I f$.
\end{proof}

\begin{prop}\label{T:proposition}
Let $m\in \mathbb N$, let $I:[0,1]\rightarrow [0,\infty)$ be a
left-continuous non-decreasing function
fulfilling~\eqref{E:lower-bound}, and let $f\in \M_+(0,1)$. Then
\begin{equation}\label{E:Gprop}
G^m_I G_I f \approx G^{m+1}_I f,
\end{equation}
up to multiplicative constants depending on $m$.
\end{prop}

\begin{proof} Fix any $f\in \M_+(0,1)$. Since  $R_If\sp*\leq G_If$,   for every $m\in\N$
\begin{equation}\label{E:inequality-for-G-trivial-part}
G_I\sp{m+1}f(t)=\sup_{t\leq s \leq 1} R^m_{I}R_If\sp*(s)
\leq \sup_{t\leq s \leq 1} R^m_{I}G_If\sp*(s)=G^m_I G_I f (t) \quad \hbox{for $t \in (0,1)$.}
\end{equation}
This shows that the right-hand side of \eqref{E:Gprop} does not exceed the left-hand side.
To show a converse inequality, consider the set $E$ defined as in \eqref{E:E}, with $m=1$.
By Proposition~\ref{T:proposition-extra}, the set $E$ is  open. Let
$\{(c_k,d_k)\}_{k\in S}$ be open intervals as in \eqref{E:decomp}.
If $t\in (c_k,d_k)$ for some $k\in S$, then, by \eqref{E:in_M} with
$m=1$,
\begin{equation}\label{nov10}
 \frac{d_k}{I(d_k)}f^{**}(d_k) =
R_{I}f^{*}(d_k) = G_If (t) \geq R_{I}f^{*}(t) \geq
\frac{t}{I(t)}f^{**}(d_k).
\end{equation}
Observe that $f^{**}(d_k)>0$. Indeed, if $f^{**}(d_k)=0$, then
$R_{I}f^{*}(t)=R_{I}f^{*}(d_k)=G_I\sp mf(t)=0$, and hence $t\notin
E$, a contradiction. Thus, we obtain from \eqref{nov10}
\begin{equation}\label{E:b_k}
\frac{d_k}{I(d_k)} \geq \frac{t}{I(t)} \quad \hbox{for $t\in
(c_k,d_k)$.}
\end{equation}

We shall now prove by induction that, given $m\in \mathbb N \cup
\{0\}$, there exists a constant $C=C(m)$  such that
\begin{equation}\label{E:prop}
R_{I}^m G_I f (t)\leq C  \left(R_{I}^{m+1}f^*(t) + \sum_{k\in S}
\chi_{(c_k,d_k)}(t) R^{m+1}_{I} f^*(d_k)\right) \quad \hbox{for $t
\in (0,1)$.}
\end{equation}
Let $m=0$. Then~\eqref{E:prop} holds with $C =1$,  by
\eqref{E:not_M} and~\eqref{E:in_M} (with $m=1$). Next, suppose
that~\eqref{E:prop} holds for some $m\in \mathbb N \cup \{0\}$. Fix
any $t\in (0,1)$. Then
\begin{align*}
R^{m+1}_{I} G_I f (t) &= \frac{1}{I(t)} \int_0^t R^m_{I} G_I
f(r)\,dr
\leq \frac{C}{I(t)} \int_0^t R_{I}^{m+1}f^*(r)\,dr \\
& \quad + \frac{C}{I(t)} \sum_{\{\ell\in S: d_\ell \leq t\}}
\int_{c_\ell}^{d_\ell}  R^{m+1}_{I} f^*(d_\ell)\,dr +
 \frac{C}{I(t)} \sum_{k\in S} \chi_{(c_k,d_k)}(t) \int_{c_k}^t R^{m+1}_{I} f^*(d_k)\,dr\\
&\leq C R_{I}^{m+2}f^*(t) + \frac{2^{m+2}C}{I(t)} \sum_{\{\ell\in S: d_\ell \leq t\}} \int_{c_\ell}^{d_\ell}  R^{m+1}_{I} f^*(r)\,dr\\
& \quad + C \frac{t}{I(t)} \sum_{k\in S} \chi_{(c_k,d_k)}(t) R^{m+1}_{I} f^*(d_k) ~~~ \qquad \qquad \textup{(by \eqref{E:1bis})}
\\ &\leq C R_{I}^{m+2}f^*(t) + \frac{2^{m+2}C}{I(t)} \int_0^t R^{m+1}_{I} f^*(r)\,dr
\\
& \quad + C \sum_{k\in S} \chi_{(c_k,d_k)}(t) \frac{d_k}{I(d_k)} R^{m+1}_{I} f^*(d_k) ~~~~~~ \qquad \qquad \qquad \textup{(by~\eqref{E:b_k})}
\\  &\leq (C+2^{m+2}C) R_{I}^{m+2}f^*(t)
\\
&\quad +C2^{m+2} \sum_{k\in S} \chi_{(c_k,d_k)}(t) \frac{1}{I(d_k)} \int_0^{d_k} R^{m+1}_{I} f^*(r)\,dr ~~~~ \qquad \qquad \textup{(by \eqref{E:1bis})}\\
&=(C+2^{m+2}C) R_{I}^{m+2}f^*(t) + C2^{m+2} \sum_{k\in S} \chi_{(c_k,d_k)}(t) R^{m+2}_{I} f^*(d_k)\\
&\leq C'\left(R_{I}^{m+2}f^*(t) + \sum_{k\in S} \chi_{(c_k,d_k)}(t)
R^{m+2}_{I} f^*(d_k)\right),
\end{align*}
where $C'=C+2^{m+2}C$. This proves~\eqref{E:prop}.
\par\noindent
Owing to \eqref{E:prop}, for every $m \in \mathbb
N$ we have that
$$
G^m_I G_I f(t)=\sup_{t\leq s \leq 1} R^m_{I} G_{I} f(s) \leq 2C
G_{I}^{m+1} f(t) \quad \hbox{for $t \in (0,1)$.}
$$
Combining this inequality with~\eqref{E:inequality-for-G-trivial-part} yields~\eqref{E:Gprop}.
\end{proof}

\begin{thm}\label{T:lenka-main}
Let $I:[0,1]\rightarrow [0,\infty)$ be a non-decreasing function
satisfying~\eqref{E:lower-bound} and let $\|\cdot\|_{X(0,1)}$ be a
rearrangement-invariant function norm. Let $m\in \mathbb
N\cup\{0\}$. Then
\begin{equation}\label{E:down-dual-equivalence}
\|R^{m+1}_I f^*\|_{X'(0,1)} \approx \left\|R^m_I\left(\left(R_I
f^*\right)^*\right)\right\|_{X'(0,1)} \approx \|G^{m+1}_I
f\|_{X'(0,1)} \approx
 \|R^{m+1}_I f^*\|_{X'_d(0,1)}
\end{equation}
for every $f\in \M_+(0,1)$, up to multiplicative constants depending
on $m$.
\end{thm}

\begin{proof} We may assume, without loss of generality, that $I$ is
left-continuous. Indeed, equation \eqref{E:down-dual-equivalence} is
not affected by a replacement of $I$ with its left-continuous
representative, since the latter can differ from $I$ at most on a
countable subset of $[0,1]$, and since Lemma~\ref{T:equal_a_e}
holds.
\par
Fix any $f\in \M_+(0,1)$, and let $m\geq 1$. By  \eqref{E:riij}  and
Proposition~\ref{T:proposition}, there exists a constant $C=C(m)$
such that
$$
R^{m+1}_I f^*(t) \leq R^m_I ((R_If^*)^*)(t) \leq R^m_{I} (G_If)(t)
\leq G^m_I G_If(t) \leq C G^{m+1}_If(t) \quad \hbox{for  $t\in
(0,1)$.}
$$
Hence,
\begin{equation}\label{E:m>0}
\|R^{m+1}_I f^*\|_{X'(0,1)} \leq \left\|R^m_I\left(\left(R_I
f^*\right)^*\right)\right\|_{X'(0,1)} \leq C \|G^{m+1}_I
f\|_{X'(0,1)}.
\end{equation}
Observe that~\eqref{E:m>0} trivially holds also when $m=0$.
\par
Let $E$ be defined as in \eqref{E:E}, with $m$ replaced with $m+1$,
and let $\{(c_k,d_k)\}_{k\in S}$  be as in \eqref{E:decomp}. For
every $g \in X(0,1)$, define
\[
A(g)=\chi_{(0,1)\setminus E}g^* + \sum_{k\in S} \chi_{(c_k,d_k)}
\frac{1}{d_k-c_k} \int_{c_k}^{d_k} g^*(t)\,dt.
\]
Then $A(g)$ is non-increasing on $(0,1)$. Moreover, if
$\|g\|_{X(0,1)} \leq 1$, then by~\cite[Theorem 4.8, Chapter 2]{BS},
\begin{equation}\label{E:A}
\|A(g)\|_{X(0,1)} \leq \|g^*\|_{X(0,1)}=\|g\|_{X(0,1)} \leq 1.
\end{equation}
Therefore,
\begin{align*}
&\int_0^1 g^*(t) G^{m+1}_If(t)\,dt =\int_{(0,1)\setminus E} g^*(t)
R^{m+1}_{I} f^* (t)\,dt
+\sum_{k\in S} \int_{c_k}^{d_k} g^*(t) R^{m+1}_{I} f^*(d_k) \,dt\\
&= \int_{(0,1)\setminus E} g^*(t) R^{m+1}_{I} f^* (t)\,dt
+ \sum_{k\in S} \frac{1}{d_k-c_k} \left(\int_{c_k}^{d_k} g^*(t)\,dt\right) (d_k-c_k) R^{m+1}_{I} f^*(d_k) \\
&\leq \int_{(0,1)\setminus E} A(g)(t) R^{m+1}_{I} f^* (t)\,dt
+2^{m+2} \sum_{k\in S} \int_{c_k}^{d_k} A(g)(t) R^{m+1}_{I} f^*(t)\,dt~~~~~~~~~\qquad \qquad \textup{(by \eqref{E:1bis})}\\
&\leq 2^{m+2} \int_0^1 A(g)(t) R^{m+1}_{I} f^*(t)\,dt
\\
& \leq 2^{m+2} \sup_{\|h\|_{X(0,1)}\leq 1} \int_0^1 h^*(t) R^{m+1}_{I} f^*(t)\,dt~~~~~~~~~ \qquad \qquad\textup{(by~\eqref{E:A})}\\
&
=2^{m+2} \|R^{m+1}_{I} f^*\|_{X'_d(0,1)}.
\end{align*}
On  taking the supremum over all $g$ from the unit ball of $X(0,1)$,
we get
\begin{equation}\label{E:Abis}
\|G^{m+1}_If\|_{X'(0,1)}=\|G^{m+1}_If\|_{X'_d(0,1)}\leq
2^{m+2}\|R^{m+1}_{I} f^*\|_{X'_d(0,1)}.
\end{equation}
On the other hand, by the very definition of $\|\cdot \|_{X'_d(0,1)}$,
\begin{equation}\label{E:Ater}
\|R^{m+1}_{I} f^*\|_{X'_d(0,1)} \leq \|R^{m+1}_{I} f^*\|_{X'(0,1)}.
\end{equation}
Equation \eqref{E:down-dual-equivalence} follows from \eqref{E:m>0},
\eqref{E:Abis} and \eqref{E:Ater}.
\end{proof}

\begin{corollary}\label{C:lenka}
Let $I:[0,1]\to[0,\infty)$ be a~non-decreasing function
satisfying~\eqref{E:lower-bound}, and let $\|\cdot\|_{X(0,1)}$ be
a~rearrangement-invariant function norm. Let $m\in\N$. Then
\begin{equation}\label{E:lenka-eq}
(X_{m,I})_1(0,1)=X_{m+1,I}(0,1)
\end{equation}
(up to equivalent norms).
\end{corollary}

\begin{proof}
By~\eqref{E:krus} and~\eqref{E:xji}, if $f\in \M_+(0,1)$, then
\[
\|f\|_{((X_{m,I})_1)'(0,1)}=\|R_If\sp*\|_{X'_{m,I}(0,1)}=(m-1)!\|R\sp
m_I((R_If\sp*)\sp*)\|_{X'(0,1)},
\]
and
\[
\|f\|_{X'_{m+1,I}(0,1)}=m!\|R\sp{m+1}_If\sp*\|_{X'(0,1)}.
\]
Hence, it follows from Theorem~\ref{T:lenka-main} that
\[
\|f\|_{((X_{m,I})_1)'(0,1)}\approx \|f\|_{X'_{m+1,I}(0,1)}.
\]
By \eqref{E:equivemb}, this establishes \eqref{E:lenka-eq}.
\end{proof}

\begin{thm}\label{T:equivalence}
Let $I:[0,1]\to[0,\infty)$ be a~non-decreasing function
satisfying~\eqref{E:lower-bound} and let $\|\cdot\|_{X(0,1)}$ be
a~rearrangement-invariant function norm.
Then, for every $m\in\N$,
\begin{equation}\label{E:equivalence}
X_{m,I}(0,1)=X_m(0,1).
\end{equation}
\end{thm}

\begin{proof}
As noted in~\eqref{E:eq-one}, we have $X_1(0,1)=X_{1,I}(0,1)$.
Assume now that  \eqref{E:equivalence} holds for some $m\in\N$.
By~\eqref{E:reit-iter}, the induction assumption
and~\eqref{E:lenka-eq},
\[
X_{m+1}(0,1)=(X_m)_1(0,1)=(X_{m,I})_1(0,1)=X_{m+1,I}(0,1).
\]
The conclusion  follows by induction.
\end{proof}

One consequence of Theorem~\ref{T:lenka-main},
specifically of the equivalence of the leftmost and the rightmost
side of~\eqref{E:down-dual-equivalence}, is the following feature of
 inequality \eqref{E:kernel-cond}, which was already mentioned
in Remark~\ref{R:star}.

\begin{corollary}\label{C:restriction-of-main-inequality-to-nonincreasing-functions}
Assume that  $(\Omega,\nu)$  fulfils
\eqref{isop-ineq} for some non-decreasing function  $I$ satisfying
\eqref{E:lower-bound}. Let $m\in\N$, and let $\|\cdot \|_{X(0,1)}$
and $\|\cdot\|_{Y(0,1)}$ be rearrangement-invariant function norms.
Then the following two assertions are equivalent:

\textup{(i)} There exists a constant $C_1$ such
that   inequality~\eqref{E:kernel-cond} holds for every nonnegative
$f \in  X(0,1)$.

\textup{(ii)} There exists a constant $C_1'$
such that   inequality~\eqref{E:kernel-cond} holds for every
nonnegative non-increasing $f \in  X(0,1)$.
\end{corollary}

\begin{proof}
The fact that (i) implies (ii) is trivial.
Conversely, assume that (ii)
holds. Fix $f\in\Mpl(0,1)$. Equation~\eqref{E:kernel-formula} with
$j=m$ reads
\begin{equation}\label{Hm}
\int_t\sp1\frac{f(s)}{I(s)}\left(\int_t\sp
s\frac{dr}{I(r)}\right)\sp{m-1}\,ds=(m-1)!H\sp m_If(t),
\end{equation}
Now, the function $H\sp m_If$ is non-increasing on $(0,1)$. Therefore, it follows from~\eqref{E:pre-duality} and the Hardy--Littlewood inequality~\eqref{E:HL} that
\[
\|H\sp m_If\|_{Y(0,1)} =
\sup_{\|g\|_{Y'(0,1)}\leq 1} \int_0\sp1 g\sp*(t)H\sp m_If(t)\,dt.
\]
Consequently, by  Fubini's theorem, we have
\begin{equation}\label{E:duality-formula-for-hmi}
\|H\sp m_If\|_{Y(0,1)} =
\sup_{\|g\|_{Y'(0,1)}\leq 1}\int_0\sp1 f(t)R\sp m_Ig\sp{*}(t)\,dt.
\end{equation}
Owing to \eqref{Hm} and to the
rearrangement-invariance of the norm $\|\cdot\|_{X(0,1)}$, assertion
(ii) tells us that
\[
C_1'
\geq
(m-1)!\sup_{\|f\|_{X(0,1)}\leq 1}
\|H\sp m_If\sp*\|_{Y(0,1)}.
\]
Hence, on applying~\eqref{E:duality-formula-for-hmi} with $f$
replaced with $f\sp*$, interchanging the suprema and recalling the
definition of the norm $\|\cdot\|_{X_d'(0,1)}$, we get
\begin{align*}
C_1'
&\geq
(m-1)!\sup_{\|f\|_{X(0,1)}\leq 1}
\sup_{\|g\|_{Y'(0,1)}\leq 1}
\int_0\sp1
f\sp*(t)R\sp m_Ig\sp{*}(t)\,dt\\
&=
(m-1)!\sup_{\|g\|_{Y'(0,1)}\leq 1}
\sup_{\|f\|_{X(0,1)}\leq 1}
\int_0\sp1
f\sp*(t)R\sp m_Ig\sp{*}(t)\,dt
=
(m-1)!
\sup_{\|g\|_{Y'(0,1)}\leq 1}\|R\sp m_Ig\sp*\|_{X'_d(0,1)}.
\end{align*}
It follows from the equivalence of the first and the last term
in~\eqref{E:down-dual-equivalence} that there exists a constant $C$
such that
\[
\|R\sp m_Ig\sp*\|_{X'(0,1)}\leq C\|R\sp m_Ig\sp*\|_{X'_d(0,1)}.
\]
Therefore,
\[
C_1'C \geq (m-1)! \sup_{\|g\|_{Y'(0,1)}\leq 1}\|R\sp
m_Ig\sp*\|_{X'(0,1)},
\]
namely, by the definition of the norm $\|\cdot\|_{X'(0,1)}$,
\[
C_1'C
\geq
(m-1)!
\sup_{\|g\|_{Y'(0,1)}\leq 1}\sup_{\|f\|_{X(0,1)}\leq 1}\int_0\sp1
f(t)R\sp m_Ig\sp{*}(t)\,dt.
\]
Interchanging suprema again and using  Fubini's theorem
and~\eqref{E:duality-formula-for-hmi} yields
\[
C_1'C \geq (m-1)! \sup_{\|f\|_{X(0,1)}\leq
1}\sup_{\|g\|_{Y'(0,1)}\leq 1}\int_0\sp1 g\sp*(t)H\sp
m_If(t)\,dt=(m-1)! \sup_{\|f\|_{X(0,1)}\leq 1}\|H\sp
m_If\|_{Y(0,1)}.
\]
Hence, inequality \eqref{E:kernel-cond}, or
equivalently assertion (i), follows.
\end{proof}

\begin{proof}[{\bf Proof of Theorem~\ref{T:reduction}}] As observed in
Section \ref{S:main}, the case when $m=1$ is already well-known, and
is in fact the point of departure of our approach. We thus focus on
the case when $m \geq 2$. On applying Proposition~\ref{P:norm} with
$j=1$, we get that
$$
\left\|\int_t\sp1\frac{f(s)}{I(s)}\,ds\right\|_{X_{1,I}(0,1)} \leq
\|f\|_{X(0,1)}
$$
for every $f\in\M_+(0,1)$. Thus,~\eqref{E:kernel-cond} holds with
$m=1$ and  $Y(0,1)=X_{1,I}(0,1)$. Hence, by the result for $m=1$,
\begin{equation}\label{embedfirst}
V\sp1X (\Omega,\nu)\to X_{1}(\Omega,\nu).
\end{equation}
Note that here we have also made use of~\eqref{E:eq-one}. By
embedding \eqref{embedfirst} applied  to each of the spaces
$X_j(\Omega,\nu)$, for $j=0, \dots , m-1$, we get
$$
V\sp1X_j(\Omega,\nu)\to X_{j+1}(\Omega,\nu),
$$
whence
\begin{equation}\label{embedm}
V\sp mX(\Omega,\nu)\to  V\sp {m-1}X_1(\Omega,\nu) \to  V\sp
{m-2}X_2(\Omega,\nu) \to \dots \to V\sp {1}X_{m-1}(\Omega,\nu) \to
X_m(\Omega,\nu).
\end{equation}
Inequality   \eqref{E:kernel-cond} tells us that
\begin{equation}\label{Eq:*}
H_I\sp m:X(0,1)\to Y(0,1).
\end{equation}
The optimality of the space $X_{m,I}(0,1)$ as a target in
\eqref{Eq:*},
proved in
Proposition~\ref{P:norm},  entails that
\begin{equation}\label{sept6}
X_{m,I}(0,1)\to Y(0,1).
\end{equation}
A combination of \eqref{embedm}, ~\eqref{E:equivalence}
and~\eqref{sept6} yields
\begin{equation}\label{sept7}
V\sp mX(\Omega,\nu)\to X_m(\Omega,\nu)= X_{m,I}(\Omega,\nu)\to
Y(\Omega,\nu),
\end{equation}
and \eqref{E:main-emb} follows.

Finally,~\eqref{E:main-poincare} is equivalent to~\eqref{E:main-emb}
by Proposition~\ref{poincsob}. Note  that assumption
~\eqref{lower-est} of that Proposition is satisfied, owing to
\eqref{E:lower-bound}.
\end{proof}

\begin{proof}[{\bf Proof of Theorem \ref{T:optran}}]
Embedding \eqref{E:optemb} is a straightforward consequence of
\eqref{sept7}. In turn, Proposition~\ref{poincsob} yields the
Poincar\'e inequality \eqref{E:optran}.

Assume now that the validity of~\eqref{E:main-emb}
implies~\eqref{E:kernel-cond}.
Let $\|\cdot\|_{Y(0,1)}$ be any rearrangement-invariant function
norm such that \eqref{E:main-emb} holds.
Then, by our assumption, inequality~\eqref{E:kernel-cond} holds as
well, namely \begin{equation}\label{**} H_I^m : X(0,1) \to Y(0,1).
\end{equation}
Since, by Proposition \ref{P:norm},  $X_{m,I}(0,1)$ is the optimal
rearrangement-invariant target space in \eqref{**},
we necessarily have
$$
X_{m,I}(0,1)\hra Y(0,1).
$$
This implies the optimality of the norm $\|\cdot\|_{X_{m,I}(0,1)}$
in
\eqref{E:optemb}.
\end{proof}

\begin{proof}[{\bf Proof of Corollary \ref{linfinity}}]
Observe that
\begin{align*}\label{linfinity11}
\sup_{ {\overset{f\geq 0}{\|f\|_{X(0,1)}\leq 1}} } \bigg\| \int
_t^1\frac{f(s)}{I(s)} \left(\int_t\sp
s\frac{dr}{I(r)}\right)\sp{m-1} ds \bigg\|_{L^\infty(0,1)} & =
\sup_{{\overset{f\geq 0}{\|f\|_{X(0,1)}\leq 1}}} \int
_0^1\frac{f(s)}{I(s)} \left(\int_0\sp
s\frac{dr}{I(r)}\right)\sp{m-1} ds
\\  & = \bigg\|\frac{1}{I(s)} \left(\int_0\sp
s\frac{dr}{I(r)}\right)\sp{m-1} \bigg\|_{X'(0,1)}.
\end{align*}
Hence, \eqref{linfinity1} is equivalent to~\eqref{E:kernel-cond}
with $Y(0,1)=L\sp{\infty}(0,1)$. The assertion thus follows from
Theorem~\ref{T:reduction}.
\end{proof}

\begin{proof}[{\bf Proof of Theorem \ref{T:reit}}]
By Theorem~\ref{T:equivalence} and \eqref{E:reit-iter},
$$
(X_{k,I})_{h,I}(0,1)=
(X_k)_h(0,1)=X_{k+h}(0,1)=X_{k+h,I}(0,1),
$$
and the claim follows.
\end{proof}

Corollaries \ref{T:convergent}, \ref{T:optimalconv} and
\ref{T:iterconv} follow from Theorems \ref{T:reduction},
\ref{T:optran} and \ref{T:reit}, respectively (via
Propositions \ref{link}--\ref{P:norm-conv}).

\begin{proof}[{\bf Proof of Proposition \ref{WV}}]
Owing to \eqref{inclWV}, equation \eqref{WV1} will follow if we show
that
\begin{equation}\label{WV3}
V^mX(\Omega , \nu )\to W^mX(\Omega , \nu ).
\end{equation}
The isoperimetric function $I_{\Omega , \nu}$ is non-decreasing on $[0,\frac13]$ by
definition. Let us define the function $I$ by
\begin{equation}\label{Ipiecewise}
I(s)=
\begin{cases}
I_{\Omega , \nu}(s) &\textup{if}\ s\in[0,\frac13],\\
I_{\Omega , \nu}(\frac13) &\textup{if}\ s\in[\frac13,1].
\end{cases}
\end{equation}
Then $I$ is non-decreasing on $[0,1]$. Moreover, by \eqref{WV2}, it satisfies
\eqref{lower-est}. Let $H_I$ be the operator defined as in \eqref{Hbis}, with $I$ given by
\eqref{Ipiecewise}.
Then,
\[
\|H_{I}f\|_{L\sp1(0,1)}\leq\int_0\sp 1f(t)\frac
t{I(t)}\,dt\leq C\|f\|_{L\sp1(0,1)},
\]
and
\[
\|H_{I}f\|_{L\sp{\infty}(0,1)}\leq\int_0\sp 1\frac
{f(t)}{I(t)}\,dt\leq
\int_0\sp1\frac{ds}{I(s)}\|f\|_{L\sp{\infty}(0,1)} \leq
C\|f\|_{L\sp{\infty}(0,1)},
\]
 for every $f\in\Mpl(0,1)$.
Thus, $H_{I}$ is well defined and bounded  both on $L\sp1(0,1)$ and on
$L\sp{\infty}(0,1)$. Owing to  an interpolation theorem of Calder\'on \cite[Chapter 3, Theorem 2.12]{BS},
the operator  $H_{I}$ is bounded on every r.i.~space
$X(0,1)$. Hence, from Theorem~\ref{T:reduction} applied with
$Y(0,1)=X(0,1)$ and $m=1$,
 we obtain that
\begin{equation}\label{july4}
V^1X(\Omega , \nu )  \to X(\Omega , \nu ).
\end{equation}
Iterating \eqref{july4} tells us that there exists a constant $C$
such that
\begin{equation}\label{nov11}\| \nabla ^h u \|_{X(\Omega , \nu )} \leq C \Big(\sum _{k=h}^{m-1}
\| \nabla ^k u \|_{L^1(\Omega , \nu )} + \| \nabla ^m u \|_{X(\Omega
, \nu )}\Big)
\end{equation}
for every $h=0, \dots , m-1$, and $u \in V^m X (\Omega , \nu )$.
Embedding
 \eqref{WV3} is a consequence of \eqref{nov11}.
\end{proof}

\section{Proofs of the Euclidean Sobolev embeddings}\label{proofeucl}

In what follows,  we shall make use of the fact that the function
$I(t)=t\sp{\alpha}$
 satisfies \eqref{E:doubling} if
$\alpha\in(0,1)$.

\begin{proof}[{\bf Proof of Theorem \ref{T:eucl_reduction-john}}]
If the one-dimensional inequality \eqref{E:eucl_condition-john}
holds, then the Sobolev embedding \eqref{E:eucl_embedd-john} and the
Poincar\'e inequality \eqref{E:eucl_reduction-john} hold as well,
owing to \eqref{Ijohn} and to Corollary \ref{T:convergent}.
This shows that (i) implies (ii) and (iii). The equivalence of (ii)
and (iii) is a consequence of Proposition \ref{poincsob}.

It thus only remains to prove that  (ii) implies (i). Assume that
the Sobolev embedding \eqref{E:eucl_embedd-john} holds. If
$m\geq n$, then there is nothing to prove,
because~\eqref{E:eucl_condition-john} holds for every rearrangement
invariant spaces $X(0,1)$ and $Y(0,1)$. Indeed,
$$
\left\|\int_t\sp1 f(s)s\sp{-1+\frac
mn}\,ds\right\|_{L\sp{\infty}(0,1)} = \int_0\sp1 f(s)s\sp{-1+\frac
mn}\,ds\leq \|f\|_{L\sp1 (0,1)}
$$
for every nonnegative $f \in L\sp1 (0,1)$, and  hence
\eqref{E:eucl_condition-john} follows from~\eqref{l1linf}.  In the
case when $m\leq n-1$, the validity of \eqref{E:eucl_condition-john}
was proved in~\cite[Theorem~A]{T2}. Note that the proof is given in
\cite{T2} for Lipschitz domains, and with the space
$W\sp{m}X(\Omega)$ in the place of $V\sp{m}X(\Omega)$. However, by
Proposition~\ref{WV}, $W\sp{m}X(\Omega)=V\sp{m}X(\Omega)$ if
$\Omega$ is a John domain, since \eqref{WV2} is fulfilled for any
such domain. Moreover,
the Lipschitz property of the domain is immaterial, since the proof
does not involve any property of the boundary and hence applies,  in
fact, to any open set $\Omega$.
\end{proof}

\begin{proof}[{\bf Proof of Theorem \ref{T:eucl_optimal-john}}]
By Theorem~\ref{T:eucl_reduction-john},  every John domain
has the property that ~\eqref{E:conv-embedding}
implies~\eqref{E:condition-convergent}. Consequently, the conclusion
follows from Corollary~\ref{T:optimalconv}.
\end{proof}

\begin{proof}[{\bf Proof of Theorem \ref{T:iteration-john}}]
The assertion is a consequence of Corollary~\ref{T:iterconv}.
\end{proof}

The following result provides us with model Euclidean domains of
revolution in $\rn$ in the class $\mathcal J _ \alpha$. It is an
easy consequence of a special case of  \cite[Section 5.3.3]{Mabook}.
In the statement, $\omega _{n-1}$ denotes the Lebesgue measure of
the unit ball in $\mathbb R^{n-1}$.

\begin{prop}\label{auxdiv}
\textup{(i)} Given  $\alpha \in [\frac 1{n'}, 1)$, define $\eta
_\alpha : [0, \frac 1{1-\alpha}] \to [0, \infty )$ as
$$\eta _\alpha (r) = \omega _{n-1}^{-\frac 1{n-1}} (1 -
(1-\alpha )r)^{\frac \alpha{(1-\alpha)(n-1)}} \qquad \hbox{for $r
\in [0, \tfrac 1{1-\alpha}]$.}$$ Let $\Omega $ be the Euclidean
domain in $\rn$ given by
\[
\Omega = \{ (x', x_n) \in \rn: x' \in \mathbb R^{n-1}, 0< x_n
<\tfrac 1{1-\alpha}, |x'| < \eta _\alpha (x_n) \}.
\]
Then $|\Omega |=1$, and
\begin{equation}\label{auxdiv1}
I_\Omega (s) \approx s^\alpha \quad \hbox{for $s \in [0, \frac
12]$.}
\end{equation}

\textup{(ii)} Define $\eta _1 : [0, \infty ) \to [0, \infty )$ as
$$\eta _1 (r) = \omega _{n-1}^{-\frac 1{n-1}} e^{- \frac r{n-1}} \qquad \hbox{for $r
\geq0$.}$$ Let $\Omega $ be the Euclidean
domain in $\rn$ given by
\[
\Omega = \{ (x', x_n) \in \rn: x' \in \mathbb R^{n-1}, x_n >0, |x'|
< \eta _1 (x_n)   \}.
\]
Then $|\Omega |=1$, and
\begin{equation}\label{auxdiv2}
I_\Omega (s) \approx s  \quad \hbox{for $s \in [0, \frac 12]$.}
\end{equation}
\end{prop}

\begin{proof}[{\bf Proof of Theorem~\ref{T:eucl_reduction-div}}]
The  Sobolev embedding \eqref{E:eucl_embedd-div} and the Poincar\'e
inequality \eqref{E:eucl_reduction-div}
 are equivalent, owing to Theorem \ref{poincsob}.
 If $\alpha \in
[\frac 1{n'}, 1)$, then inequality \eqref{E:eucl_condition-div}
implies \eqref{E:eucl_embedd-div} and \eqref{E:eucl_reduction-div},
via Corollary \ref{T:convergent}, whereas if $\alpha =1$, then
inequality \eqref{E:eucl_condition-div1}  implies
\eqref{E:eucl_embedd-div} and \eqref{E:eucl_reduction-div} via
 Theorem~\ref{T:reduction}.

 It thus remains to exhibit a domain $\Omega \in \mathcal J _\alpha$ such that the Sobolev embedding
\eqref{E:eucl_embedd-div} implies either
\eqref{E:eucl_condition-div}, or \eqref{E:eucl_condition-div1},
according to whether $\alpha \in [\frac 1{n'}, 1)$ or $\alpha =1$.
\par\noindent
If $\alpha\in[\frac1{n'},1)$, let $\Omega$ be the set given by
Proposition \ref{auxdiv}, Part (i), whereas, if $\alpha=1$, let
$\Omega$ be the set given by Proposition \ref{auxdiv}, Part (ii). By
either \eqref{auxdiv1} or~\eqref{auxdiv2}, one has that $\Omega \in
\mathcal J _\alpha$.  Consequently, embedding
\eqref{E:eucl_embedd-div} entails that
 there exists a  constant $C$ such that
\begin{equation}\label{Enec0div}
\|u\|_{Y(\Omega )} \leq C \Big(\| \nabla ^m u\|_{X(\Omega )} + \sum
_{k=0}^{m-1} \|\nabla ^k u\|_{L^1(\Omega )}\Big)
\end{equation}
for every $u \in V^m X(\Omega )$. Let us fix any nonnegative
function $f \in X (0,1)$, and define $u : \Omega \to [0, \infty )$
as
\[
u (x) = \int _{M_\alpha (x_n)} ^1 \frac{1}{r_1^\alpha} \int _{r_1}
^1 \frac{1}{r_2^\alpha} \dots \int _{r_{m-1}}^1
\frac{f(r_m)}{r_m^\alpha} dr_m\, dr_{m-1}\, \dots dr_1 \quad
\hbox{for $ x \in \Omega ,$}
\]
where $M_\alpha $ is given by
$$M_\alpha (r) = \begin{cases} (1- (1-\alpha )r)^{\frac 1{1-\alpha }}
&  \hbox{ for $r \in [0, \frac 1{1-\alpha}]$,} \quad \hbox{if
$\alpha \in [\frac 1{n'}, 1)$,}
\\ e^{-r} &  \hbox{ for $r \in [0, \infty )$,}  \quad \hbox{if $\alpha
=1$.}
\end{cases}$$
 The function $u$ is $m$-times
weakly differentiable in $\Omega$, and, since $ - M_\alpha ' =
(M_\alpha)^\alpha $,
\[
|\nabla ^k u(x)| = \frac {\partial ^k u}{\partial x_n
 ^k}  (x) =
\int _{M_\alpha (x_n)} ^1 \frac{1}{r_{k+1}^\alpha} \int _{r_{k+1}}
^1 \frac{1}{r_{k+2}^\alpha} \dots \int _{r_{m-1}}^1
\frac{f(r_m)}{r_m^\alpha} dr_m\, dr_{m-1}\, \dots dr_{k+1} \quad
\hbox{for $ x \in \Omega ,$}
\]
for $k=1, \dots , m-1$,  and
\[
 |\nabla ^m u(x)| = \frac {\partial ^m u}{\partial x_n
 ^m}  (x) =  f(M_{\alpha}(x_n)) \quad \hbox{for a.e. $ x \in \Omega
$.}
\]
 Moreover, on setting $L_\alpha = \frac 1{1-\alpha}$ if $\alpha \in [\frac 1{n'},
 1)$, and $L_\alpha = \infty $ if $\alpha =1$, we have that
\begin{align*}
|\{(x', x_n)\in \Omega: x_n >t \}| &= \omega _{n-1} \int
_{t}^{L_\alpha} \eta_\alpha (r)^{n-1}\, dr = \int _{t}^{L_\alpha}
M_\alpha(r)^\alpha\, dr
\\ &= \int _{t}^{L_\alpha} -M_\alpha '(r)\, dr = M_\alpha (t) \quad
\quad\hbox{for $t \in (0, L_\alpha)$}.
\end{align*}
Thus,
\begin{equation}\label{Enec4div}
u^*(s) =\int _{s} ^1 \frac{1}{r_1^\alpha} \int _{r_1} ^1
\frac{1}{r_2^\alpha} \dots \int _{r_{m-1}}^1
\frac{f(r_m)}{r_m^\alpha} dr_m\, dr_{m-1} \, \dots dr_1 \quad
\hbox{for $ s \in (0, 1),$}
  \end{equation}
\begin{equation}\label{Enec4divk}
|\nabla ^k u|^*(s) = \int _{s} ^1 \frac{1}{r_{k+1}^\alpha} \int
_{r_{k+1}} ^1 \frac{1}{r_{k+2}^\alpha} \dots \int _{r_{m-1}}^1
\frac{f(r_m)}{r_m^\alpha} dr_m\, dr_{m-1}\, \dots dr_{k+1} \quad
\hbox{for $ s \in (0, 1),$}
\end{equation}
for  $1 \leq k \leq m-1$, and
\begin{equation}\label{Enec6div}
 |\nabla ^m u|^*(s) =    f^*(s) \quad \hbox{for $ s
\in (0, 1).$}
\end{equation}
 \relax Equation
\eqref{Enec6div} ensures  that $u \in V^m X(\Omega )$. On the other
hand, by \eqref{Enec0div} and \eqref{Enec4div}--\eqref{Enec6div},
    \begin{align}\label{Enec7div}
    \bigg\| \int _{s} ^1 & \frac{1}{r_1^\alpha} \int _{r_1} ^1
\frac{1}{ r_2^\alpha} \dots \int _{r_{m-1}}^1 \frac{f(r_m)}{
r_m^\alpha} dr_m\, dr_{m-1} \, \dots dr_1  \bigg\|_{ Y(0,1)} \\
\nonumber &\leq C \|f\|_{ X (0,1)} +  C \sum _{k=0}^{m-1}
 \int _0^1
\int _{s} ^1 \frac{1}{r_{k+1}^\alpha} \int _{r_{k+1}} ^1
\frac{1}{r_{k+2}^\alpha} \dots \int _{r_{m-1}}^1
\frac{f(r_m)}{r_m^\alpha} dr_m\, dr_{m-1}\, \dots dr_{k+1}\, ds.
\end{align}

Subsequent applications of Fubini's Theorem tells us that
\begin{multline}\label{Enec8'div}
\int _{s} ^1 \frac{1}{r_{k+1}^\alpha} \int _{r_{k+1}} ^1
\frac{1}{r_{k+2}^\alpha} \dots \int _{r_{m-1}}^1
\frac{f(r_m)}{r_m^\alpha} dr_m\, dr_{m-1}\, \dots dr_{k+1} =
 \\ = \frac 1{(m-k-1)!}\int _{s} ^1 \frac{1}{r
^\alpha}\left(\int_s\sp r\frac{dt}{ t^\alpha}\right)\sp{m-k-1}f(r) dr
\quad \hbox{for $s \in (0, 1)$.}
\end{multline}

By \eqref{Enec8'div}, \eqref{l1linf} and~\eqref{E:bdj} applied with
$I(t)=t\sp{\alpha}$,  one has that
 \begin{align}\label{Enec8div}
\int _0^1 & \int _{s} ^1 \frac{1}{r_{k+1}^\alpha} \int _{r_{k+1}} ^1
\frac{1}{r_{k+2}^\alpha} \dots \int _{r_{m-1}}^1
\frac{f(r_m)}{r_m^\alpha} dr_m\, dr_{m-1}\, \dots dr_{k+1}\, ds  \\
\nonumber & = \frac 1{(m-k-1)!}\int _0^1 \int _{s} ^1
\frac{1}{r^\alpha}\left(\int_s\sp r\frac{dt}{
t^\alpha}\right)\sp{m-k-1}f(r) dr \, ds
=\|H\sp{m-k}_If\|_{L\sp1(0,1)}\\
&\leq C\|H\sp{m-k}_If\|_{(L\sp1)_{m-k,I}(0,1)}
\leq C'\|f\|_{L\sp1(0,1)}.\nonumber
\end{align}
for $k=0, \dots , m-1$, for some  constants $C$ and $C'$.

When $\alpha =1$,   inequality \eqref{E:eucl_condition-div1} follows
from \eqref{Enec7div}, \eqref{Enec8'div} and \eqref{Enec8div}. When
$\alpha \in [\frac 1{n'}, 1)$, inequality
\eqref{E:eucl_condition-div} follows from \eqref{Enec7div},
\eqref{Enec8'div} and \eqref{Enec8div}, via Proposition \ref{link},
Part (i).
\end{proof}

\begin{thm}\label{T:lenka-remark-3}
Let $p,q\in[1,\infty]$ and $\alpha\in\R$ be such that one of the
conditions in~\eqref{E:lz_bfs} is  satisfied. Let  $I: [0, 1] \to
[0, \infty)$ be a non-decreasing function  such that $\frac t{I(t)}$
is non-decreasing. Then
\[
\left\|R_If\sp{*}\right\|_{(L^{p,q;\alpha})'(0,1)} \approx
\left\|t^{\frac{1}{p'}-\frac{1}{q'}} \left(\log \tfrac{2}{t}
\right)^{-\alpha} R_If\sp{*}(t) \right\|_{L^{q'}(0,1)}
\]
for every $f\in \M_+(0,1)$, up to multiplicative constants depending
on $p,q,\alpha$.
\end{thm}

\begin{proof} Fix $f\in \M_+(0,1)$.
By Theorem~\ref{T:lenka-main}, applied with $m=0$ and
$X(0,1)=L\sp{p,q;\alpha}(0,1)$, and H\"older's inequality, there
exists a universal constant $C$ such that
\begin{multline*}
\left\|R_If\sp*(t)\right\|_{(L\sp{p,q;\alpha})'(0,1)} \leq C
\left\|R_If\sp *(t)\right\|_{(L\sp{p,q;\alpha})'_d(0,1)} =
C\sup_{\|g\|_{L\sp{p,q;\alpha}(0,1)}\leq 1}
\int_0\sp 1g\sp*(t)R_If\sp*(t)\,dt\\
= C\sup_{\|g\|_{L\sp{p,q;\alpha}(0,1)}\leq 1}
\int_0\sp 1g\sp*(t)t\sp{\frac 1p-\frac1q}(\log\tfrac2t)\sp{\alpha}t\sp{\frac 1{p'}-\frac1{q'}}(\log\tfrac2t)\sp{-\alpha}R_If\sp*(t)\,dt \leq C \|t\sp{\frac
1{p'}-\frac1{q'}}(\log\tfrac2t)\sp{-\alpha}R_If\sp*(t)\|_{L\sp{q'}(0,1)}.
\end{multline*}
In order to prove the reverse inequality,  assume first that either $1<p<\infty$ or $p=q=1$ and
$\alpha\geq0$ or $p=q=\infty$ and $\alpha\leq0$. By
Theorem~\ref{T:lenka-main} (with $m=0$) and~\eqref{E:lz_assoc},
\begin{align*}
&\|t\sp{\frac
1{p'}-\frac1{q'}}(\log\tfrac2t)\sp{-\alpha}R_If\sp*(t)\|_{L\sp{q'}(0,1)}
\leq
\|t\sp{\frac 1{p'}-\frac1{q'}}(\log\tfrac2t)\sp{-\alpha}\sup_{t\leq s\leq 1}R_If\sp*(t)\|_{L\sp{q'}(0,1)}\\
&= \left\|G_If\right\|_{L\sp{p',q';-\alpha}(0,1)}
\approx \left\|G_If\right\|_{(L\sp{p,q;\alpha})'(0,1)}
\approx
\|R_If\sp*\|_{(L\sp{p,q;\alpha})'(0,1)},
\end{align*}
where the last {but one} equivalence holds up to multiplicative constants
depending on $p, q, \alpha$.
\par
It remains to consider the case when $p=\infty$, $q\in [1,\infty)$
and $\alpha+\frac1q<0$.
We have that
\begin{align*}
& \left\|R_If\sp{*}(t)\right\|_{(L^{p,q;\alpha})'(0,1)}
\approx \left\|R_If\sp{*}(t)\right\|_{L^{(1,q';-\alpha-1)}(0,1)} =\left\|t^{1-\frac{1}{q'}} \left(\log \tfrac{2}{t}\right)^{-\alpha-1} \frac{1}{t} \int_0^t \left(R_If\sp{*}\right)^{*}(s)\,ds\right\|_{L^{q'}(0,1)}\\
&\geq \left\|t^{1-\frac{1}{q'}} \left(\log \tfrac{2}{t}\right)^{-\alpha-1} \frac{1}{t} \int_0^t R_If\sp{*}(s)\,ds\right\|_{L^{q'}(0,1)} =\left\|t^{-\frac{1}{q'}} \left(\log \tfrac{2}{t}\right)^{-\alpha-1} \int_0^t f^*(r)\int_r^t \frac{\,ds}{I(s)}\,dr\right\|_{L^{q'}(0,1)}\\
&\geq \left\|t^{-\tfrac{1}{q'}} \left(\log \tfrac{2}{t}\right)^{-\alpha-1} \int_0^{t^2} f^*(r)\int_r^t \frac{\,ds}{I(s)}\,dr\right\|_{L^{q'}(0,1)}
\geq \left\|t^{-\frac{1}{q'}} \left(\log \tfrac{2}{t}\right)^{-\alpha-1} \int_0^{t^2} f^*(r) \,dr \int_{t^2}^t \frac{s}{I(s)} \, \frac{ds}{s}\right\|_{L^{q'}(0,1)}\\
&\geq \left\|t^{2-\frac{1}{q'}} \left(\log \tfrac{2}{t}\right)^{-\alpha-1} f^{**}(t^2) \frac{t^2}{I(t^2)} \left(\log \frac{1}{t}\right)\right\|_{L^{q'}(0,1)}
\geq \tfrac{1}{2}\left\|\chi_{(0,\frac{1}{2})}(t) t^{2-\frac{1}{q'}} \left(\log \tfrac{2}{t}\right)^{-\alpha} \frac{t^2}{I(t^2)}f^{**}(t^2)\right\|_{L^{q'}(0,1)}\\
& \geq C
\left\|t^{1-\frac{1}{q'}}\left(\log \tfrac{2}{t}\right)^{-\alpha}
R_If\sp{*}(t)\right\|_{L^{q'}(0,1)},
\end{align*}
for some constant $C=C(\alpha ,q)$. The proof is complete.
\end{proof}

\begin{proof}[{\bf Proof of Theorem~\ref{EX:eucl_lebesgue}}]
By Corollary~\ref{T:optimalconv},
\[
\|f\|_{((L\sp
p)_{m,\alpha})'(0,1)}=\|s\sp{m(1-\alpha)}f\sp{**}(s)\|_{L\sp{p'}(0,1)}
\]
for $f \in \M_+(0,1)$. If $m(1-\alpha)<1$ and $p<\frac 1{m(1-\alpha)}$ and $r$ is given by $\frac 1r=\frac 1p-m(1-\alpha)$ (note that $1<r<\infty$), then this,~\eqref{E:glz-identity} and~\eqref{E:lz_assoc} yield
\[
(L\sp p)_{m,\alpha}(0,1)=(L\sp{(r',p')})'(0,1)=(L\sp{r',p'})'(0,1)=L\sp{r,p}(0,1).
\]
Since $p<r$, we have $L\sp{r,p}(0,1)\hra L\sp{r}(0,1)$, and the claim follows. If $m(1-\alpha)<1$ and $p=\frac
1{m(1-\alpha)}$, then, by~\eqref{E:lorentz-identity}, $L\sp
{r'}(0,1)\hra L\sp{(1,p)}(0,1)$   for every $r\in[1,\infty)$, and
hence
\[
(L\sp p)_{m,\alpha}(0,1)=(L\sp{(1,p')})'(0,1)\hra L\sp r(0,1).
\]
Finally, if either $m(1-\alpha)\geq1$, or $m(1-\alpha)<1$ and
$p>\frac 1{m(1-\alpha)}$, then~\eqref{linfinity1} is satisfied.
The conclusion thus follows from Corollary~\ref{linfinity}.
\end{proof}

\begin{proof}[{\bf Proof of Theorem~\ref{EX:eucl_lorentz}}] First, assume that either $m(1-\alpha)\geq 1$, or $m(1-\alpha)<1$,
$p=\frac1{m(1-\alpha)}$ and $q=1$, or $m(1-\alpha)<1$ and
$p>\frac1{m(1-\alpha)}$. In each of these cases,
condition~\eqref{linfinity1} is satisfied with $I(t)=t\sp{\alpha}$ and $X(0,1)=L\sp{p,q}(0,1)$.
Hence, by Corollary~\ref{linfinity},
$V\sp{m}L\sp{p,q}(\Omega)\hra L\sp{\infty}(\Omega)$.

Next, assume that $m(1-\alpha)<1$, and either $1\leq p<\frac1{m(1-\alpha)}$, or $p=\frac1{m(1-\alpha)}$ and $q>1$.
Set $J(t)=t\sp{-m(1-\alpha)+1}$ for $t \in [0, 1]$. Then $J$ is
a~non-decreasing function  such that $\frac t{J(t)}$ is
non-decreasing on $(0,1)$. Given $f \in \M_+(0,1)$, by
Theorem~\ref{T:lenka-remark-3} (with $\alpha=0$),
\begin{align*}
\|f\|_{(L\sp{p,q}_{m,\alpha})'(0,1)} &=
\left\|s\sp{-1+m(1-\alpha)}\int_0\sp sf\sp*(r)\,dr\right\|_{(L\sp{p,q})'(0,1)}=\|R_Jf\sp*\|_{(L\sp{p,q})'(0,1)}\\
&\approx \left\|t\sp{\frac
1{p'}-\frac1{q'}}R_Jf\sp*(t)\right\|_{L\sp{q'}(0,1)} =
\left\|t\sp{\frac
1{p'}-\frac1{q'}+m(1-\alpha)}f\sp{**}(t)\right\|_{L\sp{q'}(0,1)} =
\|f\|_{L\sp{(r',q')}(0,1)},
\end{align*}
where the equivalence holds up to constants depending on $p$ and
$q$, and $r'$ satisfies $\frac 1{r'}=\frac1{p'}+m(1-\alpha)$. Owing
to~\eqref{E:glz-identity}, \eqref{E:lz_assoc} and~\eqref{E:lorentz-identity},
$L\sp{(r',q')}(0,1)=(L\sp{\frac{p}{1-mp(1-\alpha)},q})'(0,1)$ if
$m(1-\alpha)<1$ and $1\leq p<\frac 1{m(1-\alpha)}$, and
$L\sp{(r',q')}(0,1)=(L\sp{\infty,q;-1})'(0,1)$ if $m(1-\alpha)<1$,
$p=\frac 1{m(1-\alpha)}$ and $q>1$. The conclusion follows.
\end{proof}

\begin{proof}[{\bf Proof of Theorem~\ref{T:eucl_lebesgue-alpha=1}}]
Since~\eqref{isop-ineq} holds with $I(t)=t$, in the case when $1\leq p<\infty$, the assertion follows  from
Theorem~\ref{EX:expo_lz} applied with $\beta=1$ (see Section~\ref{S:probability_proofs} below for the proof of Theorem~\ref{EX:expo_lz}). If
$p=\infty$,   Theorem~\ref{EX:expo_lz} has to be combined with
an~appropriate relation from~\eqref{E:glz-identity}.
\end{proof}

\begin{proof}[{\bf Proof of Theorem~\ref{T:eucl_lorentz-alpha=1}}]
Since~\eqref{isop-ineq} holds with $I(t)=t$, the conclusion is a consequence of Theorem~\ref{EX:expo_lz} applied
with $\beta=1$.
\end{proof}

\section{Proofs of the Sobolev embeddings in product probability
spaces}\label{S:probability_proofs}

This final section is devoted to the proof of the results of Section
\ref{S:probability}.

\begin{lemma}\label{delta2}
Let $\Phi$ be as in  \eqref{E:measure}. Then:

\textup{(i)} The function $L_\Phi$ defined by
\eqref{Lphi} is nondecreasing on $[0,1]$;

\textup{(ii)} The inequality
\begin{equation}\label{E:delta21}
s \Phi ' \big(\Phi ^{-1}(\log \big (\tfrac
1s\big)\big)\big) \leq L_\Phi(s) \leq
2s \Phi ' \big(\Phi ^{-1}(\log \big (\tfrac
1s\big)\big)\big)
\end{equation}
holds for every $s\in (0,\frac{1}{2}]$;

\textup{(iii)} The inequality
\begin{equation}\label{E:e}
\frac{\Phi^{-1}(s)}{2s} \leq
\frac{1}{\Phi'(\Phi^{-1}(s))} \leq
\frac{\Phi^{-1}(s)-\Phi^{-1}(t)}{s-t} \leq \frac{\Phi^{-1}(s)}{s}
\end{equation}
holds whenever $0\leq t <s <\infty$.
\end{lemma}

\begin{proof}
\textup{(i)} The convexity of $\Phi$
and the concavity of  $\sqrt{\Phi}$ imply that
\[
0 \leq \Phi''(t)\leq\frac{\Phi'(t)\sp2}{2\Phi(t)} \quad \hbox{for $t
>0$.}
\]
Therefore,
\begin{equation*}
L_\Phi '(s)= \Phi '(\Phi ^{-1}(\log \tfrac 2s)) - \frac{\Phi ''(\Phi ^{-1}(\log \tfrac
2s))}{\Phi '(\Phi ^{-1}(\log \tfrac 2s))}
\geq \frac{\Phi ''(\Phi ^{-1}(\log \tfrac
2s))}{\Phi '(\Phi ^{-1}(\log \tfrac 2s))} \left(2\log \tfrac2s -1\right)>0 \quad \hbox{for $s \in (0,1)$.}
\end{equation*}
Hence,  \textup{(i)} follows.

\textup{(ii)} The first inequality in \eqref{E:delta21} trivially holds, since both
$\Phi '$ and $\Phi ^{-1}$ are non-decreasing  functions. The latter inequality follows from \textup{(i)} and from the fact that
$$
2s \Phi ' \big(\Phi ^{-1}(\log \big (\tfrac
1s\big)\big)\big) = L_\Phi(2s) \quad \hbox{for $s \in (0,\frac{1}{2}]$.}
$$

\textup{(iii)} Let $0\leq r_1 <r_2 <\infty$. Owing
to the convexity of $\Phi$ and the fact that $\Phi(0)=0$, we obtain
that
\begin{equation}\label{E:equation}
\frac{\Phi(r_2)}{r_2} \leq \frac{\Phi(r_2)-\Phi(r_1)}{r_2-r_1}\leq \Phi'(r_2).
\end{equation}
Furthermore, by the concavity of $\sqrt{\Phi}$ and
the fact that  $\sqrt{\Phi(0)}=0$,
\begin{equation*}
(\sqrt{\Phi})'(r_2) = \frac{\Phi'(r_2)}{2\sqrt{\Phi(r_2)}} \leq \frac{\sqrt{\Phi(r_2)}}{r_2},
\end{equation*}
and, therefore,
\begin{equation}\label{E:equation2}
\Phi'(r_2) \leq \frac{2\Phi(r_2)}{r_2}.
\end{equation}

Let $0\leq t<s<\infty$. If we set
$r_1=\Phi^{-1}(t)$, $r_2=\Phi^{-1}(s)$, then $0\leq r_1<r_2<\infty$.
Hence, inequalities~\eqref{E:equation} and~\eqref{E:equation2}
yield
$$
\frac{s}{\Phi^{-1}(s)} \leq \frac{s-t}{\Phi^{-1}(s)-\Phi^{-1}(t)} \leq \Phi'(\Phi^{-1}(s)) \leq \frac{2s}{\Phi^{-1}(s)}.
$$
Assertion~\eqref{E:e} follows.
\end{proof}

Let $m\in \mathbb N$. We define the operator
$P^m_\Phi$ from $\mathcal M_+(0,1)$ into $\mathcal M_+(0,1)$ by
$$
P^m_\Phi f(t)=\left(\frac{\Phi^{-1}(\log
\frac{2}{t})}{\log \frac{2}{t}} \right)^{m} \int_t^1 \frac{f(s)}{s}
\left(\log \frac{s}{t}\right)^{m-1}\,ds \quad \hbox{for}\,\, t\in
(0,1),
$$
and for $f\in \mathcal M_+(0,1)$.  Moreover, let
$H^m_{L_\Phi}$ be the operator defined as in \eqref{Hj} (see also~\eqref{E:kernel-formula}), with
$I=L_\Phi$, namely
$$
H^m_{L_\Phi} f(t) =\frac{1}{(m-1)!}\int_t^1 \frac{f(s)}{L_\Phi(s)}
\left(\int_t^s \frac{\,dr}{L_\Phi(r)}\right)^{m-1}\,ds \quad
\hbox{for}\,\, t\in (0,1),
$$
and for $f\in\Mpl(0,1)$. Observe that, by the change of
variables $\tau\mapsto \Phi\sp{-1}\left(\log\frac 2t\right)$, we have
\begin{align}
\frac{1}{{L_\Phi}(r)} \left(\int_s\sp r\frac{dt}{
{L_\Phi}(t)}\right)\sp{m-1}
&  =
\frac{1}{r\Phi'\left(\Phi\sp{-1}\left(\log\frac
2r\right)\right)}\left(\int_s\sp
r\frac{dt}{t\Phi'\left(\Phi\sp{-1}\left(\log\frac
2t\right)\right)}\right)\sp{m-1}\label{dec104}\\
 \nonumber & =
\frac{\left(\Phi\sp{-1}\left(\log \frac
2s\right)-\Phi\sp{-1}\left(\log \frac
2r\right)\right)\sp{m-1}}{r\Phi'\left(\Phi\sp{-1}\left(\log\frac
2r\right)\right)}
 \qquad \hbox{ for  $0 < s \leq r
< 1$}.
\end{align}
In particular, this yields
\begin{equation}\label{E:formula-for-operator-H}
H^m_{L_\Phi} f(t)
=\frac{1}{(m-1)!}\int_t^1 \frac{f(s)}{s\Phi'(\Phi^{-1}(\log \frac{2}{s}))}
\left(\Phi^{-1}\left(\log \frac{2}{t}\right)-\Phi^{-1}\left(\log
\frac{2}{s}\right)\right)^{m-1}\,ds\quad \hbox{for $t\in (0,1)$,}
\end{equation}
and $f\in \mathcal M_+(0,1)$.

A connection between the operators $P^m_\Phi$ and
$H^m_{L_\Phi}$ is described in the following proposition.

\begin{prop}\label{T:proposition2}
Suppose that $\Phi$ is as in~\eqref{E:measure}, $m\in \mathbb N$ and $f\in \mathcal M_+(0,1)$. Then
\begin{equation}\label{E:1-operators-P-and-H}
\frac{1}{2\sp m(m-1)!} P^m_\Phi f(t) \leq H^m_{L_\Phi} f(t)
\quad \textup{for $t\in (0,1)$.}
\end{equation}
Moreover, if $f$ is nonincreasing on $(0,1)$,
then
\begin{equation}\label{E:2}
H^m_{L_\Phi} f(t) \leq \frac{1}{(m-1)!} P^m_\Phi f(t) \quad
\textup{for $t\in (0,1)$.}
\end{equation}
\end{prop}

\begin{proof}
Let $f\in \mathcal M_+(0,1)$. Since the function $s
\mapsto \frac{1}{\Phi'(\Phi^{-1}(\log \frac{2}{s}))}$ is
nondecreasing on $(0,1)$, from the first inequality in~\eqref{E:e}
we obtain that
\begin{align*}
H^m_{L_\Phi} f(t) &
=\frac{1}{(m-1)!}\int_t^1 \frac{f(s)}{s\Phi'(\Phi^{-1}(\log \frac{2}{s}))} \left(\int_t^s \frac{\,dr}{r\Phi'(\Phi^{-1}(\log \frac{2}{r}))}\right)^{m-1}\,ds\\
&
\geq \frac{1}{(m-1)!}\frac{1}{\left(\Phi'(\Phi^{-1}(\log \frac{2}{t}))\right)^m} \int_t^1 \frac{f(s)}{s} \left(\int_t^s \frac{\,dr}{r}\right)^{m-1}\,ds\\
&
=\frac{1}{(m-1)!}\frac{1}{\left(\Phi'(\Phi^{-1}(\log \frac{2}{t}))\right)^m} \int_t^1 \frac{f(s)}{s} \left(\log \frac{s}{t} \right)^{m-1}\,ds\\
&
\geq \frac{1}{(m-1)!}\frac{1}{2^m}\left(\frac{\Phi^{-1}(\log \frac{2}{t})}{\log \frac{2}{t}}\right)^m \int_t^1 \frac{f(s)}{s} \left(\log \frac{s}{t} \right)^{m-1}\,ds\\
&
=\frac{1}{(m-1)!}\frac{1}{2^m} P^m_\Phi f(t) \quad \textup{for $t\in (0,1)$.}
\end{align*}

Now, assume that $f$ is nonincreasing on $(0,1)$.
In the special case when $f$ is a characteristic function of an open
interval, namely, $f=\chi_{(0,b)}$ for some $b\in (0,1]$, equation~\eqref{E:formula-for-operator-H} tells us that
$$
H^m_{L_\Phi}(\chi_{(0,b)})(t)=\frac{1}{m!}
\chi_{(0,b)}(t) \left(
\Phi^{-1}\left(\log \frac{2}{t}\right) - \Phi^{-1}\left(\log
\frac{2}{b}\right) \right)^m
$$
and
$$
P^m_\Phi (\chi_{(0,b)})(t)= \chi_{(0,b)}(t) \frac{1}{m} \left(\frac{\Phi^{-1}(\log \frac{2}{t})}{\log \frac{2}{t}}\right)^m \left(\log \frac{2}{t}
 - \log \frac{2}{b}\right)^m
$$
for $t\in (0,1)$. By the last inequality
in~\eqref{E:e},
$$
\left(\frac{ \Phi^{-1}\left(\log \frac{2}{t}\right) - \Phi^{-1}\left(\log \frac{2}{b}\right)}{\log \frac{2}{t} - \log \frac{2}{b}}\right)^m \leq \left(\frac{\Phi^{-1}(\log \frac{2}{t})}{\log \frac{2}{t}}\right)^m \quad \textup{for $t\in (0,b)$.}
$$
Hence,
\begin{equation}\label{sept05}
H^m_{L_\Phi}(\chi_{(0,b)}) \leq {\frac{1}{(m-1)!}}P^m_\Phi (\chi_{(0,b)}).
\end{equation}

Assume next that  $f$ is a  nonnegative
non-increasing simple function on $(0,1)$. Then there exist $k\in
\mathbb N$,  nonnegative  numbers $a_1, a_2, \dots, a_k \in \mathbb
R$ and $0<b_1 <b_2<\dots <b_k\leq 1$ such that $f=\sum_{i=1}^{k} a_i
\chi_{(0,b_i)}$ a.e.\ on $(0,1)$. Hence, owing to \eqref{sept05},
$$
H^m_{L_\Phi} f(t)= \sum_{i=1}^k a_i
H^m_{L_\Phi}(\chi_{(0,b_i)}) (t)\leq \frac{1}{(m-1)!}\sum_{i=1}^k a_i
P^m_\Phi(\chi_{(0,b_i)}) (t) = \frac{1}{(m-1)!}P^m_\Phi f(t) \quad \textup{for $t\in
(0,1)$.}
$$

Finally, if $f\in \mathcal M_+(0,1)$
 is nonincreasing on $(0,1)$, then
 there exists a sequence $f_k$ of nontrivial nonnegative nonincreasing simple functions on $(0,1)$ such that $f_n\uparrow f$. Clearly,
$$
H^m_{L_\Phi} f(t)= \lim_{n\to \infty} H^m_{L_\Phi}
f_n (t) \leq \frac{1}{(m-1)!}\lim_{n\to \infty} P^m_\Phi f_n (t) = \frac{1}{(m-1)!}P^m_\Phi f(t)
\quad \textup{for $t\in (0,1)$,}
$$
whence \eqref{E:2} follows.
\end{proof}

Proposition~\ref{T:proposition2} has an important consequence.

\begin{prop}\label{T:theorem}
Let $\Phi$ be as in~\eqref{E:measure}, let $m\in
\mathbb N$ and let $\|\cdot\|_{X(0,1)}$ and $\|\cdot\|_{Y(0,1)}$ be
rearrangement-invariant function norms. Then
$$
P^m_\Phi: X(0,1) \rightarrow Y(0,1) \quad
\textup{if and only if} \quad H^m_{L_\Phi}: X(0,1) \rightarrow
Y(0,1).
$$
\end{prop}

\begin{proof}
By~\eqref{E:1-operators-P-and-H}, the boundedness of the operator
$H^m_{L_\Phi}$ implies the boundedness of $P^m_\Phi$. Conversely, if
$P^m_\Phi$ is bounded from $X(0,1)$ into $Y(0,1)$ then, in
particular, there exists a constant $C$ such that
$$
\|P^m_\Phi f\|_{Y(0,1)} \leq C\|f\|_{X(0,1)}
$$
 for every
nonnegative non-increasing function $f\in X(0,1)$.
Combining this inequality with~\eqref{E:2}, we obtain that
$$
\|H^m_{L_\Phi} f\|_{Y(0,1)} \leq C \|f\|_{X(0,1)}
$$
for every nonnegative non-increasing $f\in
X(0,1)$. In view of Corollary~\ref{C:restriction-of-main-inequality-to-nonincreasing-functions}, this is equivalent to the
boundedness of $H^m_{L_\Phi}$ from $X(0,1)$ into $Y(0,1)$.
\end{proof}

\begin{proof}[{\bf Proof of Theorem \ref{T:prob-reduction}}]
Properties (ii) and (iii) are equivalent, by Proposition
\ref{poincsob}. Let us show that (i) and (ii) are equivalent as
well. First, assume that \textup{(i)} is
satisfied. Owing to Proposition~\ref{T:theorem}, there exists a
constant $C$ such that
$$
\left\|\int_t^1 \frac{f(s)}{L_\Phi(s)} \left(\int_t^s \frac{\,dr}{L_\Phi(r)}\right)^{m-1}\,ds\right\|_{Y(0,1)} \leq C\|f\|_{X(0,1)}
$$
for every nonnegative $f\in X(0,1)$. By
Lemma~\ref{delta2} \textup{(i)}, the function $L_\Phi$ is
non-decreasing on $[0,1]$. Furthermore,
condition~\eqref{E:lower-bound} is clearly satisfied with
$I=L_\Phi$. Thanks to these facts and to~\eqref{E:profile}, the
assumptions of Theorem \ref{T:reduction} are fulfilled with $(\Omega
, \nu) = (\rn, \mu _{\Phi ,n})$ and $I = L_\Phi$. Hence,
\textup{(ii)} follows.

It only remains to prove that (ii) implies (i). Assume that (ii)
holds, namely, there exists a constant $C$,
such that
\begin{align}\label{100000}
\|u\|_{Y(\rn,{\mu _{\Phi, n}})} & \leq C \Big( \|\nabla ^m
u\|_{X(\rn,{\mu _{\Phi, n}})} + \sum _{k=0}^{m-1} \| \nabla ^k
u\|_{L^1(\rn,{\mu _{\Phi , n}})}\Big)
\end{align}
for every  $u \in  V^m X(\rn,{\mu _{\Phi, n}})$.

Given any nonnegative function $f \in  X (0,1)$ such that $f(s) =0$
if $s \in (\frac 12 , 1)$, consider  the function $u : \rn \to \R$
defined as
\[
u(x) = \int _{H(x_1)}^{1} \frac 1{F_\Phi (r_1)} \int _{r_1}^{1}\frac
1{F_\Phi (r_2)} \dots \int _{r_{m-1}}^{1} \frac {f(r_m)}{F_\Phi
(r_m)} \,dr_m\,\,dr_{m-1} \dots \, dr_1 \qquad \hbox{for $x \in
\rn$,}
\]
where $H$ is given by \eqref{H}. Note that, since $H' (t) = - F_\Phi
(H (t))$,  then
\[
|\nabla ^k u (x)|= \frac {\partial ^k u}{\partial x_1^k}(x) = \int
_{H(x_1)}^{1} \frac 1{F_\Phi (r_{k+1})} \int _{r_{k+1}}^{1}\frac
1{F_\Phi (r_{k+2})} \dots \int _{r_{m-1}}^{1} \frac {f(r_m)}{F_\Phi
(r_m)} \,dr_m\,\,dr_{m-1} \dots \, dr_{k+1}
\]
for a.e. $x \in \rn$,  for  $k=1, \dots , m-1$, and
\[
|\nabla ^m u (x)|= \frac {\partial ^m u}{\partial x_1^m}(x) =  f(H
(x_1)) \qquad
 \hbox{for a.e. $x \in \rn$.}
\]
Thus, by \eqref{measphi},
\begin{equation}\label{100003k}
|\nabla ^k u|^*(s) = \int _{s}^{1} \frac 1{F_\Phi (r_{k+1})} \int
_{r_{k+1}}^{1}\frac 1{F_\Phi (r_{k+2})} \dots \int _{r_{m-1}}^{1}
\frac {f(r_m)}{F_\Phi (r_m)} \,dr_m\,\,dr_{m-1} \dots \, dr_{k+1}
\quad  \hbox{for $s \in (0,1)$,}
\end{equation}
for $k=0, \dots , m-1$, and
\begin{equation}\label{100005}
|\nabla ^m u|^*(s) = \bigg |\frac {\partial ^m u}{\partial
x_1^m}\bigg |^*(s) = f\sp*(s) \qquad
 \qquad \hbox{for $s \in (0,1)$.}
\end{equation}
By \eqref{100005}, $u \in  V^m X(\rn,{\mu_{\Phi,n}})$. \relax From
\eqref{100000}, \eqref{100003k} and~\eqref{100005} we thus deduce
that
\begin{align}\label{E:***}
&\bigg \|\int _{s}^{1}  \frac 1{F_\Phi (r_1)} \int _{r_1}^{1}\frac
1{F_\Phi (r_2)} \dots \int _{r_{m-1}}^{1} \frac {f(r_m)}{F_\Phi
(r_m)}
\,dr_m\,\,dr_{m-1} \dots \, dr_1 \bigg \|_{ Y(0,1)} \\
\nonumber & \leq C \|f \|_{ X(0,1)} + C \sum _{k=0}^{m-1}\int _0^1
\int _{s}^{1} \frac 1{F_\Phi (r_{k+1})} \int _{r_{k+1}}^{1}\frac
1{F_\Phi (r_{k+2})} \dots \int _{r_{m-1}}^{1} \frac {f(r_m)}{F_\Phi
(r_m)} \,dr_m\,\,dr_{m-1} \dots \, dr_{k+1}\,ds.
\end{align}
Owing to  Fubini's Theorem,  \eqref{E:profile} and  \eqref{dec104},
\begin{align}\label{A}
\int _{s}^{1} & \frac 1{F_\Phi(r_1)} \int _{r_1}^{1}\frac 1{F_\Phi
(r_2)} \dots \int _{r_{m-1}}^{1} \frac {f(r_m)}{F_\Phi(r_m)}
\,dr_m\,\,dr_{m-1} \dots \, dr_1
  \approx
\int _s^1  \frac{f(r)}{ {F_\Phi}(r)}\left(\int_s\sp r\frac{dt}{
{F_\Phi}(t)}\right)\sp{m-1}
 \, dr
 \\ \nonumber &
\approx
 \int _s^1  \frac{f(r)}{ {L_\Phi}(r)}\left(\int_s\sp
r\frac{dt}{{L_\Phi}(t)}\right)\sp{m-1}
 \, dr
 =
 \int _s^1   f(r)  \frac{\left(\Phi\sp{-1}\left(\log \frac
2s\right)-\Phi\sp{-1} \left(\log \frac
2r\right)\right)\sp{m-1}}{r\Phi'\left(\Phi\sp{-1}\left(\log\frac
2r\right)\right)} \, dr  \,\,\, \hbox{for $s \in (0,1)$.}
\end{align}
  Note that the second equivalence makes use of  the fact that $f$ vanishes in $(\frac 12 , 1)$. On the other
hand, by \eqref{A} (with $m$ replaced with $m-k$),
\eqref{l1linf}, and~\eqref{E:bdj} (with $I$ replaced with
$L_{\Phi}$),
\begin{align}\label{B}
\int _0^1  & \int _{s}^{1} \frac 1{F_\Phi(r_{k+1})}\int
_{r_{k+1}}^{1}\frac 1{F_\Phi(r_{k+2})} \dots \int _{r_{m-1}}^{1}
\frac
{f(r_m)}{F_\Phi(r_m)} \,dr_m\,\,dr_{m-1} \dots \, dr_{k+1}\,ds\\
\nonumber & \approx \int _0^1 \int _s^1  \frac{f(r)}{
{L_\Phi}(r)}\left(\int_s\sp r\frac{dt}{
{L_\Phi}(t)}\right)\sp{m-k-1}
 \, dr \, ds \approx\|H_{L_{\Phi}}\sp{m-k}f\|_{L\sp1(0,1)} \nonumber \\
& \leq
 C \|H_{L_{\Phi}}\sp{m-k}f\|_{(L\sp1)_{m-k,L_{\Phi}}(0,1)}\leq C'\|f\|_{L\sp1(0,1)} \leq C''\|f\|_{X(0,1)}\nonumber
\end{align}
for some constants $C$, $C'$ and $C''$. \relax From inequalities
\eqref{E:***} -- \eqref{B}, we deduce that
 there exists a constant
$C$ such that
\begin{align*}
\left\|\int _s^1 f(r) \frac{\left(\Phi\sp{-1}\left(\log \frac
2s\right)-\Phi\sp{-1}\left(\log \frac
2r\right)\right)\sp{m-1}}{r\Phi'\left(\Phi\sp{-1}\left(\log\frac
2r\right)\right)}  \, dr \right\|_{\bY(0,1)}
   & \leq C \left\|f\right\|_{\bX(0,1)}
  \end{align*}
for every nonnegative function $f \in
 X (0,1)$ such that $f(s) =0$ if $s \in (\frac 12 , 1)$.
By Proposition~\ref{T:proposition2}, for each such function $f$ we also have
\begin{equation}\label{E:half}
\left\|\left(\frac{\Phi^{-1}(\log \frac{2}{t})}{\log \frac{2}{t}} \right)^{m} \int_t^1 \frac{f(s)}{s} \left(\log \frac{s}{t}\right)^{m-1}\,ds\right\|_{Y(0,1)} \leq 2^mC\|f\|_{X(0,1)}.
\end{equation}

Finally, assume that $f$ is any  nonnegative function from $ X
(0,1)$ (which need not vanish in $(\frac 12 , 1)$). Then, by the
boundedness of the dilation operator on $Y(0,1)$, there exists a
constant $C$ such that
\begin{align}\label{E:nov11}
&
\left\|\left(\frac{\Phi^{-1}(\log \frac{2}{t})}{\log \frac{2}{t}} \right)^{m} \int_t^1 \frac{f(s)}{s} \left(\log \frac{s}{t}\right)^{m-1}\,ds\right\|_{Y(0,1)}\\
&
\leq C\left\|\chi_{(0,\frac{1}{2})}(t)\left(\frac{\Phi^{-1}(\log \frac{1}{t})}{\log \frac{1}{t}} \right)^{m} \int_{2t}^1 \frac{f(s)}{s} \left(\log \frac{s}{2t}\right)^{m-1}\,ds\right\|_{Y(0,1)}.\nonumber
\end{align}
Furthermore, since
$$
\frac{\Phi^{-1}(\log \frac{1}{t})}{\log
\frac{1}{t}} \leq \frac{\Phi^{-1}(\log \frac{2}{t})}{\log
\frac{1}{t}} \leq \frac{2\Phi^{-1}(\log \frac{2}{t})}{\log
\frac{2}{t}} \quad \textup{for $t\in (0,\tfrac{1}{2})$,}
$$
from inequality~\eqref{E:half} with $f$ replaced
with $\chi_{(0,\frac{1}{2})}(t)f(2t)$, and the boundedness of the
dilation operator, we obtain
\begin{align}\label{E:nov12}
&
\left\|\chi_{(0,\frac{1}{2})}(t)\left(\frac{\Phi^{-1}(\log \frac{1}{t})}{\log \frac{1}{t}} \right)^{m} \int_{2t}^1 \frac{f(s)}{s} \left(\log \frac{s}{2t}\right)^{m-1}\,ds\right\|_{Y(0,1)}\\ \nonumber
&
\leq 2^m\left\|\chi_{(0,\frac{1}{2})}(t)\left(\frac{\Phi^{-1}(\log \frac{2}{t})}{\log \frac{2}{t}} \right)^{m} \int_{t}^\frac{1}{2} \frac{f(2s)}{s} \left(\log \frac{s}{t}\right)^{m-1}\,ds\right\|_{Y(0,1)}\\ \nonumber
&
\leq C'\|\chi_{(0,\frac{1}{2})}(t)f(2t)\|_{X(0,1)} \leq C'' \|f\|_{X(0,1)}
\end{align}
for some constants $C'$ and $C''$ independent of
$f$. Coupling~\eqref{E:nov11} with~\eqref{E:nov12} yields
\eqref{E:hardy}.
\end{proof}

\begin{proof}[{\bf Proof of Theorem \ref{T:optimal_range}}]
Set $J(s)=s$ for $s\in [0,1]$.
Then condition~\eqref{E:lower-bound} is obviously
fulfilled with $I=J$. The norm
$\|\cdot\|_{\widetilde X_{m,J}(0,1)}$ is thus
well defined and, moreover, $\|\cdot\|_{\widetilde X_m(0,1)}=\|\cdot\|_{X_{m,J}(0,1)}$. Therefore,
Proposition~\ref{P:norm} tells us that
$\|\cdot\|_{\widetilde X_m(0,1)}$ is a
rearrangement-invariant function norm. We shall now verify that
$\|\cdot\|_{X_{m,\Phi}(0,1)}$ is a rearrangement-invariant function
norm as well. The first two properties in \textup{(P1)} and
properties \textup{(P2)} and \textup{(P3)} are straightforward
consequences of the corresponding properties for
$\|\cdot\|_{\widetilde X_m(0,1)}$. To prove the
triangle inequality, fix $f$, $g\in \mathcal M_+(0,1)$.
By~\eqref{subadd}, $\int_0^s (f+g)^*(r)\,dr \leq \int_0^s
(f^*(r)+g^*(r))\,dr$ for $s\in (0,1)$. We observe that for each
$t\in (0,1)$, the function $s\mapsto \chi_{(0,t)}(s)\left(\frac{\log
\frac 2s}{\Phi^{-1}(\log \frac 2s)} \right)^m$ is nonnegative and
non-increasing on $(0,1)$. The Hardy's lemma therefore yields that
$$
\int_0^t \left(\frac{\log \frac 2s}{\Phi^{-1}(\log \frac 2s)} \right)^m (f+g)^*(s)\,ds \leq \int_0^t \left(\left(\frac{\log \frac 2s}{\Phi^{-1}(\log \frac 2s)} \right)^m f^*(s) + \left(\frac{\log \frac 2s}{\Phi^{-1}(\log \frac 2s)} \right)^m g^*(s)\right)\,ds
$$
for $t\in (0,1)$. The triangle inequality now
follows using the Hardy - Littlewood - P\'olya principle and the
triangle inequality for $\|\cdot\|_{\widetilde
X_m(0,1)}$.

One has that
\[
\exp L^{\frac{1}{m}}(0,1) =(L^\infty)_m(0,1)
\rightarrow \widetilde X_m(0,1),
\]
where the equality is a consequence of
Theorem~\ref{T:eucl_lorentz-alpha=1}. Thus,
 there exists a constant $C$ such
that
\begin{align*}
\|1\|_{X_{m,\Phi}(0,1)}  & =
\left\|\left(\frac{\log \frac 2s}{\Phi^{-1}(\log \frac 2s)}
\right)^m\right\|_{\widetilde X_m(0,1)}
\leq C \left\|\left(\frac{\log \frac 2s}{\Phi^{-1}(\log \frac 2s)} \right)^m\right\|_{\exp L^{\frac 1m}(0,1)}\\
&
\approx C \left\|\frac{1}{(\Phi^{-1}(\log \frac 2s))^m}\right\|_{L^{\infty}(0,1)}
=\frac{C}{(\Phi^{-1}(\log 2))^m} <\infty.
\end{align*}
This proves \textup{(P4)}.

Finally, by property \textup{(P5)} for
$\|\cdot\|_{\widetilde X_m(0,1)}$, there exists a
positive constant $C$ such that for all $f\in \mathcal M_+(0,1)$,
$$
\|f\|_{X_{m,\Phi}(0,1)} \geq \left(\frac{\log
2}{\Phi^{-1}(\log 2)} \right)^m
\|f^*\|_{\widetilde X_m(0,1)} \geq
\left(\frac{C\log 2}{\Phi^{-1}(\log 2)} \right)^m \int_0^1
f^*(s)\,ds.
$$
Therefore, $\|\cdot\|_{X_{m,\Phi}(0,1)}$ satisfies
\textup{(P5)}. Since the property \textup{(P6)} holds trivially,
$\|\cdot\|_{X_{m,\Phi}(0,1)}$ is actually a rearrangement-invariant
norm.

It follows from the proof of Theorem~\ref{T:prob-reduction} that the assumptions of
 Theorem~\ref{T:optran} are fulfilled with $(\Omega,\nu)=(\mathbb R^n, \mu_{\Phi,n})$ and $I=L_\Phi$.
  Therefore, $\|\cdot \|_{X_{m,L_\Phi}(0,1)}$ is the   optimal rearrangement-invariant target function norm for $\|\cdot \|_{X(0,1)}$
   in  the Sobolev embedding~\eqref{E:embedding}. Thus, the proof will be complete if we show
   that
   $X_{m,\Phi}(0,1)=X_{m,L_\Phi}(0,1)$.
We have that
\begin{align}\label{E:1}
\|f\|_{X'_{m,L_\Phi}(0,1)} &
=(m-1)!\|R^m_{L_\Phi} f^*\|_{X'(0,1)} \approx \|R^m_{L_\Phi}
f^*\|_{X'_d(0,1)} \quad \textup{(by Theorem~\ref{T:lenka-main})} \\
\nonumber
&  =\sup_{\|g\|_{X(0,1)}\leq 1} \int_0^1
g^*(t) R^m_{L_\Phi}f^*(t)\,dt\\
\nonumber &
=\sup_{\|g\|_{X(0,1)}\leq 1} \int_0^1 f^*(t)
H^m_{L_\Phi}g^*(t)\,dt\\
\nonumber & \approx
\sup_{\|g\|_{X(0,1)}\leq 1} \int_0^1 f^*(t) P^m_{\Phi}g^*(t)\,dt
\quad \textup{(by Proposition~\ref{T:proposition2})}\\
\nonumber &
\approx\sup_{\|g\|_{X(0,1)}\leq 1} \int_0^1 f^*(t)
\left(\frac{\Phi^{-1}(\log \frac 2t)}{\log \frac 2t} \right)^m
H^m_{J}g^*(t)\,dt \\
\nonumber &
=\sup_{\|g\|_{X(0,1)}\leq 1} \int_0^1 g^*(t) R^m_J\left(f^*(s)
\left(\frac{\Phi^{-1}(\log \frac 2s)}{\log \frac 2s}
\right)^m\right)(t) \,dt\\
\nonumber &
=\left\|R^m_J\left(f^*(s) \left(\frac{\Phi^{-1}(\log \frac 2s)}{\log
\frac 2s} \right)^m\right)\right\|_{X'_d(0,1)} \quad \textup{for
$f\in L^1(0,1)$,}
\end{align}
up to multiplicative constants depending on $m$.

We now claim that, given $f\in L^1(0,1)$, there
exists a non-decreasing function $I$ on $[0,1]$
fulfilling~\eqref{E:lower-bound} and a function $h\in \mathcal
M_+(0,1)$ such that
\begin{equation}\label{E:21}
f^*(s) \left(\frac{\Phi^{-1}(\log \frac 2s)}{\log \frac 2s} \right)^m \approx R_I h^*(s) \quad \textup{for $s\in (0,1)$},
\end{equation}
up to multiplicative constants depending on $m$.
Indeed, let $s_0\in (0,1)$  be chosen in such a way that the
function $s\mapsto s\big(\log \frac 2s\big)^{m+1}$ is non-decreasing
on $(0,s_0)$. Then we set
$$
I(s)=\frac{1}{f^*(s)} \quad \textup{for $s\in (0,1]$} \quad \textup{and} \quad I(0)=0,
$$
and
$$
h(s)=
\begin{cases}
\frac{\left(\Phi^{-1}(\log \frac 2s)\right)^{m-1} \left(\Phi^{-1}(\log \frac 2s) - \frac{\log \frac 2s}{\Phi'(\Phi^{-1}(\log \frac 2s))}\right)}{s(\log \frac 2s)^{m+1}}, & s\in (0,s_0]\\
\frac{\left(\Phi^{-1}(\log \frac 2s)\right)^{m-1} \left(\Phi^{-1}(\log \frac 2s) - \frac{\log \frac 2s}{\Phi'(\Phi^{-1}(\log \frac 2s))}\right)}{s_0(\log \frac 2{s_0})^{m+1}}, & s\in (s_0,1).
\end{cases}
$$
It follows from~\eqref{E:e} that the function $h$ is non-negative on $(0,1)$.
To verify~\eqref{E:21} we first show that $h$ is non-increasing on $(0,1)$. The function $\Phi^{-1}$ is clearly non-decreasing on $(0,\infty)$.
 Furthermore, we deduce from the convexity of $\Phi$ that the function  $s \mapsto \Phi^{-1}(s) - \frac{s}{\Phi'(\Phi^{-1}(s))}$
  is non-decreasing on $(0,\infty)$. Altogether, this ensures that
$$
s\mapsto \left(\Phi^{-1}\left(\log \frac 2s\right)\right)^{m-1} \left(\Phi^{-1}\left(\log \frac 2s \right) - \frac{\log \frac 2s}{\Phi'(\Phi^{-1}(\log \frac 2s))}\right)
$$
is non-increasing on $(0,1)$. By the definition of $s_0$, the function
$$
s\mapsto
\begin{cases}
s(\log \frac 2s)^{m+1},& s\in (0,s_0]\\
s_0(\log \frac 2{s_0})^{m+1}, & s\in (s_0,1)
\end{cases}
$$
is non-decreasing (and continuous) on $(0,1)$, and
therefore, in particular, $h=h^*$.

Consequently, we have
\begin{align*}
f^*(s) \left(\frac{\Phi^{-1}(\log \frac 2s)}{\log \frac 2s} \right)^m
&
= \frac{m}{I(s)} \int_0^s \frac{\left(\Phi^{-1}(\log \frac 2r)\right)^{m-1} \left(\Phi^{-1}(\log \frac 2r) - \frac{\log \frac 2r}{\Phi'(\Phi^{-1}(\log \frac 2r))}\right)}{r(\log \frac 2r)^{m+1}}\,dr \\
&
\approx \frac{1}{I(s)} \int_0^s h^*(r)\,dr= R_I h^*(s) \quad \textup{for $s\in (0,1)$},
\end{align*}
up to multiplicative constants depending on $m$.
This proves~\eqref{E:21}. Furthermore, it can be easily verified
that the function $I$ fulfils also the remaining required
properties.

Coupling~\eqref{E:1} with~\eqref{E:21}  entails
that, for the fixed function $f$,
\begin{equation}\label{E:31}
\|f\|_{X'_{m,L_\Phi}(0,1)} \approx \|R^m_J R_I h^*\|_{X_d'(0,1)},
\end{equation}
up to multiplicative constants depending on $m$.

Now, the same proof as that of Theorem~\ref{T:lenka-main} yields that
\begin{equation}\label{E:41}
\|R^m_J R_I h^*\|_{X_d'(0,1)} \approx \|R^m_J (R_I h^*)^*\|_{X'(0,1)},
\end{equation}
up to multiplicative constants still depending only on $m$.

On combining \eqref{E:31}, \eqref{E:41} and
\eqref{E:21}, we obtain that for every $f\in L^1(0,1)$,
\begin{align}\label{E:51}
\|f\|_{X'_{m,L_\Phi}(0,1)}
&
\approx \left\|R^m_J \left(f^*(\cdot) \left(\frac{\Phi^{-1}(\log \frac 2{(\cdot)})}{\log \frac 2{(\cdot)}} \right)^m\right)^*(t)\right\|_{X'(0,1)}\\
&  \approx \left\|f^*(t)
\left(\frac{\Phi^{-1}(\log \frac 2t)}{\log \frac 2t}
\right)^m\right\|_{\widetilde X_m'(0,1)}\nonumber
\end{align}
up to mulplicative constants depending on $m$.
Consequently, by~\eqref{E:51},
 we have
that, for every $g\in \mathcal M_+(0,1)$,
\begin{align*}
\|g\|_{X_{m,L_\Phi}(0,1)}
&
=\sup \left\{\int_0^1 f^*(s) g^*(s)\,ds:  \|f\|_{X_{m,L_\Phi}'(0,1)}\leq 1 \right\}\\
&
\approx \sup \left\{\int_0^1 f^*(s) g^*(s)\,ds:  \left\|f^*(t) \left(\frac{\Phi^{-1}(\log \frac 2t)}{\log \frac 2t} \right)^m\right\|_{\widetilde X_m'(0,1)}\leq 1 \right\}\\
& \leq \left\|g^*(t) \left(\frac{\log \frac
2t}{\Phi^{-1}(\log \frac 2t)}
\right)^m\right\|_{\widetilde X_m(0,1)},
\end{align*}
up to mulplicative constants depending on $m$.

Conversely,
\begin{align}
&
\left\|g^*(t) \left(\frac{\log \frac 2t}{\Phi^{-1}(\log \frac 2t)} \right)^m\right\|_{\widetilde X_m(0,1)}\label{sept07}\\
&
=\sup \left\{\int_0^1 g^*(t) \left(\frac{\log \frac 2t}{\Phi^{-1}(\log \frac 2t)} \right)^m f^*(t)\,dt:   \|f\|_{\widetilde X_m'(0,1)}\leq 1\right\}\nonumber\\
&
\approx \sup \left\{\int_0^1 g^*(t) \left(\frac{\log \frac 2t}{\Phi^{-1}(\log \frac 2t)} \right)^m f^*(t)\,dt:
  \left\|f^*(t) \left(\frac{\log \frac 2t}{\Phi^{-1}(\log \frac 2t)} \right)^m\right\|_{X_{m,L_\Phi}'(0,1)}\leq 1\right\}\nonumber\\
& \leq \|g\|_{X_{m,L_\Phi}(0,1)},\nonumber
\end{align}
up to multiplicative constants depending on $m$.
Note that the equivalence in \eqref{sept07} holds by
\eqref{E:51} and the fact that the function $t
\mapsto \Big(\frac{\log \frac 2t}{\Phi^{-1}(\log \frac 2t)} \Big)^m$
is non-increasing. Hence,
$X_{m,\Phi}(0,1)=X_{m,L_\Phi}(0,1)$. The proof is complete.
\end{proof}

\begin{proof}[{\bf Proof of Proposition~\ref{P:prop3}}]
Since the $m$-th  iteration  of the double-star
operator $g\mapsto g\sp{**}$ associates a~ function $g$ with $\frac
1s\int_0\sp s (\log\frac sr)\sp{m-1}g\sp*(r)\,dr$ for $s\in (0,1)$,
we obtain from the boundedness of the double-star operator on
$X'(0,1)$ that
\[
\|g\|_{\widetilde X'_m(0,1)}\approx \|g\|_{X'(0,1)}.
\]
Thus, $\widetilde X_m(0,1)=X(0,1)$. Consequently, the assertion follows from~\eqref{E:optimal_range}.
\end{proof}

\begin{proof}[{\bf Proof of Theorem \ref{T:reiteration}}]
This is a consequence of Theorem \ref{T:reit} and of the fact that $X_{m,\Phi}(0,1)=X_{m,L_\Phi}(0,1)$.
\end{proof}

\begin{proof}[{\bf Proof of Theorem~\ref{EX:expo_lz}}]
Denote $X(0,1)=L\sp{p,q;\alpha}(0,1)$. We claim that
\[
\widetilde X_m(0,1)=
\begin{cases}
L\sp{p,q;\alpha}(0,1) &\textup{if}\ p<\infty,\\
L\sp{\infty,q;\alpha-m}(0,1) &\textup{if}\ p=\infty.
\end{cases}
\]
Indeed, let $p<\infty$ and set $\Phi(t)=t$, $t\in[0,\infty)$. Then, by Remark~\ref{R:note},
\[
\widetilde X_m(0,1)=X_{m,\Phi}(0,1).
\]
By~\eqref{E:lz_assoc} and~\eqref{E:glz-identity}, the operator $f\mapsto f\sp{**}$ is bounded on $X'(0,1)$. Therefore, by Proposition~\ref{P:prop3},
\[
X_{m,\Phi}(0,1)=X(0,1)=L\sp{p,q;\alpha}(0,1).
\]
Now, let $p=\infty$, and set $I(s)=s$, $s\in[0,1]$. Then $R_If\sp*=f\sp{**}$, whence,  by Theorem~\ref{T:lenka-remark-3},
\[
\|f\|_{(\widetilde X_1)'(0,1)} = \|f\sp{**}\|_{X'(0,1)} \approx
\|t\sp{ 1-\frac1{q'}}(\log\tfrac
2t)\sp{-\alpha}f\sp{**}(t)\|_{L\sp{q'}(0,1)}
=\|f\|_{L\sp{(1,q';-\alpha)}(0,1)}.
\]
Owing to~\eqref{E:lz_assoc} and~\eqref{E:glz-identity},
\[
(L\sp{(1,q';-\alpha)})'(0,1)=L\sp{\infty,q;\alpha-1}(0,1).
\]
Thus,
\[
\widetilde X_1(0,1)=L\sp{\infty,q;\alpha-1}(0,1).
\]
By making use of Theorem~\ref{T:reiteration} combined with Remark~\ref{R:note}, we obtain that
\[
\widetilde X_m(0,1)=L\sp{\infty,q;\alpha-m}(0,1).
\]
The conclusion is now a consequence of Theorem~\ref{optimalboltz}.
\end{proof}

\end{document}